\documentclass[11pt]{amsart}

\usepackage{mathpazo}
\usepackage{graphicx}
\usepackage[margin=1in]{geometry}
\usepackage{hyperref}
\usepackage{color}
\usepackage{mathdots}
\usepackage{tikz}
\definecolor{darkred}{rgb}{1,0,0} 
\definecolor{darkgreen}{rgb}{0,1,0}
\definecolor{darkblue}{rgb}{0,0,1}

\hypersetup{colorlinks,
linkcolor=darkblue,
filecolor=darkblue,
urlcolor=darkred,
citecolor=darkred}

\addtocontents{toc}{\setcounter{tocdepth}{1}} 

\usepackage[all,cmtip]{xy}

\newcommand{\quat}{\mathbb{H}}
\newcommand{\oct}{\mathbb O}

\newcommand{\C}{\mathbb C}
\newcommand{\R}{\mathbb R}
\newcommand{\Z}{\mathbb Z}

\renewcommand{\H}{\mathbf{H}^{4,2}}
\renewcommand{\SS}{\mathbf{S}^{3,3}}
\renewcommand{\P}{\mathbf P}

\newcommand{\q}{\mathbf q}
\newcommand{\g}{\mathbf g}
\newcommand{\om}{\boldsymbol{\omega}}
\renewcommand{\j}{\mathsf j}
\newcommand{\J}{\mathbf J}
\renewcommand{\i}{\mathrm{i}}

\newcommand{\Id}{\mathrm{Id}}

\newcommand{\Stab}{\mathrm{Stab}}

\newcommand{\G}{\mathsf{G}_2'}
\newcommand{\SU}{\mathsf{SU}}

\newcommand{\sG}{\mathsf{G}}
\newcommand{\sK}{\mathsf{K}}

\newcommand{\sO}{\mathsf{O}}
\newcommand{\sS}{\mathsf{S}}
\newcommand{\sT}{\mathsf{T}}

\newcommand{\sSO}{\mathsf{SO}}
\newcommand{\sSL}{\mathsf{SL}}

\newcommand{\sGL}{\mathsf{GL}}
\newcommand{\sPSL}{\mathsf{PSL}}
\newcommand{\sSU}{\mathsf{SU}}
\newcommand{\sU}{\mathsf{U}}

\newcommand{\K}{\mathsf{K}}
\newcommand{\SO}{\mathsf{SO}}

\newcommand{\End}{\mathsf{End}}
\newcommand{\Ker}{\mathsf{Ker}}
\newcommand{\SL}{\mathsf{SL}}

\newcommand{\Sym}{\mathbf{Sym}}
\newcommand{\sSym}{\mathsf{Sym}}
\newcommand{\Pic}{\mathsf{Pic}}
\newcommand{\Cy}{\mathbf{X}}

\newcommand{\T}{\mathrm T}
\newcommand{\N}{\mathrm N}
\newcommand{\B}{\mathrm B}

\newcommand{\DD}{\mathcal D}

\newcommand{\JJ}{\mathcal J}

\newcommand{\OO}{\mathcal O}

\newcommand{\HH}{\mathcal H}
\newcommand{\TT}{\mathcal T}
\newcommand{\F}{\mathcal F}

\newcommand{\cB}{\mathcal B}
\newcommand{\cE}{\mathcal E}
\newcommand{\cF}{\mathcal F}
\newcommand{\cH}{\mathcal H}
\newcommand{\cK}{\mathcal K}
\newcommand{\cL}{\mathcal L}
\newcommand{\cM}{\mathcal M}
\newcommand{\cN}{\mathcal N}
\newcommand{\cO}{\mathcal O}
\newcommand{\cP}{\mathcal P}
\newcommand{\cT}{\mathcal T}
\newcommand{\cU}{\mathcal U}
\newcommand{\cV}{\mathcal V}
\newcommand{\cW}{\mathcal W}
\newcommand{\cX}{\mathcal X}
\newcommand{\cI}{\mathcal I}
\newcommand{\cY}{\mathcal Y}

\newcommand{\fg}{\mathfrak g}
\newcommand{\fp}{\mathfrak p}
\newcommand{\fk}{\mathfrak k}

\newcommand{\ft}{\mathfrak t}

\newcommand{\fso}{\mathfrak{so}}

\renewcommand{\span}{\operatorname{span}}
\newcommand{\Hom}{\operatorname{Hom}}
\newcommand{\tr}{\operatorname{tr}}
\newcommand{\vol}{\operatorname{vol}}
\newcommand{\II}{\mathrm{II}}
\newcommand{\III}{\mathrm{III}}

\newcommand{\diag}{\mathrm{diag}}
\newcommand{\Hit}{\mathrm{Hit}}
\newcommand{\Aut}{\mathrm{Aut}}

\renewcommand{\S}{\Sigma}
\newcommand{\bbullet}{\,\begin{picture}(-1,1)(-1,0)\circle*{2}\end{picture}\ }
\newcommand{\GC}{\mathsf{G}_2^{\mathbb C}}
\newcommand{\Diff}{\mathsf{Diff}}
\newcommand{\fG}{\mathfrak G}
\renewcommand{\deg}{\text{deg}}
\newcommand{\Mod}{\mathsf{Mod}}

\newtheorem{proposition}{Proposition}[section]
\newtheorem{definition}[proposition]{Definition}
\newtheorem{remark}[proposition]{Remark}
\newtheorem{lemma}[proposition]{Lemma}
\newtheorem{theorem}[proposition]{Theorem}
\newtheorem{corollary}[proposition]{Corollary}

\newtheorem{MainTheorem}{Theorem}

\theoremstyle{definition}
\newtheorem{example}{Example}

\title{Holomorphic curves in the $6$-pseudosphere and cyclic surfaces}
\author[B. Collier]{Brian Collier}
\address{University of California Riverside, USA}
\email{brian.collier@ucr.edu}
\author[J. Toulisse]{J\'{e}r\'{e}my Toulisse}
\address{Universit\'e C\^ote d'Azur, CNRS,  LJAD,  France}
\email{jeremy.toulisse@univ-cotedazur.fr}

\begin{document}

\begin{abstract}
	The space $\mathbf{H}^{4,2}$ of vectors of norm $-1$ in $\R^{4,3}$ has a natural  pseudo-Riemannian metric and a compatible almost complex structure. The group of automorphisms of both of these structures is the split real form $\sG_2'$. In this paper we consider a class of holomorphic curves in $\H$ which we call alternating. 
	We show that such curves admit a so called Frenet framing. Using this framing, we show that the space of alternating holomorphic curves which are equivariant with respect to a surface group are naturally parameterized by certain $\sG_2'$-Higgs bundles. This leads to a holomorphic description of the moduli space as a fibration over Teichm\"uller space with a holomorphic action of the mapping class group. Using a generalization of Labourie's cyclic surfaces, we then show that equivariant alternating holomorphic curves are infinitesimally rigid. 
\end{abstract}
\maketitle
\tableofcontents
\section{Introduction}

The pseudohyperbolic space $\mathbf H^{4,2}$ is the set of vectors of norm $-1$ in a real vector space equipped with a nondegenerate quadratic form of signature $(4,3)$. There is a natural signature $(4,2)$ pseudo-Riemannian metric on $\H$. Multiplication in the split octonion algebra defines a nonintegrable almost complex structure on $\H$ which is compatible with the pseudo-Riemannian metric.  
This is the noncompact analogue of the famous almost complex structure on the $6$-sphere. 
The group which preserves this structure is the split real Lie group $\sG_2'$ of type $G_2.$
In this paper, we consider a class of $\J$-holomorphic curves in $\H$ which we call \emph{alternating}. Similar to parameterized curves in $\R^3$, holomorphic curves in $\H$ generically admit a Frenet framing. 
The alternating condition concerns the signature of the subbundles in the Frenet framing, and can be equivalently described using cyclic Higgs bundles. We focus especially on alternating holomorphic curves which are preserved by a surface group, and describe the resulting moduli space.

Let $\Sigma$ be a closed oriented surface of genus $g\geq 2$. For a real semisimple Lie group $\sG$, the character variety $\cX(\Sigma,\sG)$ is the space of $\sG$-conjugacy classes of reductive group homomorphisms $\pi_1(\Sigma)\to\sG$ from the fundamental group of $\Sigma$ into $\sG$. 
For some Lie groups, the character variety has distinguished connected components which generalize many features of the Teichm\"uller space of $\Sigma$. Such components will be called higher rank Teichm\"uller spaces. These spaces have been studied in detail by a variety of authors from various perspectives, see for example  \cite{BeatriceBourbaki,AnnaICM,KasselICM,OscarSurvey}.

When the rank of $\sG$ is two, the relevant Lie groups are locally isomorphic to $\sSL(3,\R)$,  $\sSO(2,n)$ and $\sG_2'$; the relevant spaces are known as the Hitchin components for $\sSL(3,\R)$, $\sSO(2,3)$ and $\sG_2'$, and the space of maximal representations for $\sSO(2,n)$. 
Each of these spaces admits a mapping class group invariant complex analytic parameterization as a fibration over the Teichm\"uller space of $\Sigma$. 
This was described independently by Loftin \cite{flatmetriccubicdiff} and Labourie \cite{CubicDifferentialsRP2} for $\sSL(3,\R)$ and Labourie \cite{cyclicSurfacesRank2} for Hitchin components; for maximal representations it was described by Alessandrini and the first author \cite{SO23LabourieConj} and Tholozan and the authors \cite{CTT}.

The complex analytic description has two inputs,  a parameterization of the spaces using Higgs bundles on a fixed Riemann surface and the  uniqueness of a $\pi_1(\Sigma)$-invariant minimal surface in the Riemannian symmetric space of the group. 
A central tool in the rank 2 analysis is the restriction to a set of Higgs bundles which have a cyclic group symmetry. 
When the rank of the group is at least 3, the relevant Higgs bundles do not necessarily have a cyclic symmetry. In fact, Markovic, Sagman and Smillie recently showed that, in rank at least 3, there are representations in higher rank Teichm\"uller spaces which have multiple $\pi_1(\Sigma)$-invariant minimal surface in the Riemannian symmetric space \cite{Markovic,MarkovicSagmanSmilie,SagmanSmilie}. 
Among other thing, this breakthrough suggests that representations arising from cyclic Higgs bundles deserve special attention. 

For $\sSL(3,\R)$ Hitchin representations, the uniqueness of the minimal surface was deduced from the uniqueness of an invariant affine sphere. Similarly, for $\sSO(2,n)$ maximal representations, the uniqueness of the minimal surface follows from the uniqueness of an invariant maximal spacelike surface in the pseudohyperbolic space $\mathbf{H}^{2,n-1}.$ In both cases, these surfaces characterize the appropriate higher rank Teichm\"uller space.  That is, they only exist for representations in the $\sSL(3,\R)$-Hitchin component or for maximal $\sSO(2,n)$-representations, respectively.
In \cite{g2geometry}, Baraglia pioneered the use of cyclic Higgs bundles to study Hitchin representations in rank 2. In particular, he recovered the affine sphere description for $\sSL(3,\R)$ and proved that Hitchin representations into $\sG_2'$ define equivariant holomorphic curves $f:\widetilde X\to \H.$ 

The equivariant holomorphic curves constructed by Baraglia satisfy the alternating condition defined in this paper. However, we prove that preserving an alternating holomorphic curve in $\H$ does not characterize the $\sG_2'$ Hitchin component. 
As a result, the theory of alternating holomorphic curves identifies representations $\pi_1(\Sigma)\to \sG_2'$ which are not Hitchin representations but share some features with Hitchin representations. 
This is similar to what happens for $\sSO(2,3)$. Namely, $\sSO(2,3)$ Hitchin representations are maximal, and so preserve a maximal spacelike surface in $\mathbf{H}^{2,2}$, but this does not characterize $\sSO(2,3)$ Hitchin representations since there are non-Hitchin maximal $\sSO(2,3)$ representations. 
Unlike the $\sSO(2,3)$ situation, the only higher rank Teichm\"uller space for the group $\sG_2'$ is the Hitchin component.
In particular, little is known about the non-Hitchin $\sG_2'$ representations which preserve an alternating holomorphic curve. It would be interesting to determine whether these $\sG_2'$ representations are Anosov, and to study the uniqueness properties of the alternating holomorphic curves.

\subsection{Alternating holomorphic curves}
Let $S$ be an oriented surface, and $(\H,\J)$ be the almost complex manifold described above. An immersion $f:S\to\H$ is called a holomorphic curve if the image of the tangent bundle is $\J$-invariant. This induces a Riemann surface structure on $S.$ 
As mentioned above,  holomorphic curves in $\H$ generically admit a Frenet framing. We say that $f:S\to \H$ is \emph{alternating} if the tangent bundle of $S$ is positive definite and the image of the second fundamental form is negative definite. It turns out that when the second fundamental form is nonzero, its image defines a subbundle of rank two, see \textsection \ref{sec alt hol curves} for details. 
Alternating holomorphic curves are in particular maximal spacelike surfaces.

By construction, the Frenet framing defines a lift to a homogeneous space $\Cy$ which we call the cyclic space. Surfaces in $\Cy$ arising from the Frenet lift of alternating holomorphic curves are very particular and define what we call \emph{$\alpha_1$-cyclic surfaces}, a slight generalization of the cyclic surfaces introduced by Labourie in \cite{cyclicSurfacesRank2}. 
A holomorphic interpretation of the Frenet framing defines a special type of cyclic $\sG_2'$ harmonic bundles on the Riemann surface induced by $f$. That is, cyclic $\sG_2'$-Higgs bundles equipped with a cyclic harmonic metric, see \textsection \ref{sec cyclic higgs and harmonic}. 

To establish a 1-1 correspondence between isomorphism classes of these objects, it is necessary to consider equivariant alternating holomorphic curves. Such objects are pairs $(\rho, f)$, where $\rho:\pi_1(S)\to \sG_2'$ is a representation and $f:\widetilde S\to \H$ is a alternating holomorphic curve which is $\rho$-equivariant. There are natural notions of isomorphisms in this setting.

\begin{MainTheorem}\label{MainTheo:Equivalence}
Let $S$ be an oriented surface. Then the Frenet lift defines a one-to-one correspondence between isomorphism classes of
\begin{enumerate}
	\item equivariant alternating holomorphic curves in $\H$,
	\item equivariant $\alpha_1$-cyclic surfaces in $\Cy$, and
	\item $\G$ cyclic harmonic bundles of the form \eqref{eq cyclic G2 higgs B}.
\end{enumerate}
\end{MainTheorem}
\begin{remark}
	Note that Theorem \ref{MainTheo:Equivalence} does not assume the surface $S$ is compact. This is why we have $\G$-harmonic bundles rather than polystable $\G$-Higgs bundles.  In particular, the induced Riemann surfaces could be biholomorphic to the complex plane. Existence of cyclic harmonic metrics has recently been studied by Li and Mochizuki \cite{MochizukiLi2,MochizukiLi1}.
\end{remark}

Now consider a closed oriented surface $\Sigma$ of genus $g\geq 2$. The set of equivalence classes of equivariant alternating holomorphic curves $(\rho,f)$ is denoted by
\[\cH(\Sigma)=\{(\rho, f) ~\text{equivariant alternating holomorphic curves}\}/\sim
~,\]
where the equivalence is given by the natural action of $\sG_2'\times \Diff_0(\Sigma).$ The mapping class group $\Mod(\Sigma)$ of $\Sigma$ naturally acts on $\cH(\Sigma).$ Furthermore, forgetting the map $f$, and taking the induced Riemann surface structure of $f$ define $\Mod(\Sigma)$-equivariant maps to the $\sG_2'$-character variety and the Teichm\"uller space of $\Sigma$
\[\xymatrix{\cH(\Sigma)\ar[d]_\pi\ar[r]^{\mathrm{Hol}\ \ \ }&\cX(\Sigma,\sG_2')\\\TT(\Sigma)&}\]

For closed surfaces, isomorphism classes of $\sG_2'$-cyclic harmonic bundles on a fixed Riemann surface $X=(\Sigma,\j)$ are identified with isomorphism classes of  polystable $\sG_2'$-cyclic Higgs bundles. 
In Theorem \ref{thm: moduli fixed Riemann}, we parameterize isomorphism classes of polystable $\sG_2'$-Higgs bundles which arise from equivariant alternating holomorphic curves. 
Combining this description with Simpson's construction of the moduli space of Higgs bundles in families \cite{SimpsonModuli1} (see also \cite{SO23LabourieConj}) leads to a complex analytic structure on $\cH(\Sigma)$.


\begin{MainTheorem}\label{MainTheo:Moduli1}
The moduli space $\cH(\S)$ has the structure of a complex analytic space. With respect to this structure, the mapping class group  $\Mod(\S)$ acts analytically and the projection map $\pi:\cH(\Sigma)\to \TT(\Sigma)$ is a surjective analytic map with smooth fibers.  
Moreover, the space $\cH(\Sigma)$ decomposes as 
\[\cH(\Sigma)=\coprod_{d\in\{0,\cdots,6g-6\}}\cH_d(\Sigma)~,\]
where $\cH_d(\Sigma)$ has complex dimension $d+8g-8$. 
The fiber of $\pi:\cH_d(\Sigma)\to \TT(\Sigma)$ over a Riemann surface $X$ is 
\begin{itemize}
	\item a rank $(2d-g+1)$ holomorphic vector bundle over the $(6g-6-d)$-symmetric product of $X$ when $g\leq d\leq 6g-6$, and 
	\item a bundle over a $H^1(X,\Z_2)$-cover of the $2d$-symmetric product of $X$ whose fiber is $(\C^{5g-5-d}\setminus\{0\})/\pm\Id$ when $0\leq d\leq g-1.$
\end{itemize}
\end{MainTheorem}

For the extremal values $d=6g-6$ and $d=0$, the image of the holonomy map can be classified. When $d=6g-6,$ the space $\cH_{6g-6}\to \TT(\Sigma)$ is a holomorphic vector bundle whose fiber over a Riemann surface $X$ are holomorphic differentials of degree $6.$ 
In this case, the holonomy map establishes a diffeomorphism with the $\sG_2'$-Hitchin component. When $d=0,$ the space $\cH_0(\Sigma)$ has $2^{2g}$ connected components, and the image of the holonomy map consists of all representations $\rho:\pi_1(\Sigma)\to \sG_2'$ which factor through the nonsplit $\Z_2$
-extension of $\sSL(3,\R)$ and whose $\sSL(3,\R)$-factor is  a Hitchin representation, see \textsection \ref{sec d=0} for more details. 
For $0<d<6g-6$, the representations in the image of the holonomy map have not been studied previously.

Holomorphic curves in the almost complex 6-sphere were first studied by Bryant in \cite{bryant}. His work can be thought of as an analogue of our Higgs bundle description of alternating holomorphic curves in $\H$, where the main difference is that we work in the equivariant case and noncompact symmetric spaces instead of with compact Riemann surfaces and compact symmetric spaces. 
In \cite{eschenburg}, Eschenburg and Vlachos prove that a holomorphic curve in $\mathbf S^6$ is contained in $\mathbf S^2$ (the totally geodesic case), a totally geodesic hypersurface $\mathbf S^5\subset\mathbf S^6$, or spans the entire sphere (generic case). 
Equivariant alternating holomorphic curves in $\H$ witness the same phenomenon. They either span $\H$ or are in a totally geodesic $\mathbf H^{3,2}\subset\H$, the latter case only occurs in $\cH_0(\Sigma).$ Furthermore, for $0\leq d\leq g-1,$ the components $\cH_d(\Sigma)$ contain natural degenerations of alternating holomorphic curves which lie in totally geodesic copies of the hyperbolic disc.

Finally, we study the infinitesimal properties of the holonomy map $\mathrm{Hol}:\HH(\S)\to\cX(\S,\G)$. Specifically, we show that alternating holomorphic curves are infinitesimally rigid. 
\begin{MainTheorem}\label{MainTheo:Rigidity}
For $t\in(-\epsilon,\epsilon)$, let $(\rho_t,f_t)$ be a smooth path of equivariant alternating curves such that $[\overset{\bullet}{\rho_0}]=0$ in the tangent space $\T_{[\rho_0]}\cX(\S,\G)$. Then, at $t=0,$ $(\rho_t,f_t)$ is tangent to the $\big(\G\times \Diff_0(\S)\big)$-orbit through $(\rho_0,f_0).$
\end{MainTheorem}
To prove Theorem \ref{MainTheo:Rigidity}, we show that Labourie's proof of infinitesimal rigidity of the cyclic surfaces arising from Hitchin representations into split rank 2 Lie groups can be adapted to our more general setting. In the process, we streamline many of the main ideas and computations of \cite{cyclicSurfacesRank2}. A general theory of cyclic surfaces and infinitesimal rigidity will appear in \cite{CTcyclic}. 
Theorem \ref{MainTheo:Rigidity} should be thought of as saying the holonomy map is an immersion. In particular, for a representation $\rho$ in the image of $\mathrm{Hol}$, $\rho$-equivariant alternating holomorphic curves do not come in one parameter families. 
In the case of $\sG_2'$-Hitchin representations, infinitesimal rigidity is enough to prove global uniqueness of the equivariant alternating holomorphic curve, this is not the case for the components $\cH_d(\Sigma)$ when $0<d<6g-6$.

\subsection{Related results} While this paper was being written, some analogous results were proved. In \cite{XnieCyclic}, Nie studies equivariant alternating surfaces in $\mathbf H^{p,q}$ for $(p,q)=(2k,2k)$ or $(2k,2k-2)$. The name alternating in this paper was chosen because of Nie's work. Using different techniques, he proves an infinitesimal rigidity result analogous to Theorem \ref{MainTheo:Rigidity} under some assumptions. 
For $\H$, Nie's results apply to a subset $\sSO(4,3)$ representations. Interestingly, the intersection of this subset with $\sG_2'$-representations is exactly the $\sG_2'$-Hitchin representations. In particular, the overlap of the Nie's results with Theorem \ref{MainTheo:Rigidity} is exactly the case originally covered by Labourie \cite{cyclicSurfacesRank2}.

In \cite{EvansG2poly}, Evans studies holomorphic curves in $\H$ when the underlying surface is biholomorphic to the complex plane. The holomorphic curves he considers are alternating and, in many aspects, analogous to those equivariant for a Hitchin representation. It would be interesting if similar analysis applies to general alternating holomorphic curves from the complex plane.

\subsection{Organization of the paper} In \textsection \ref{sec oct and g2}, we recall different constructions of the split octonions and define the group $\G$. The almost-complex structure on $\H$, the notion of alternating holomorphic curves and their Frenet framings are introduced in \textsection \ref{sec hol curve}. In \textsection \ref{s:HBandHolCurves} we recall the theory of $\mathsf G$-Higgs bundle and $\mathsf G$-harmonic bundle, then specifying to the case of $\mathsf G=\G$ and prove the equivalence between equivariant holomorphic bundles and certain $\sG_2'$ cyclic harmonic bundles. In \textsection \ref{sec higgs bundles and and hol curves}, contains the parameterization of the moduli spaces $\cH_d(\Sigma)$ and properties of the extremal cases. 
In \textsection \ref{s cyclic surfaces}, we develop the notion of cyclic surfaces and prove Theorem \ref{MainTheo:Rigidity}.

\subsection{Acknowledgements:}  We thank Sorin Dumitrescu, Parker Evans, Fran\c{c}ois Labourie, Xin Nie, Andy Sanders, Carlos Simpson and Richard Wentworth for insightful conversations concerning various aspects of the paper. The first authors research is partially funded by NSF DMS grant 2103685.

\section{Split octonions and the Lie group $\sG_2'$}\label{sec oct and g2}
In this section we discuss the algebra of split octonions and some of its properties. We refer the reader to \cite{baezOct,FONTANALSG2Notes} for more details. We start with the quaternions and split quaternions. 

\subsection{Quaternions and split quaternions}
Recall that the quaternions are the real associative algebra $\quat$ generated as a $\R$-vector space by $\{1,j,d,e\}$ with the relations\footnote{We do not use $i,j,k$ since we will need $\i$ for complexified objects. Our convention is that Greek letters square to $+1$ and Roman letters square to $-1.$}
\[\xymatrix{j^2=d^2=e^2=-1&\text{and}&j\cdot d=-d\cdot j=e~.}\]
On the other hand, the split quaternions are the real associative algebra $\quat'$ generated as a $\R$-vector space by $\{1,j,\delta,\epsilon\}$ with the relations
\[\xymatrix{j^2=-\delta^2=-\epsilon^2=-1&\text{and}&j\cdot \delta=-\delta\cdot j=\epsilon~.}\]

We have a vector space decomposition 
\[\xymatrix{\quat=\R\oplus \Im(\quat)&\text{and}&\quat'=\R\oplus \Im(\quat')},\]
where $\R$ is the span of the unit $1$ and $\Im$ is the span of the remaining generators. These spaces are the $\pm1$-eigenspaces of the conjugation involution $x\mapsto \overline x$. The projection onto the eigenspaces is given by taking real and imaginary $x=\Re(x)+\Im(x)$, where
\[\xymatrix{\Re(x)=\frac{1}{2}(x+\overline x)&\text{and}&\Im(x)=\frac{1}{2}(x-\overline x)}.\]

On the imaginary spaces $\Im(\quat)$ and $\Im(\quat')$, taking real and imaginary parts of the product define nondegenerate symmetric bilinear forms $\langle\cdot,\cdot\rangle$ with respective signatures $(0,3)$ and $(2,1)$, and a skew symmetric products $\times$. Namely, for $x_1,x_2$ in $\Im(\quat)$ or $\Im(\quat')$  
\[\xymatrix{\langle x_1,x_2\rangle=\Re(x_1\cdot x_2)&\text{and}&x_1\times x_2=\Im(x_1\cdot x_2)}.\] 
Combining these structures defines volume forms $\Omega$ on $\Im(\quat)$ and $\Im(\quat')$ defined by
\[\Omega(x_1,x_2,x_3)=\langle x_1\times x_2, x_3\rangle.\]

In both cases, the group of algebra automorphisms acts trivially on the real part, and so act on imaginary subspace. Hence, the groups $\Aut(\quat)$ and $\Aut(\quat')$ act on respectively $\Im(\quat)$ and $\Im(\quat')$ preserving $\langle\cdot ,\cdot\rangle$, $\times$ and $\Omega.$ Thus, 
\[\xymatrix{\Aut(\quat)< \sSO(3)&\text{and}&\Aut(\quat')< \sSO(2,1)~.}\]
In fact, $\Aut(\quat)=\sSO(3)$ and $\Aut(\quat')$ is the identity component  $\sSO_0(2,1)$ of $\sSO(2,1).$


\subsection{The split octonions from split quaternions}\label{sec split oct split quat}
The split octonions $\oct'$  can be described by the Cayley-Dickson process on $\quat'$ or $\quat$. Both perspectives will be useful but we start with the split quaternion description.
\begin{definition}
	The split octonions $\oct'$ are the real algebra with underlying vector space $\quat'\oplus\quat'$ equipped with the product
	\begin{equation}\label{eq split oct mult split quat}
	\xymatrix{\oct'=\quat'\oplus\quat'& \text{with}&(x_1,y_2)\cdot(x_2,y_2)=(x_1\cdot x_2-\overline y_2\cdot y_1~,~ y_2\cdot x_1+y_1\cdot\overline x_2)}~.
\end{equation}
The group $\sG_2'$ is defined to be the group of algebra automorphisms of $\oct'$. It is the split real form of the exceptional complex simple Lie group $\sG_2^\C$.
\end{definition}

We will write $\ell=(0,1)\in\oct'$, and write a split octonion $z=(x,y)$ as $z=x+y \ell$. Note that $\ell^2=-1$ and that, for any $y\in\quat'$, we have
\[y\ell=y\cdot \ell=(y,0)\cdot(0,1)=(0,y)=-\ell\cdot y~.\]  

Nonassociativity of the product can be seen by considering $(j \cdot \delta)\cdot \ell$ and $j\cdot(\delta\cdot \ell)$ since
\[\xymatrix{(j \cdot \delta)\cdot \ell=(\epsilon,0)\cdot(0,1)=\epsilon \ell&\text{and}&j \cdot (\delta\cdot \ell)=(j,0)\cdot(0,\delta)=(0,-\epsilon)=-\epsilon\ell~.}\]
However, one checks that any subalgebra generated by two elements of $\oct'$ is associative. In particular, for any $z_1,z_2\in\oct'$ we have 
\begin{equation}
	\label{eq assoc for 2 gen} z_1\cdot( z_1\cdot z_2)=(z_1\cdot z_1)\cdot z_2~.
\end{equation}

We have a decomposition $\oct'=\R\oplus \Im(\oct')$, where $\R$ is the span of the unit $1$ and $\Im(\oct')=\Im(\quat')\oplus \quat'$. These are the $\pm1$-eigenspaces of the conjugation involution
\[\xymatrix{z=x+y\ell\ar@{|->}[r]& \overline z=\overline x-y\ell~.}\]
 Since any two distinct generators $\{j,\delta,\epsilon,\ell,j\ell,\delta\ell,\epsilon\ell\}$ of $\Im(\oct')$ anticommute, for all $z_1,z_2\in\Im(\oct')$ we have
\[\overline{z_1\cdot z_2}=\overline z_2\cdot\overline z_1=z_2\cdot z_1~.\]
There is a nondegenerate symmetric bilinear form and a skew symmetric product defined by 
\[\xymatrix@=1em{\langle z_1,z_2\rangle=\Re(z_1\cdot z_2)=\frac{1}{2}(z_1\cdot z_2+\overline{z_1\cdot z_2})&\text{and}&z_1\times z_2=\Im(z_1\cdot z_2)=\frac{1}{2}(z_1\cdot z_2-\overline{z_1\cdot z_2})~.}\]
Let $\q$ be the quadratic form on $\Im(\oct')$ associated to $\langle\cdot,\cdot\rangle.$ Some immediate consequences of these definitions of $(\Im(\oct'),\q,\times)$ are:
\begin{lemma}
\begin{itemize}
	\item The signature of $\q$ is $(4,3)$ and $\{j,\delta,\epsilon,\ell,j\ell,\delta\ell,\epsilon\ell\}$ is an orthonormal basis.
	\item $z_1,z_2\in\Im(\oct')$ are orthogonal if and only if  $z_1\times z_2=z_1\cdot z_2.$
	\item $\Im(\oct')=\Im(\quat')\oplus\quat'$ is an orthogonal splitting, and the restriction of $\q$ to $\Im(\quat')$ and $\quat'$ has signature $(2,1)$ and $(2,2)$, respectively.
	\item For $z_1,z_2\in\Im(\quat')$ and $w_1,w_2\in\quat'$, we have 
	\[\xymatrix{z_1\times z_2\in\Im(\quat')&z_1\times w_1\in \quat'&\text{and}&w_1\times w_2\in\Im(\quat')~.}\]
	\item There is a three from $\Omega\in\Lambda^3(\Im(\oct')^*)$ defined by
\begin{equation}
	\label{eq G2 inv 3-form}\Omega(z_1,z_2,z_3)=\langle z_1\times z_2,z_3\rangle~.
\end{equation}
\end{itemize}
\end{lemma}
\begin{remark}
	The algebra structure on $\oct'$ is fully determined by $(\Im(\oct'),\q,\times)$ since for $a_1,a_2\in \R$ and $z_1,z_2\in\Im(\oct')$ we have
	\[(a_1+z_1)\cdot(a_2+z_2)=a_1a_2+\langle z_1,z_2\rangle + a_2z_1+a_1z_2+z_1\times z_2.\]
\end{remark}
The group $\sG_2'$ acts trivially on the real part. Hence, $\sG_2'$ acts on $\Im(\oct')$ preserving $\q$, $\times$ and $\Omega$. Since the algebraic structure of $\oct'$ is fully encoded in $(\Im(\oct'),\q,\times)$, $\sG_2'$ has the following description.
 \begin{proposition}
  	The group $\sG_2'$ is isomorphic to the group of linear transformations of $\Im(\oct')$ which preserve $\q$ and $\times.$ Such an automorphism also preserves $\Omega.$
  	Derivations of $\times$ defines the Lie algebra
  	\[\fg_2'=\{X\in\End(\Im(\oct'))~|~X(z_1\times z_2)=(X(z_1)\times z_2)+(z_1\times X(z_2)) \text{ for all }z_1,z_2\in\Im(\oct')\}.\] 
  	In particular, $\sG_2'<\sO(\Im(\oct'),\q)\cong \sO(4,3)$ and $\fg_2'\subset\fso(4,3).$
  \end{proposition} 
It follows from \cite[Lemma 4.2]{FONTANALSG2Notes} that there is a unique $\sG_2'$-invariant volume form $\vol\in\Lambda^7(\Im(\oct')^*)$ such that, 
\begin{equation}
	\label{eq G2 inv volume form}\langle z_1,z_2\rangle\vol=\iota_{z_1}\Omega\wedge \iota_{z_2}\Omega\wedge \Omega
\end{equation}
for all $z_1,z_2\in\Im(\oct')$. Hence $\sG_2'<\sSO(4,3)$. 
Explicitly, in the basis $\{f_1,\cdots,f_7\}=\{j,\delta,\epsilon,\ell,j\ell,\delta\ell,\epsilon\ell\}$, we have 
\[\Omega=6\Big[f_1^*\wedge (f_2^*\wedge f_3^*- f_4^*\wedge f_5^*-\wedge f_6^*\wedge f_7^*)+f_2^*\wedge (f_4^*\wedge f_6^*+ f_5^*\wedge f_7^*)+f_3^*\wedge (f_4^*\wedge f_7^*- f_5^*\wedge f_6^*)\Big]~.\]
and one checks $\vol=-\frac{1}{144}f_1^*\wedge f_2^*\wedge\cdots\wedge f_7^*$.

\begin{lemma}\label{lem:StabilizerSplitQuaternion}
	The subgroup of $\G$ which preserves the splitting $\Im(\oct')=\Im(\quat')\oplus \quat'$ is isomorphic to $\sSO(2,2)$.
\end{lemma}

\begin{proof}
Let $g=(g_1\oplus g_2)$ be an element of $\sG_2'$ which preserves the splitting $\Im(\quat')\oplus\quat'$. Since $\sG_2'$ preserve the quadratic form $\q$, then $(g_1,g_2)\in\sO(\Im(\quat'),\q)\times \sO(\quat',\q).$ 
Observe that the restriction of the $3$-form $\Omega$ from \eqref{eq G2 inv 3-form} to $\Lambda^3(\Im(\quat'))$ is nonzero. Since $\sG_2'$ preserves $\Omega$ and the volume form from \eqref{eq G2 inv volume form}, we conclude that $(g_1,g_2)\in\sSO(\Im(\quat'),\q)\times \sSO(\quat',\q).$

To conclude the proof, $\sG_2'$ preserves the product $\times$ and
	\[\xymatrix{j=\ell\times j\ell~,&\delta=\ell\times \delta\ell&\text{and}&\epsilon=\ell\times \epsilon\ell~.}\]
Hence, any $g_2\in\in\sSO(\quat',\q)$ uniquely determines $g_1\in\sSO(\Im(\quat'),\q)$.
\end{proof}

\subsection{Split octonions from quaternions}
The split octonions can also be defined using the quaternions:
\begin{equation}\label{eq split oct mult quat}
	\xymatrix{\oct'=\quat\oplus\quat& \text{with}&(a_1,b_1)\cdot(a_2,b_2)=(a_1\cdot a_2+\overline b_2\cdot b_1~,~ b_2\cdot a_1+b_1\cdot\overline a_2)}~.
\end{equation}
We will write $(a,b)=a+b\lambda$, where $\lambda=(0,1)$. Note that $\lambda^2=+1.$ 
An isomorphism between these two presentations of $\oct'$ is given by 
\begin{equation}
	\label{eq split quat vs quat iso}
	\{1,j,\delta,\epsilon,\ell,j\ell,\delta\ell,\epsilon \ell\}\to\{1,j,e\lambda,d\lambda,d,e,j\lambda,\lambda\}.
\end{equation}

In this presentation, $\Im(\oct')=\Im(\quat)\oplus \quat$, and the restriction of $\q$ to the summands $\Im(\quat)$ and $\quat$ have signatures $(0,3)$ and $(4,0)$, respectively.

\begin{lemma}\label{lem SO4 max compact}
	Let $\sK$ be the subgroup of $\sG_2'$ which preserves the splitting $\Im(\oct')=\Im(\quat)\oplus \quat$. Then we have
	\[\sK\cong \sSO(\quat,\q)\cong\sSO(4).\]
\end{lemma}
\begin{proof}
Analogous to the proof of Lemma \ref{lem:StabilizerSplitQuaternion}.
\end{proof}
\begin{remark}\label{rem times iso with selfdual}
The exterior product decomposes as $\Lambda^2(\Im(\oct'))\cong\Lambda^2(\Im(\oct'))\oplus \Im(\quat)\otimes \quat\oplus \Lambda^2\quat$. The restriction, $\times$ to $\Im(\quat)\otimes \quat$ is zero while the restrictions to the other summands give maps 
\[\xymatrix{\times|_{\Lambda^2\Im(\quat)}:\Lambda^2\Im(\quat)\to \Im(\quat)&\text{and}&\times|_{\Lambda^2\quat}:\Lambda^2\quat\to \Im(\quat)~.}\] 
The first map is an isomorphism. The second map defines an isomorphism between the vector space $\Lambda_+^2\quat$ of self-dual two forms on $\quat$ with respect to $\q_{\quat}$ and the volume form $\lambda\wedge(j\lambda)\wedge(d\lambda)\wedge (e\lambda)$ and $\Im(\quat)$. 
\end{remark}


\subsection{Complexification of $\Im(\oct')$}\label{sec complex oct'}
Let $\Im(\oct')_\C$ denote the complexification of $\Im(\oct')$ and we will denote the complex linear extensions of $\q$ and $\times$ by the same symbol. Elements of $\Im(\oct')_\C$ will be written $w+\i z$ where $w,z\in\Im(\oct').$ 

The following basis $\{f_{-3},f_{-2},\cdots, f_3\}$ of $\Im(\oct')_\C$ will be used often:
\begin{equation}\label{eq basis of complexification}
\left(\frac{1}{\sqrt{2}}(\delta\ell+\i \epsilon \ell),~\frac{1}{\sqrt{2}}(\ell+\i j \ell),~\frac{1}{\sqrt{2}}(\delta+\i \epsilon ),~j,~\frac{1}{\sqrt{2}}(\delta-\i \epsilon ),~\frac{1}{\sqrt{2}}(\ell-\i j \ell),~\frac{1}{\sqrt{2}}(\delta\ell-\i \epsilon \ell)\right)
\end{equation}
The multiplication table of this basis is given by
\begin{center}
	\begin{tabular}
        {|c|c|c|c|c|c|c|c|}\hline
            $col\times row$&$f_{-3}$&$f_{-2}$&$f_{-1}$& $f_0$&$f_1$&$f_2$&$f_3$\\ 
            \hline
            $f_{-3}$&$0$&$0$&$0$&$-if_{-3}$&$-\sqrt{2}f_{-2}$&$-\sqrt 2 f_{-1}$&$-if_0$\\    
            \hline
            $f_{-2}$&$0$&$0$&$-\sqrt 2 f_{-3}$&$if_{-2}$&$0$&$-if_0$&$\sqrt 2 f_1$\\
            \hline
            $f_{-1}$&$0$&$\sqrt 2 f_{-3}$&$0$&$if_{-1}$&$if_0$&$0$&$\sqrt 2f_2$\\
            \hline
            $f_0$&$if_{-3}$&$-if_{-2}$&$-if_{-1}$&$0$&$if_1$&$if_2$&$-i f_3$\\
            \hline
            $f_1$&$\sqrt 2 f_{-2}$&$0$&$-if_0$&$-if_1$&$0$&$\sqrt 2 f_3$&$0$\\
            \hline
            $f_2$&$\sqrt 2 f_{-1}$&$if_0$&$0$&$-if_2$&$-\sqrt 2 f_3$&$0$&$0$\\
            \hline
            $f_3$&$if_0$&$-\sqrt 2 f_1$&$-\sqrt 2 f_2$&$if_3$&$0$&$0$&$0$\\
            \hline
        \end{tabular}
	\end{center}
	 The quadratic form $\q$ and the endomorphism $J_{f_0}(\cdot)=f_0\times(\cdot)$ are given by 
\begin{equation}
	\label{eq q and J in comp basis}
	\q=\begin{pmatrix}
	&&&&&&1\\&&&&&-1\\&&&&1\\&&&-1\\&&1\\&-1\\1
\end{pmatrix}\ \ \ \ \ \ \text{and}\ \ \ \ J_{f_0}=\begin{pmatrix}
	\i\\&-\i\\&&-\i\\&&&0\\&&&&+\i\\&&&&&+\i\\&&&&&&-\i
\end{pmatrix}
\end{equation}
\begin{remark}\label{rem quatC span basis}
	Note that $\span(f_{-1},f_0,f_1)$ equals to $\Im(\quat')_\C$ and $\span(f_{-3},f_{-2},f_{2},f_3)$ equals $\quat'_\C$. With respect to the isomorphism $\Im(\quat')\oplus \quat'\to \Im(\quat)\oplus \quat$ given by \eqref{eq split quat vs quat iso}, we have 
	\[\xymatrix{\span(f_{-2},f_0,f_2)=\Im(\quat)_\C&\text{and}& \span(f_{-3},f_{-1},f_{1},f_3)=\quat_\C}.\] 
\end{remark}



\section{Holomorphic curves in the 6-pseudosphere}\label{sec hol curve}
Consider a smooth, connected and oriented surface $S$. In this section we introduce the notion of an {\em alternating holomorphic curve} in $\H$. These curves have a naturally defined Frenet framing which is used to give a holomorphic description of alternating holomorphic curves in Theorem \ref{thm hol descp of C FF}. We start by introducing relevant homogeneous spaces.
\subsection{Pseudospheres}\label{s:AlmostComplexPseudosphere}
Consider a real vector space $V$ equipped with a non-degenerate quadratic form $Q$ of signature $(p,q)$. The \emph{signature $(p-1,q)$ pseudosphere} is defined as
\[\mathbf S^{p-1,q} = \{z\in V~,~Q(z)=1\}~.\]
The tangent space to $\mathbf S^{p-1,q}$ at a point $z$ is naturally identified with
\[\T_z \mathbf S^{p-1,q} = \{x\in V~,~\langle z,x\rangle =0 \} = z^\bot~.\]
In particular, the quadratic form $Q$ restricts to a signature $(p-1,q)$ pseudo-Riemannian metric $\g$ on $\mathbf S^{p-1,q}$. This metric has constant curvature $+1$ and the orthogonal group $\sO(V,Q)$ acts by isometries on $(\mathbf S^{p-1,q},\g)$.

Similarly, the \emph{signature $(p,q-1)$ pseudohyperbolic space} is the quadric
\[\mathbf H^{p,q-1} = \{z\in V~,~Q(z)=-1\}~.\]
In the same way as for the pseudosphere, $\mathbf H^{p,q-1}$ inherits a signature $(p,q-1)$ pseudo-Riemannian metric $\g$ of curvature $-1$.

\begin{remark}
Changing the quadratic form from $Q$ to $-Q$ defines an anti-isometry between $\mathbf H^{p,q-1}$ and $\mathbf S^{q-1,p}$. As a result we will also refer to $\mathbf H^{p,q-1}$ as a pseudosphere.
\end{remark}


We will consider $(V,Q)=(\Im(\oct'),\q)$ with $\q(z)=z\cdot z$. Since $\q$ has signature $(4,3)$, the relevant pseudospheres are $\H$ and $\SS$. 
The product $\times$ on $\Im(\oct')$ will induce extra $\sG_2'$-invariant structures on $\H$ and $\SS$.

\subsection{The space $\H$}
Consider the left multiplication map
\begin{equation}
	\label{eq left mult map}L:\xymatrix@R=0em{\Im(\oct')\ar[r]&\End(\Im(\oct'))\\z\ar@{|->}[r]&L_z:w\mapsto z\times w}~.
\end{equation}
For $z\in \H$, the kernel of $L_z$ is spanned by $z$ and, since the form $\Omega$ from \eqref{eq G2 inv 3-form} is skew symmetric, the image of $L_z$ is $z^\bot$. In particular, $L_z$ restricts to an endomorphism of $\T_z\H$. 

\begin{lemma}\label{lem:AlmostComplexPseudosphere}
Let $\J:\T\H\to\T\H$ be the endomorphism defined by $\J_z={L_z}_{\vert_{z^\bot}}:\T_z\H\to \T_z\H$. Then $\J$ defines a $\sG_2'$-invariant almost complex structure on $\H$ which is compatible with the metric $\g$.
\end{lemma}

\begin{proof}
As $\J$ is constructed using the $\sG_2'$-invariant product $\times$ on $\Im(\oct')$, $\J$ is $\sG_2'$-invariant. Since $\T_z\H=z^\bot$, 
we have $\J_z(w)=z\times w = z\cdotp w$. By \eqref{eq assoc for 2 gen}, $\J$ is an almost complex structure:
\[\J_z^2(w)=z\cdot(z\cdot w)=(z\cdot z)\cdot w=-w~.\]
To see that $\J$ preserves the metric $\g,$ we compute
\[\q(\J_z(w))=(z\cdot w)\cdot(z\cdot w)=z\cdot((w\cdot z)\cdot w)=-z\cdot((z\cdot w)\cdot w)=-z^2\cdot w^2=\q(w)~,\]
 where we used that $z,w$ span an associative subalgebra of $\oct'$ and $z\cdot w=z\times w=-w\cdot z$.
\end{proof}

We now show that $\J$ is not integrable by computing $\nabla \J$, where $\nabla$ is the Levi-Civita connection of $\H$. 
Since the map $L$ from \eqref{eq left mult map} is linear, its covariant derivative using the trivial flat connection $D$ on $\Im(\oct')$ is given by
\[(D_XL)_z(Y)=X(z) \times Y(z)~,\]
where $X,Y$ are vector fields on $\Im(\oct')$. 
Since the Levi-Civita connection $\nabla$ on $\H$ is induced by $D$ and $\J$ is the restriction of $L$ to $\H$, we obtain
\begin{equation}\label{eq:nablaJ}
(\nabla_X\J)_z(Y)=\big(X(z)\times Y(z)\big)^{z^\bot}
\end{equation}
where $(~)^{z^\bot}$ is the orthogonal projection of $\Im(\oct')$ onto $z^\bot$.

\begin{lemma}\label{lem: times on C line 0}
For any $z$ in $\H$ and any nonzero $w$ in $\T_z\H$, the restriction of $\times$ to the complex line $\span(w,J_z(w))$ is zero.
\end{lemma}

\begin{proof} 
Since $z,w$ span an associative subalgebra of $\oct'$, we have
\[\J_z(w)\times w= \big((z\cdotp w)\cdotp w \big)^{z^\bot} = \big(z\cdotp (w\cdotp w)\big)^{z^\bot} =\big(\q(w) z\big)^{z^\bot}=0~.\]	
\end{proof}
Consider now the non-degenerate $2$-form $\om$ on $\H$ defined by
\[\om(X,Y)=\g(\J X,Y)~.\]
Since $\J$ is non-integrable, $d\om\neq0$. In fact, $d\om$ is the restriction of the three from $\Omega$ from \eqref{eq G2 inv 3-form} to $\H$.

\begin{lemma}
Let $W,X,Y$ be vector fields on $\H$. Then
\[d\om(W,X,Y)=\Omega(W,X,Y)= \g(W\times X,Y)~.\]
\end{lemma}
\begin{proof}
By the definition of the exterior derivative, $3d\om (W,X,Y)$ is given by
\[W(\om(X,Y))-X(\om(W,Y)+Y(\om(W,X))-\om([W,X],Y)+\om([W,Y],X)-\om([X,Y],W)~.\]
By definition we have 
\[W(\om(X,Y))=W(\g(\J X,Y))=\g((\nabla_W\J)X,Y)+ \g(\J(\nabla_WX),Y)+\g(\J X,\nabla_WY)~.\]
Using equation \eqref{eq:nablaJ} we have
\[W(g(\J X,Y))=\Omega(W,X,Y)+\om(\nabla_WX,Y)+\om(X,\nabla_WY).\]
As the Levi-Civita connection is torsion free, adding the terms gives
\[d\om(W,X,Y)=\frac{1}{3}(\Omega(W,X,Y)-\Omega(X,W,Y)+\Omega(W,Y,X))=\Omega(W,X,Y)~.\]
\end{proof}
Considering $d\om$ as a complex $3$-form on $(\H,\J)$, we can decompose it into types:
\[d\om = \theta + \zeta + \overline \zeta + \overline\theta~,\]
where $\theta$ has type $(3,0)$ and $\zeta$ has type $(2,1)$. 

\begin{lemma}\label{lem volume form}
The $(3,0)$-form $\theta$ is nowhere zero and the $(2,1)$-form $\zeta$ is identically zero. 
\end{lemma}
\begin{proof}
	The transitive action of $\sG_2'$ on $\H$ preserves $\J$ and $\Omega=d\om$, and hence preserves type decomposition of $d\om$. Thus, it suffices to compute $\theta$ and $\zeta$ at a point. 

	The imaginary octonion $z=j$ has norm $-1$ and hence $j\in\H.$ Recall the basis $\{f_{-3},\cdots,f_3\}$ of $\Im(\oct')^\C$ from \eqref{eq basis of complexification}. The complexified tangent space is given by
	\[\T_j^\C\H=j^\perp\otimes\C=\span(f_{-3},f_{-2},f_{-1},f_1,f_2,f_3).\] 
	The quadratic form $\q$ and $\J_{j}:\T^\C_j\H\to \T^\C_j\H$ are given by \eqref{eq q and J in comp basis}. In particular, $\J$ acts with eigenvalue $+\i$ on $f_{-3},f_1,f_2$ and eigenvalue $-\i$ on $f_{-2},f_{-1},f_3.$ 

	To see the form $\theta$ is nonzero at $j\in\H$ we compute
	\[\theta_j(f_{-3},f_1,f_2)=\Omega(f_{-3},f_1,f_2)=\langle f_{-3}\times f_1,f_2\rangle = -\sqrt{2}\langle f_{-2},f_2\rangle=\sqrt 2.\]
	For $\zeta_j$, it suffices to consider $\Omega(f_a,f_b,f_c)$, where $a<b\in\{-3,1,2\}$ and $c\in \{-2,-1,3\}.$ 
	We have 
	$$f_{-3}\times f_1=-\sqrt{2}f_{-2}\ ~\ ~,~\ ~\ f_{-3}\times f_2=-\sqrt 2 f_{-1}~\ \text{ and }~\ f_{1}\times f_2=\sqrt{2}f_3~,$$
	 which are all orthogonal to $\span(f_{-2},f_{-1},f_3)$.
	 Hence $\zeta_j=0.$
\end{proof}

Recall that the \emph{canonical bundle} of $\H$ is the determinant of the holomorphic cotangent bundle of $\H$. Its sections are complex $(3,0)$-forms. Since $\theta$ is never vanishing, we get

\begin{corollary}\label{cor:TrivialCanonical}
The $(3,0)$-form $\theta$ defines a trivialization canonical bundle of $(\H,\J)$. 
\end{corollary}

\begin{corollary}\label{cor:EmbeddingSU21}
The $\sG_2'$-stabilizer of a point $z\in\H$ is isomorphic to $\sSU(2,1)$.
\end{corollary}
\begin{proof}
	For $z\in\H$, the stabilizer $\Stab_{\sG_2'}(z)$ preserves the signature $(2,1)$ hermitian form $\mathbf h_z=\g_z+\i\om_z$. Since $\sG_2'$ also preserves the volume form $\theta\wedge\bar\theta$, we have $\Stab_{\sG_2'}(z)<\SU(2,1)$. Equality follows from connectedness of $\sSU(2,1)$ and a dimension count. Namely, $\dim(\H)=6$ and $\dim(\sG_2')=14$, so $\dim(\Stab_{\sG_2'}(z))=8=\dim(\sSU(2,1)).$ 
\end{proof}

\subsection{The space $\SS$}\label{ss:pseudosphere}

Recall that an \emph{almost paracomplex structure} on a manifold $M$ is an endomorphism $\psi:\T M\to\T M$ such that $\psi^2=\Id$ and $\dim(\ker(\Id-\psi))=\dim(\ker(\Id+\psi))$. A pseudo-Riemannian metric $g$ on $M$ is \emph{compatible} with $\psi$ if $\psi$ is an anti-isometry of $g$. 
Such a $g$ must have neutral signature and the distributions $\mathcal D^\pm = \ker(\psi\mp \Id)$ define a pair of transverse half dimensional isotropic subspace. The subgroup of $\sO(n,n)$ which preserves a pair of transverse half dimensional isotropic subspaces is isomorphic to $\sGL(n,\R)$, where the corresponding representation on $\mathcal D^+$ and $\mathcal D^-$ are dual to each other.

For each $z\in \SS$, the multiplication map \eqref{eq left mult map} defines an endomorphism of $L_z:\T_z\SS\to\T_z\SS.$
\begin{lemma}
Let $\psi:\T\SS\to\T\SS$ be the endomorphism defined by $\psi_z=L_z:\T_z\SS\to \T_z\SS$. Then $\psi$ defines a $\sG_2'$-invariant paracomplex structure on $\SS$ which is compatible with the metric $\g$.
\end{lemma}
\begin{proof}
	The proof is analogous to the proof of Lemma \ref{lem:AlmostComplexPseudosphere} with the exception that for $z\in \SS$ we have $z\cdot z=+1.$
\end{proof}

Each point $z\in \mathbf S^{3,3}$ defines two transverse isotropic $3$-planes $\mathcal D^\pm_z\subset z^\perp\subset \Im(\oct)$. The antipodal map on $\SS$ exchanges $\mathcal D^+$ and $\mathcal D^-$. 
Denote the quotient of $\SS$ by the antipodal map by $\P_+(\Im(\oct'))$ (that is, $\P_+(\Im(\oct'))$ is the space of positive definite lines in $\Im(\oct')$).

\begin{lemma}\label{lem 3form on dist}
The restriction of the $3$-form $\Omega$ to the distribution $\mathcal D^+$ is nowhere zero.
\end{lemma}
\begin{proof}
	The action $\G$ on $\SS$ is transitive and preserves the three form $\Omega.$ Hence, it suffices to compute $\Omega$ at any point of $\SS.$ In particular, the imaginary octonion $\delta$ has norm $+1$. Hence, $\delta \in \SS$ and $\psi_\delta=\delta\times (\cdotp) : \delta^\bot\to\delta^\bot$.

	Using the notation $z=x+y\ell$ from Section \ref{sec split oct split quat} we compute 
	\[\xymatrix{\delta\cdot (j-\epsilon)=-\epsilon +j~,&\delta\cdot (\ell+\delta\ell)= \delta\ell+\ell&\text{and}&\delta\cdot(j\ell+\epsilon\ell)=\epsilon\ell+j\ell~.}\]
	Thus, $\mathcal D_\delta^+=\span(j-\epsilon,\ell+\delta\ell,j\ell+\epsilon\ell)$. 
As $(j-\epsilon)\cdotp(\ell+\delta\ell) = 2(j\ell-\epsilon \ell)~,$ the restriction of $\Omega$ to $\mathcal D_\delta^+$ is given by 
\[\Omega(j-\epsilon,\ell+\delta\ell,j\ell+\epsilon\ell)=2\langle j\ell-\epsilon \ell, j\ell+\epsilon\ell\rangle=-4.\]
We conclude that $\Omega_\delta$ is nonzero. 
\end{proof}

\begin{corollary}\label{cor:EmbeddingSL3}
The $\sG_2'$-stabilizer of a point $z\in\SS$ is isomorphic to $\sSL(3,\R)$.
\end{corollary}

\begin{proof}
For $z\in\SS$, the stabilizer $\Stab_{\sG_2'}(z)$ preserves the distributions $\mathcal D^+_z\oplus \mathcal D^-_z$. Since $\sG_2'$ preserves the $3$-form $\Omega$, and $\Omega|_{\mathcal D^+_z}$ is nonzero, we have $\Stab_{\sG_2'}(z)<\sSL(\mathcal D_z^+)$. As in Corollary \ref{cor:EmbeddingSU21}, equality follows from connectedness of $\sSL(3,\R)$ and a dimension count. 
\end{proof}
\begin{remark}\label{rem stab of poin in proj33}
Since the antipodal map exchanges the $\pm$ distributions, the $\sG_2'$-stabilizer of a point in $[x]\in\mathbf{P}_+(\Im(\oct'))$ is isomorphic to the unique nonsplit extension 
		\[1\to\sSL(3,\R)\to \widehat\sSL(3,\R)\to\Z_2~.\]
\end{remark}




\subsection{Alternating holomorphic curves into $\H$}\label{sec alt hol curves} Let $S$ be a connected oriented surface. 
\begin{definition}
A \emph{holomorphic curve} in $\H$ is an immersion $f:S\to(\H,\J)$ such that the tangent bundle of $S$ is $\J$-invariant, that is
\[\J(df(\T S))=df(\T S) ~,\]
and the corresponding action of $\J$ on $\T S$ is orientation preserving.
\end{definition}

Given a holomorphic curve $f$, there is a unique complex structure $\j$ on $S$ such that 
\[df\circ \j = \J\circ df.\] 
We will refer to $\j$ as \emph{the induced complex structure}.

\begin{lemma}\label{lem:analytic}
Holomorphic curves into $\H$ are analytic.
\end{lemma}

\begin{proof}
Since $\G$ acts analytically on $\H$, the almost-complex structure $\J$ depends analytically on the base point. Moreover, holomorphic curves into $\H$ are solution to the Cauchy-Riemann equation $df\circ \j= \J\circ df$, which is a non-linear elliptic equation.
The result then follows from the main result of \cite{morrey}: solutions to a non-linear elliptic equations with analytic coefficients are analytic.
\end{proof}

Given a holomorphic curve $f$ from $S$ into $\H$, the pullback bundle $f^*\T\H$ is equipped with the pullback metric, connection and  complex structure that we still denote by $\g,~\nabla$ and $\J$ respectively.

A holomorphic curve $f$ is called \emph{spacelike} if the induced metric on $S$ is Riemannian. In this case, identifying $\T S$ with $df(\T S)\subset f^*\T\H$, there is $\g$-orthogonal splitting
\[f^*\T\H=\T S\oplus (\T S)^\bot~.\]
Since the almost complex structure $\J$ is compatible with the metric $\g,$ this splitting is $\J$-invariant. 
The \emph{second fundamental form} of $f$ is the $1$-form $\II\in\Omega^1(S,\Hom(\T S,(\T S)^\bot))$
 defined by 
\[\II(X,Y) = (\nabla_X Y)^\bot~.\]
Recall that $\II$ is symmetric and that $f$ is totally geodesic if and only if $\II=0.$

\begin{lemma}\label{lem:SecondFund}
Let $f:S\to \H$ be a spacelike holomorphic curve with induced complex structure $\j$ and second fundamental form $\II$. Then for any vector fields $X,Y$ on $S$, we have
\[\II(\j (X),Y)=\II(X,\j (Y))= \J(\II(X,Y))~.\]
\end{lemma}
\begin{proof}
By definition we have
\[\II(X,\j(Y)) = \big( \nabla_X(\J(Y))\big)^\bot 
= \big( (\nabla_X\J)(Y) + \J(\nabla_X Y) \big)^\bot~.\]
Since $X$ are $Y$ contained in the complex line,  Lemma \ref{lem: times on C line 0} implies $X\times Y=0.$ 
Thus, $(\nabla_X\J)(Y)=0$ by equation \eqref{eq:nablaJ}. 
Since $\J$ is compatible with the metric, we have 
\[\II(X,\j(Y)) =(\J(\nabla_X Y))^\bot=\J(\nabla_X Y)^\bot=\J(\II(X,Y))~.\]
Since $\II$ is symmetric, we have
\[\J(\II(X,Y))= \J(\II(Y,X))=\II(Y,\j(X)) = \II(\j(X),Y)~.\]
\end{proof}

\begin{proposition}\label{prop:VanishingMeanCurv}
Let $f$ be a spacelike holomorphic curve which is not totally geodesic. Then
\begin{enumerate}
	\item $f$ has vanishing mean curvature, and
	\item there is a unique $\J$-invariant rank 2 subbundle $\N S\subset(\T S)^\bot$ such that the image of $\II$ lies in $\N S$.
\end{enumerate}
\end{proposition}
We will call the bundle $\N S$ the \emph{normal bundle of $f$}.
\begin{proof}
Let $(X_1,X_2)$ be a local orthonormal framing of $\T S$ with $\j (X_1)=X_2$.
Lemma \ref{lem:SecondFund} gives
\[\II (X_2,X_2)= \II(X_2,\j(X_1)) = \J (\II(X_2,X_1))= \J^2(\II(X_1,X_1))=- \II(X_1,X_1)~.\]
In particular, the trace of $\II$ vanishes, and so the mean curvature of $f$ is zero.

For item $(2)$, there is an open set $U\subset S$ on which the image of $\II$ has maximal dimension. Locally, the image of $\II$ is spanned by $\{ \II(X_1,X_1),\II(X_1,X_2),\II(X_2,X_2)\}$.
The image of $\II$ has rank 2 on $U$ since $\II(X_1,X_1)=-\II(X_2,X_2)= -\J(\II(X_1,X_2))$ and $f$ is not totally geodesic. 

We claim that $U$ is dense. If not were not dense, there would be an open set on which the rank of $\II$ is $0$, meaning that $f(S)$ would be totally geodesic on some open space. But holomorphic curves in $\H$ are analytic by Lemma \ref{lem:analytic}, so $f(S)$ would be totally geodesic everywhere. Thus the image of $\II$ defines a complex line bundle on a dense open set, which thus extends to a complex line bundle $\N S$ on $S$.
\end{proof}
We now define the notion of a alternating holomorphic curve in $\H.$
\begin{definition}\label{def:NonDegHolCurv}
An alternating holomorphic curve is a holomorphic curve $f:S\to\H$ which is spacelike, not totally geodesic, and has negative-definite normal bundle $\N S$.
\end{definition}
We define the \emph{binormal bundle} $\B S$ of an alternating holomorphic map $f$ to be the rank two bundle 
\[\B S=(\T S\oplus \N S)^\perp.\]
Hence, an alternating holomorphic curve defines an orthogonal splitting
\[f^*\T\H = \T S\oplus\N S\oplus\B S.\]
Since $\J$ preserves $\g,$ $\T S$ and $\N S$, $\J$ also preserves the binormal bundle $\B S$.
\begin{remark}
	In \cite{XnieCyclic}, Nie defines the notion of an A-surface in $\mathbf{H}^{p,q}$. Such surfaces are assumed to be spacelike and have rank two ``higher'' normal bundles with alternating signature. In particular, for $A$-curves in $\H$ the image of the second fundamental form $\II$ is assumed to define a rank two subbundle which is assumed to be negative definite. 
	Thus, alternating holomorphic curves define $A$-surfaces in $\H$. 
\end{remark}

 \subsection{Frenet Framing}\label{sec Frenet Framing}

The trivial bundle $\underline \Im(\oct')=\H\times\Im(\oct')\to\H$ has a tautological section $\sigma$ which defines a splitting $\underline{\Im}(\oct')=\underline{\R}\oplus \T\H$, where $\sigma(x)=(1,0)$. In this splitting, the metric $g_{\Im(\oct')}=-1\oplus\g_{\H}$, and the trivial connection $\underline{D}$ decomposes as \[\underline{D}=\begin{pmatrix}
	d&\mathbb{1}^*\\\mathbb 1&\nabla
\end{pmatrix},\]
where $\nabla $ is the Levi-Civita connection on $\H$, $\mathbb 1\in\Omega^1(\H,\Hom(\underline\R,\T\H))$ is the identity and $\mathbb 1^*$ is the adjoint of with respect to the metric $g_{\Im(\oct')}$.  
By construction, the product with the tautological section $\sigma$ defines the almost complex $\J$ on $\T\H$, 
\[(1,0)\times (u,v)=(0,\J(v)).\]

Given a map $f:S\to\H$, let $(E,D,\g,\times)$ denote the pullback of the trivial bundle $\underline \Im(\oct')$, the trivial connection $\underline D$, the signature $(4,3)$ metric $g_{\Im(\oct')}$ and the product $\times:\Lambda^2\Im(\oct')\to \Im(\oct')$. We have the following decompositions
\[E=\underline \R\oplus f^*T\H\ , \ \ \ \ \ D=\begin{pmatrix}
	d&df^*\\df&f^*\nabla
\end{pmatrix}\ \ \ \ \ and\ \ \ \ \ \g=-1\oplus f^*\g_{\H}.\]
The map $f$ is identified with the pullback of the tautological section $f^*\sigma:S\to\underline \R$, and $(1,0)\times (u,v)=(0,f^*\J(v)).$

\begin{definition}\label{def frenet framing}
When $f:S\to\H$ is an alternating holomorphic curve, $E=f^*\underline\Im(\oct')$ decomposes as 
	\begin{equation}
 	\label{eq Frenet framing bundle} E=\underline \R\oplus \T S\oplus \N S\oplus \B S.
 \end{equation} 
 We call this decomposition, the Frenet framing of $f.$
\end{definition}

In the Frenet framing, the metric decomposes as $\g=-1\oplus\g_\T\oplus \g_\N \oplus \g_\B$, where $1\oplus\g_\T\oplus (-\g_\N)\oplus \g_\B$ is positive definite, and the almost complex structure decomposes as $\J=\J_\T\oplus\J_\N\oplus \J_\B$. In particular,
 \begin{equation}
 	\label{eq pullback J in FF}(1,0,0,0)\times (a,t,n,b)=(0,\J_\T(t),\J_\N(n),\J_B(b)).
 \end{equation}

Also, in the Frenet framing the connection $D$ decomposes as
\begin{equation}
	\label{eq Frenet Framing nabla}
	D= \begin{pmatrix}d&\mathbb 1^*&0&0\\
		\mathbb 1&\nabla^\T & S_2 & 0\\ 0&\II & \nabla^\N & S_3 \\0& 0 & \III & \nabla^\B 
	\end{pmatrix} 
\end{equation}
where:
 \begin{itemize}
\item $\nabla^\T,~\nabla^\N,~\nabla^\B$ are metric connections on $(\T S,\g_\T,\J_\T),~(\N S,-\g_\N,\J_\N),~(\B S,\g_\B,\J_\B)$, respectively; they are unitary with respect to the induced hermitian metrics.
\item $\mathbb 1\in\Omega^1(\underline \R, \T S)$  is the identity between the tangent bundle of $S$ and $\T S\cong df(\T S)\subset f^*\T \H$. 
\item  $\II\in\Omega^1(S,\Hom(\T S,\N S))$ is the second fundamental form of $f$.
\item $\III\in\Omega^1(S,\Hom(\N S,\B S))$ will be called the \emph{third fundamental form} of $f$.
\item $S_2\in\Omega^1(S,\Hom(\N S,\T S))$ and $S_3\in\Omega^1(S,\Hom(\B S,\N S))$ satisfy
\[\xymatrix{-\g_\N(\II(X,Y),\nu)=\g_\T(Y,S_2(X,\nu))~&\text{and}& \g_\B(\III(X,\nu),\beta) = -\g_\N(\nu,S_3(X,\beta))~,}\]
where $X,Y$ are vector fields on $S$ and $\nu, \beta$ are sections of $\N S$ and $\B S$, respectively. 
\end{itemize}

\begin{lemma}The tensors $\II$, $\III$, $S_2,$ and $S_3$ are all covariantly constant.
\end{lemma}
\begin{proof}We show $\III$ is covariantly constant, the computations for $\II$, $S_2$ and $S_3$ are similar. This follows from the Codazzi equations in this context. 
As $\H$ has constant sectional curvature, its curvature tensor $R$ satisfies $R(X,Y)(\nu)=0$. Hence,
\[0=\pi_\B(R(X,Y)\nu)=\pi_\B(\nabla_X\nabla_Y\nu-\nabla_Y\nabla_X\nu-\nabla_{[X,Y]}\nu)\]
A computation using the Frenet framing \eqref{eq Frenet Framing nabla} shows
\[\pi_\B(\nabla_X\nabla_Y\nu)=\III(X,\nabla_Y\nu)+\nabla_X^\B(\III(Y,\nu)).\]
Thus,
\[0=\III(X,\nabla_Y\nu)+\nabla_X^\B(\III(Y,\nu))-\III(Y,\nabla_X\nu)+\nabla_Y^\B(\III(X,\nu)-\III([X,Y],\nu)\]
\[=\nabla_X^{\N\B}(\III(Y,\nu))-\nabla_Y^{\N\B}(\III(X,\nu))-\III([X,Y],\nu)=2(\nabla^{\N\B}\III)(X,Y,\nu).\]
where $\nabla^{\N\B}$ is the induced connection on $\Hom(\N S,\B S).$
\end{proof}
\begin{lemma}\label{lem symm of III of II}
	The tensor $\III(\cdot,\II(\cdot,\cdot))$ is symmetric. 
\end{lemma}
\begin{proof}
Let $X,Y,Z$ be vector fields on $S.$ Since $\II(Y,Z)=\II(Z,Y)$, it suffices to compute
\[\III(X,\II(Y,Z))-\III(Y,\II(X,Z))~=~ \pi_\B\big( \nabla_X\nabla_Y Z - \nabla_Y\nabla_X Z\big) ~=~\pi_\B\big(R^\nabla(X,Y)Z - \nabla_{[X,Y]}Z \big)~.\]
Since $\nabla$ is the pull-back of a constant curvature metric, $R^\nabla(X,Y)Z \in \Omega^0(S,\T S)$. Moreover, $\nabla_{[X,Y]}Z \in \Omega^0(S,\T S\oplus \N S)$ so $\pi_\B\big(R^\nabla(X,Y)Z - \nabla_{[X,Y]}Z \big)=0$.
\end{proof}
\subsection{Complexified Frenet framing}\label{ss cFrenet}
Let $f:S\to\H$ be an alternating holomorphic curve, and let $X=(S,\j)$ be the Riemann surface defined by the induced complex structure. The complexification of the Frenet framing of $f$ is given by 
\[E_\C=\underline \C\oplus \T'\oplus \T''\oplus \N'\oplus \N'' \oplus \B'\oplus \B'',\]
where $\T_\C S=\T'\oplus \T''$ is the decomposition of the complexification of $\T S$ into $\pm\i$-eigenspaces of the complex linear extension of $\J_\T$, similarly  $\N_\C S=\N'\oplus \N''$ and $\B_\C S=\B'\oplus \B''$.
 \begin{lemma}\label{lem product in complex FF}
 In the complexified Frenet framing, the product $\times$ satisfies the following properties:
 \begin{enumerate}
 	\item For the section $s_0=(1,0,\cdots,0)$ of $E_\C$, the endomorphism $s_0\times (\cdot)$ is given by 
 	\[s_0\times(\cdot)=\diag(0,+\i,-\i,+\i,-\i,+\i,-\i).\]
 	\item For a sections $s'=(t',b',n')$ of $\T'\oplus\N'\oplus\B'$ and $s''=(t'',n'',b'')$ of $\T''\oplus\N''\oplus \B''$, we have 
 	\[s'\times s''=t'\times t''+n'\times n''+b'\times b''=-\i(\g_\T^\C(t',t'')+\g_\N^\C(n',n'')+\g_\B^\C(b',b''))s_0~.\]
 	\item The product defines isomorphisms $\T'\otimes \B'\cong \N''$ and $\T''\otimes \B''\cong \N'.$
 \end{enumerate}
 \end{lemma}
 \begin{proof}
 	The first point follows from Equation \eqref{eq pullback J in FF} and the definitions of the $\pm\i$-eigenspaces. The remaining points follow  Lemma \ref{lem volume form} and the fact that $\g^\C(s_0,s_0)=-1.$ 
 	\end{proof}

 	The complex linear extension $\II_\C\in\Omega^1(\Hom(\T' \oplus\T'',\N'\oplus\N''))$ is given by 
 \[\II_\C=\begin{pmatrix}
 	\II'&\overline \II''\\\II''&\overline \II'
 \end{pmatrix},\]
 where $\II'\in \Omega^1(\Hom(\T',\N'))$, $\II''\in\Omega^1(\Hom(\T',\N''))$, and $\overline\II'$ and $\overline\II''$ are the conjugate maps.  The tensors $\III_\C,~(S_2)_\C,$ and $(S_3)_\C$ have analogous decompositions.
\begin{proposition}\label{prop:Codazzi} With the above notation, we have
 \begin{enumerate}
 	\item $\II'' = 0$ and $S_2''=0,$ and
 	\item for any vector field $X$ on $S$ and sections $\nu,\beta$ of $\N'
 	S,\B'S$, respectively we have
 	\[\xymatrix{\III'' (X,\nu) = -\frac{\i}{2}(X \times \nu)&\text{and}&S_3''(X,\beta)=-\frac{\i}{2}(X\times \beta)}~.\]
 \end{enumerate}
 \end{proposition}
 \begin{proof}
	Note that $\II'' = 0$ if and only if $\II \circ \J_\T = \J_\N\circ \II$. So the first item is a direct consequence of Lemma \ref{lem:SecondFund} and the duality between $S_2$ and $\II$.

For the second item, denote the orthogonal projection onto $\B S$ by $\pi_\B$, note that $\pi_\B$ commutes with $\J.$ By definition, $\III(X,\nu)=\pi_\B(\nabla_X\nu),$ and
 \[\III'' (X,\nu) = \frac{1}{2}(\III_\C(X,\nu)  + \i \J(\III_\C(X,\nu)))~.\]
 Using equation \eqref{eq:nablaJ},  $\J(\nu)=\i\nu$ and $\J\circ \pi_\B=\pi_\B\circ\J$, we have
 \[\i\J(\III_\C(X,\nu))=\J(\III_\C(X,\J\nu)) = \J ( \pi_\B((\nabla_X\J)(\nu) + \J(\nabla_X \nu) ))\]
 \[=\J(\pi_\B(X\times \nu))-\III_\C(X,\nu).\]
Thus $\III'' (X,\nu) = \frac{1}{2}\J(\pi_\B(X\times \nu)).$ By Lemma \ref{lem product in complex FF}, $X\times \nu\in \B''$ since $X\in\T S$ and $\nu\in \N'$. Hence $\III'' (X,\nu)=-\frac{\i}{2}(X\times \nu).$ 
The computation for $S_3$ now follows from duality between $\III$ and $S_3$.
\end{proof}
 	

The hermitian metric on $(\T S, g_\T,\J_\T)$ induces a hermitian metric $h_{\T'}\oplus h_{\T''}$ on $\T'\oplus \T''$. The complex linear extension of $\nabla^T$ induces hermitian connections $\nabla^{\T'}$ and $\nabla^{\T''}$ on $\T'$ and $\T''$, respectively. The $(0,1)$-part of these connections define holomorphic structures $\bar\partial_{\T'}$ and $\bar\partial_{\T''}$ on $\T'$ and $\T''$, respectively. Let $\cT$ and $\overline\cT$ denote the associated holomorphic line bundles on the Riemann surface $X.$ 
Since $\cT$ and $\overline\cT$ are isotropic with respect to $\g_\T^\C$, the hermitian metric defines an isomorphism between $\overline\cT$ and the dual holomorphic line bundle $\cT^{-1}$; under this identification
\[(\cT\oplus \overline\cT,\g_\T^\C)=\left(\cT\oplus\cT^{-1},\begin{pmatrix}
	0&1\\1&0
\end{pmatrix}\right).\] 
Repeating the above discussion for the bundles $\N_\C S$ and $\B_\C S$, we see that the $(0,1)$-part of the complex linear extension of the metric connections $d,\nabla^\T,\nabla^\N, \nabla^\B$ define a holomorphic structure $\bar\partial_{E_\C}$ on $E_\C$ which decomposes as
\begin{equation}\label{eq complex Frenet hol}
	\cE=(E_\C,\bar\partial_{E_\C})=\cO_X\oplus \cT\oplus \cT^{-1}\oplus \cN\oplus \cN^{-1}\oplus \cB\oplus \cB^{-1}, 
\end{equation}
where $\cO_X$ is the trivial holomorphic line bundle. With respect to this decomposition, the complex linear extension of the positive definite metric and the Hermitian metric are given by 
\[(1\oplus \g_\T\oplus(-\g_\N)\oplus\g_\B)^\C=1\oplus \begin{pmatrix}
	0&1\\1&0
\end{pmatrix}\oplus \begin{pmatrix}
	0&1\\1&0
\end{pmatrix}\oplus \begin{pmatrix}
	0&1\\1&0
\end{pmatrix},\]
\[h=h_0\oplus h_\cT\oplus h_\cT^{-1}\oplus h_{\cN}\oplus h_{\cN}^{-1}\oplus h_{\cB}\oplus h_{\cB}^{-1}.\]
As the isomorphisms $\cT^{-1}\cong\overline\cT$, $\cN^{-1}\cong \overline\cN$ and $\cB^{-1}\cong\overline\cB$ are given by the Hermitian metric, the second form decompose as 
\[\II_\C=\begin{pmatrix}
	\II'&(\II'')^{*_h}\\\II''&(\II')^{*_h}
\end{pmatrix},\]
where $(\II')^{*_h},(\II'')^{*_h}$ are the Hermitian adjionts of $\II'$ and $\II''$. The tensors, $\III_\C,$ $S_2$ and $S_3$ decompose similarly.

In the complex Frenet framing, the flat connection $D$ decomposes as follows. 

\begin{theorem}\label{thm hol descp of C FF}
	Let $f:S\to \H$ be an alternating holomorphic curve and $X=(S,\j)$ be the induced Riemann surface. Let $(E,D)=(f^*\underline \Im(\oct'),f^*\underline D)$ be the pullback of the trivial bundle with trivial connection on $\H.$ Then, with respect to the complexified Frenet framing
	\[\cE=\cO_X\oplus\cT\oplus \cT^{-1}\oplus\cN\oplus \cN^{-1}\oplus \cB\oplus \cB^{-1},\] 
the flat connection $D$ decomposes a
	\begin{equation}
		\label{flat connection decomp}
		D=\begin{pmatrix}
			d&1^{*_h}&1&0&0&0&0\\
			1&\nabla_{h_\cT}&0&\beta^{*_h}&0&0&0\\
			1^{*_h}&0&\nabla_{h_\cT^{-1}}&0&\beta&0&0\\0&\beta&0&\nabla_{h_\cN}&0&\delta^{*_h}&\alpha^{*_h}\\0&0&\beta^{*_h}&0&\nabla_{h_{\cN}^{-1}}&\alpha&\delta\\0&0&0&\delta&\alpha^{*_h}&\nabla_{h_\cB}&0\\0&0&0&\alpha&\delta^{*_h}&0&\nabla_{h_\cB^{-1}}
		\end{pmatrix},
	\end{equation}
	where
	\begin{itemize}
		\item $1\in\Omega^{1,0}(\Hom(\cO_X,\cT))$ is the identity between $\T'X$ and $\partial f(\T'X)\cong \cT\subset f^*(\T_\C\H),$
		\item $\beta\in \Omega^{1,0}(\Hom(\cT,\cN))$ is holomorphic, 
		\item $\delta\in \Omega^{1,0}(\Hom(\cN,\cB))$ is holomorphic, 
		\item $\alpha\in \Omega^{1,0}(\Hom(\cB,\cN^{-1}))$ is nonzero, holomorphic and given by $\alpha(b)(Y)=\frac{\i}{2}b\times Y,$ for any section $b$ of $\cB$ and vector field $Y$.  
		\item $1^{*_h},~\beta^{*_h},~ \delta^{*_h}$ and $\alpha^{*_h}$ are the hermitian adjionts of $1,~\beta,~\delta$ and $\alpha$, respectively.
	\end{itemize}
\end{theorem}

\begin{remark}\label{remark branched alt curves framing}
	The line bundles $\cT$ and $\cT^{-1}$ are the holomorphic tangent and cotangent bundles of the Riemann surface $X$, respectively. Dropping the immersion assumption in the definition of an alternating holomorphic curves would give the same type of framing with $\cT$ replaced by $\cT\otimes \cO_{\mathsf{D}}$, where $\mathsf{D}\subset X$ is the divisor where the differential $\partial f$ vanishes.
\end{remark}

\begin{proof}
The flat connection in the Frenet framing is given by complex linearly extending the expression \eqref{eq Frenet Framing nabla}. 
The term $\mathbb 1_\C\in\Omega^1(\Hom(\cO_X,\cT\oplus\cT^{-1}))$ is the identification of $d_\C(\T'S\oplus\T''S)=\cT\oplus \cT^{-1}\subset f^*\T_\C\H$. As $f$ is a holomorphic immersion, $\mathbb 1_\C=1\in\Omega^{1,0}(\Hom(\cO_X,\cT)).$ 

For the second fundamental form $\II_\C$,  $\II''=0$ by Proposition \ref{prop:Codazzi}. The tensor $\beta$ is $\II'$; it is a $(1,0)$-form by Lemma \ref{lem:SecondFund}. Namely,  
\[\II'(Y,t')=\frac{1}{2}(\II_\C(Y,t')-\i\J\II_\C(Y,t'))=\frac{1}{2}(\II_\C(Y,t')-\i\II_\C(\j(Y),t')),\]
and $\II'(Y,t')$ vanishes when $\j(Y)=-\i Y$. Part 3 of Proposition \ref{prop:Codazzi} implies $\beta$ is holomorphic. 

The tensor $\alpha$ is $\III''_\C$; it has type $(1,0)$ since $\T''\times \N'=0$. The tensor $\delta$ is $\III_\C'$. Generically on the Riemann surface, we have $\III'(Y,\nu)=\III'(Y,\II'(Z,\tau))$, for $\tau$ a section of $\cT.$  By Lemma \ref{lem symm of III of II}  
\[\III'(Y,\nu)=\III'(Y,\II'(Z,\tau))=\III'(Z,\II'(Y,\tau)).\]
Hence, $\III'$ is a $(1,0)$-form since $\II'$ is a $(1,0)$-form. Part 3 of Proposition \ref{prop:Codazzi} implies $\delta$ is holomorphic.
To conclude, note that  $(S_2)_\C$ and $(S_3)_\C$ are given by
\[(S_2)_\C=\begin{pmatrix}
	0&1\\1&0
\end{pmatrix}\begin{pmatrix}
	\beta&0\\0&\beta^{*_h}
\end{pmatrix}\begin{pmatrix}
	0&1\\1&0
\end{pmatrix}\ \ \ \ \ \text{and}\ \ \ \ (S_3)_\C=\begin{pmatrix}
	0&1\\1&0
\end{pmatrix}\begin{pmatrix}
	\delta&\alpha\\\alpha^{*_h}&\delta^{*_h}
\end{pmatrix}\begin{pmatrix}
	0&1\\1&0
\end{pmatrix}.\]
\end{proof}

The group $\sG_2'$ acts on the set of alternating holomorphic curves by $(g\cdot f)(x)=g\cdot f(x).$ This action does not change the induced Riemann surface, but it can change the Frenet framing of Theorem \ref{thm hol descp of C FF} slightly. 
Namely, 
if $(\cE,h)$ and $(\hat\cE,\hat h)$ are the holomorphic bundles with induced hermitian metrics associated to $f$ and $g\cdot f$, respectively, then $f^*g$ defines an isomorphism $(\cE,h)\cong (\hat\cE,\hat h)$. 
Under this isomorphism, the holomorphic tangent bundles are the same and the holomorphic normal and binormal bundles are identified $(\cN_,h_\cN)\cong(\hat\cN,\hat h_{\hat\cN})$ and $(\cB,h_\cB)\cong(\hat\cB,\hat h_{\hat \cB})$.  Since such an isomorphism preserves the product, and the product identifies $\cN\cong \cB^{-1}\cT$, the space of such identifications is a $\sU(1)$-torsor. 
Under such an isomorphism, there is $\lambda\in \sU(1)$ such that the associated holomorphic sections $(\alpha,\beta,\delta)$ from \eqref{flat connection decomp} are changed by 
\begin{equation}
	\label{eq auto of C frentet}(\alpha,\beta,\delta)\mapsto (\alpha, \lambda\beta,\lambda^{-2}\delta)~.
\end{equation}

\section{Higgs bundles and equivariant alternating holomorphic curves}\label{s:HBandHolCurves}
Fix a Riemann surface $X=(S,\j)$ and denote its holomorphic cotangent bundle by $\cK.$ We do not assume $X$ is compact until \textsection \ref{s: moduli}. In this section we introduce the basics of Higgs bundles and harmonic bundles that are relevant to this paper. We then explain how certain $\sG_2'$-cyclic harmonic bundles determine and are determined by alternating holomorphic curves which are equivariant with respect to a representation $\rho:\pi_1(S)\to\sG_2'.$ 
 We start by recalling some Lie theory. 
\subsection{Cartan decompositions}\label{sec cartan decomp}
Every semisimple real Lie group $\sG$ has a maximal compact subgroup $\sK$ which is unique up to conjugation, and $\sG$ is homotopy equivalent to $\sK$. A choice of maximal compact subgroup defines an Lie algebra involution $\theta:\fg\to\fg$ a called a Cartan involution. The $+1$-eigenspace of the Cartan involution is the Lie algebra $\fk$ of the maximal compact $\sK$ and the $(-1)$-eigenspace $\fp$ is the subspace perpendicular to $\fk$ with respect to the Killing form. Hence a choice of maximal compact defines a \emph{Cartan decomposition} $\fg=\fk\oplus\fp$. A Cartan decomposition satisfies the bracket relations
\[\xymatrix{[\fk,\fk]\subset\fk~,&[\fk,\fp]\subset\fp&\text{and}&[\fp,\fp]\subset\fk}.\] 
Complexifying, defines the \emph{complexified Cartan decomposition} $\fg_\C=\fk_\C\oplus\fp_\C.$ 
\begin{remark}
	If $\sG_0<\sG$ is a reductive subgroup, then we can choose a maximal compact $\sK<\sG$ such that $\sK_0=\sG_0\cap\sK$ is a maximal compact of $\sG_0$. Conversely, if $\sK_0<\sG_0$ is a maximal compact subgroup, we can extend it to a maximal compact subgroup $\sK<\sG,$ see for example \cite[\textsection 3]{CartanGaloisCohomAdams}.
\end{remark}

If $(E,\q)$ is a $7$-dimensional real vector space equipped with a signature $(4,3)$ quadratic form, then the maximal compact of $\sSO_0(E,\q)\cong\sSO_0(4,3)$ is isomorphic to $\sSO(4)\times \sSO(3)$. A choice of maximal compact is given by choosing an orthogonal splitting $E=U\oplus V$ such that the restriction $\q_U$ and $\q_V$ of $\q$ to $U$ and $V$ have signature $(4,0)$ and $(0,3)$, respectively. With respect to this splitting, the Lie algebra $\fso(4,3)$ is 
\[\fso(4,3)=\left\{X=\begin{pmatrix}
	A&\eta^\dagger \\\eta&B\end{pmatrix}~|~A\in\fso(U,\q_U),~B\in\fso(V,\q_V),~\eta\in\Hom(U,V),~\eta^\dagger=-\q_U^{-1}\eta^T\q_V\right\}~.
\]
Thus, the Cartan decomposition of $\fso(4,3)$ is 
\begin{equation}
	\label{eq cartan so43}\fso(4,3)\cong\big(\fso(U,\q_U)\oplus\fso(V,\q_V)\big)\oplus\Hom(U,V).
\end{equation}

By Lemma \ref{lem SO4 max compact}, the maximal compact subgroup of $\sG_2'$ is isomorphic to $\sSO(4)$. A choice of maximal compact is given by choosing a decomposition $\Im(\oct')=\Im(\quat)\oplus\quat$, where the multiplication is given by \eqref{eq split oct mult quat}. The corresponding Cartan decomposition of $\fg_2'$ is then given by intersecting with the Cartan decomposition of $\fso(4,3)=\fso(\Im(\oct'),\q).$ 
Specifically, if 
\[(\Im(\oct'),\q)=(\quat,\q_\quat)\oplus (\Im(\quat),\q_{\Im(\quat)})=(U,\q_U)\oplus (V,\q_V),\]
then the Cartan decomposition of $\fg_2'$ is given by $\fg_2'=\fk\oplus\fp,$ where
\[\fk=\left\{X=\begin{pmatrix}
	A&\\&B
\end{pmatrix}~|~A\in\fso(U,\q_U)~,~B\in\fso(V,\q_V)~,~X(z_1\times z_2)=X(z_1)\times z_2+z_1\times X(z_2)\right\}\] 
\[\fp=\left\{X=\begin{pmatrix}
	&\eta^\dagger\\\eta&
\end{pmatrix}~|~\eta^\dagger=-\q_U^{-1}\eta^T\q_V~,~X(z_1\times z_2)=X(z_1)\times z_2+z_1\times X(z_2)\right\}~.\]

\subsection{Higgs bundles definitions}
Let $X=(S,\j)$ be a Riemann surface. Fix a maximal compact subgroup $\sK<\sG$ and consider the Cartan decomposition $\fg=\fk\oplus\fp$ of the Lie algebra of $\sG$. Let $\sK_\C<\sG_\C$ be the complexification of $\sK$ and denote the complexified Cartan decomposition by $\fg_\C=\fk_\C\oplus \fp_\C.$ 
Note that $\sK_\C$ acts on $\fp_\C$; given a principal $\sK_\C$-bundle $P\to X$ we denote $P[\fp_\C]$ the associated vector bundle $P\times_{\sK_\C}\fp_\C$ with fiber $\fp_\C$.
\begin{definition} 
With the above notation, a $\sG$-Higgs bundle on $X$ is a pair $(\cP,\Phi)$, where 
\begin{itemize}
	\item $\cP$ is a holomorphic principal $\sK_\C$-bundle over $X$ and 
	\item $\Phi$ is a holomorphic section of $\cP[\fp_\C]\otimes\cK$ called the Higgs field.
\end{itemize}
\end{definition}
If we fix the underlying smooth $\sK_\C$-bundle $P$, then two $\sG$-Higgs bundles $(\cP,\Phi)$ and $(\cP',\Phi')$ are isomorphic if differ by the action of the $\sK_\C$-gauge group. That is, if there is a smooth bundle isomorphism of $P$ which defines an holomorphic isomorphism $(\cP,\Phi)\cong(\cP',\Phi').$

A complex group $\sG_\C$ is a real form of $\sG_\C\times \sG_\C$. In this case,  $\sK_\C\cong\sG_\C$ and $\fp_\C\cong\fg_\C$, and a $\sG_\C$-Higgs bundle is a pair $(\cP,\Phi)$, where $\cP$ is a holomorphic $\sG_\C$-bundle on $X$ and $\Phi$ is a holomorphic section of the adjoint bundle $\cP[\fg_\C]\otimes \cK$.
Using faithful representations of $\sG_\C$, we can work with vector bundles. We give a few relevant examples:

\begin{example}
	An $\sSL(n,\C)$-Higgs bundle on $X$ is a tuple $(\cE,\vol_\cE,\Phi)$, where 
\begin{itemize}
	\item $\cE\to X$ is a holomorphic vector bundle of rank $n$,
	\item $\vol_\cE:\Lambda^n\cE\to\cO_X$ is a holomorphic isomorphism, and
	\item $\Phi$ is a holomorphic section of $\End(\cE)\otimes \cK$ with $\tr(\Phi)=0.$ 
\end{itemize}
Similarly, an $\sSO(n,\C)$-Higgs bundle is an $\sSL(n,\C)$ Higgs bundle equipped with a nondegenerate holomorphic section $Q_\cE$ of the symmetric power $S^2\cE^*$ such that $\Phi^TQ_\cE+Q_\cE\Phi=0,$ where $Q_\cE$ is viewed an isomorphism $Q_\cE:\cE\to\cE^*$ and $\Phi^T:\cE^*\to \cE^*\otimes \cK$. 
\end{example}
\begin{example}
	The complexified Cartan decomposition of $\fso(4,3)$ is given by complexifying \eqref{eq cartan so43}.
 An $\sSO_0(4,3)$-Higgs bundle is thus a tuple $(\cU,Q_\cU,\vol_\cU,\cV,Q_\cV,\vol_\cV,\eta)$, where 
\begin{itemize}
	\item $(\cU,Q_\cU,\vol_\cU)$ and $(\cV,Q_\cV,\vol_\cV)$ are respectively rank $4$ and $3$ holomorphic vector bundles equipped with holomorphic orthogonal structures $Q_\cU,Q_\cV$ and volume forms $\vol_\cU,\vol_\cV$ and
	\item $\eta:\cU\to\cV\otimes \cK$ is a holomorphic bundle map. 
\end{itemize}
The associated $\sSO(7,\C)$-Higgs bundle is $(\cE,\vol_\cE,Q_\cE,\Phi)=(\cU\oplus \cV,\vol_\cU\wedge \vol_\cV,Q_\cU\oplus -Q_\cV,\Phi),$ where for $\eta^\dagger=Q_\cU^{-1}\circ \eta^T\circ Q_\cV$ we have 
\begin{equation}
	\label{eq phi for so43 higgs}\Phi=\begin{pmatrix}
	0&\eta^\dagger\\\eta&0
\end{pmatrix}:\cU\oplus \cV\to(\cU\oplus \cV)\otimes \cK~.
\end{equation}
\end{example}

We now describe $\G$-Higgs bundles. Recall that $\fg_2'\subset \fso(\Im(\oct'),\q)\cong \fso(4,3)$ is the subalgebra of derivations of $(\Im(\oct'),\times)$.  We will suppress the volume forms from the notation.
 
\begin{definition}
	A $\sG_2'$-Higgs bundle on $X$ is an $\sSO_0(4,3)$-Higgs bundle $(\cU,Q_\cU,\cV,Q_\cV,\eta)$ equipped with a holomorphic bundle map $\times:\Lambda^2(\cV\oplus \cU)\to \cV\oplus \cU$ such that
\begin{itemize}
	\item $(\cV_p\oplus\cU_p,\times_p)\cong(\Im(\quat)_\C\oplus\quat_\C,\times)$ for every $p\in X$, and
	\item $\Phi(s\times t)=\Phi(s)\times t+s\times\Phi(t)$, were $\Phi$ is given by \eqref{eq phi for so43 higgs} and $s,t$ are any sections of $\cU\oplus\cV$.
\end{itemize} 
\end{definition}
\begin{remark}
By Remark \ref{rem times iso with selfdual}, this implies that $\times|_{\Lambda^2\cU}:\Lambda^2\cU\to \cV$ defines a holomorphic isomorphism between $\cV$ and the self-dual part $\Lambda^2_+\cU$ of the second exterior product.
\end{remark}
\begin{remark}
	Consider the flat connection $D$ arising from an alternating holomorphic curve. Using the notation of Theorem \ref{thm hol descp of C FF} $(\cE,\times,\Phi)$ is a $\sG_2'$-Higgs bundle, where  $\Phi$ is the $(1,0)$-part of $D-\nabla$.
\end{remark}

\subsection{An example}\label{sec Example Higgs}
The Higgs bundles relevant to this paper have the additional property that the bundle will have a holomorphic reduction to the maximal torus. This means that the vector bundle decomposes as a direct sum of holomorphic line bundles. Specifically, 
\begin{equation}
	\label{eq line bundle splitting}\cE=\cL_{-3}\oplus\cL_{-2}\oplus\cL_{-1}\oplus\cL_0\oplus\cL_{1}\oplus\cL_{2}\oplus\cL_{3},
\end{equation}
where each $\cL_i$ is holomorphic a line bundle. We furthermore impose $\cL_{-i}=\cL_i^{-1}$ and $\cL_0=\cO$. 
With respect to the splitting \eqref{eq line bundle splitting}, a Higgs field $\Phi:\cE\to\cE\otimes \cK$ can be written as a $7\times 7$-matrix whose entries are holomorphic sections of appropriate line bundles. We will consider Higgs fields of the form
\begin{equation}
	\label{eq: Higgs field in splitting}
	\Phi= \begin{pmatrix}
		0&0&0&0&0&\delta&0\\\alpha&0&0&0&0&0&\delta\\0&\beta&0&0&0&0&0\\0&0&\gamma&0&0&0&0\\0&0&0&\gamma&0&0&0\\0&0&0&0&\beta&0&0\\0&0&0&0&0&\alpha&0
	\end{pmatrix}:\cE\to\cE\otimes \cK,
\end{equation}
where $\alpha,\beta,\gamma,\delta$ are holomorphic sections of $\cL_3\cL_{-2} \cK$, $\cL_2\cL_{-1} \cK,$ $\cL_1\cL_0\cK$ and $\cL_{-3}\cL_{-2} \cK$, respectively. To avoid writing $7\times7$ matrices, we will represent such an object by
\begin{equation}\label{eq:Higgsdiagram}(\mathcal E,\Phi) := \xymatrix{\cL_{-3}  \ar@/_/[r]_{\alpha} & \cL_{-2}  \ar@/_/[r]_\beta & \cL_{-1} \ar@/_/[r]_{\gamma}  & \cL_0 \ar@/_/[r]_{\gamma} & \cL_1 \ar@/_/[r]_\beta  & \cL_2 \ar@/_/[r]_{\alpha} \ar@/_2pc/[lllll]_{\delta} & \cL_3 \ar@/_2pc/[lllll]_{\delta}}.\end{equation}
Since $\tr(\Phi)=0$ and $\Lambda^7\cE$ is trivializable, $(\cE,\Phi)$ an $\sSL(7,\C)$-Higgs bundle. 

In fact, this example defines an $\sSO_0(4,3)$-Higgs bundle $(\cU,Q_\cU,\cV,Q_\cV,\eta)$, where 
\begin{itemize}
	\item $\cU=\cL_{-3}\oplus\cL_{-1}\oplus\cL_{1}\oplus\cL_{3}$ and $\cV=\cL_{-2}\oplus \cL_0\oplus\cL_{2}$,
	\item $Q_\cU,~Q_\cV$ are orthogonal structures on $\cU$ and $\cV$ defined by
	\[Q_\cU=\begin{pmatrix}
		&&1\\&\iddots&\\1
	\end{pmatrix}:\cU\to\cU^*\ \ \ \ \text{and}\ \ \ \ Q_\cV=\begin{pmatrix}
		&&1\\&1\\1
	\end{pmatrix}:\cV\to\cV^*,\]
	\item $\eta=\begin{pmatrix}
		\alpha&0&0&\delta\\0&\gamma&0&0\\0&0&\beta&0
	\end{pmatrix}:\cU\to\cV\otimes K$.
\end{itemize}   
For this example to be a $\sG_2'$-Higgs bundle, we need a holomorphic map 
\[\times: \Lambda^2(\cU\oplus\cV)\to \cU\oplus\cV,\]
such that $(\cU_p\oplus \cV_p,\times)\cong (\Im(\quat)_\C\oplus \quat_\C,\times)$ for each $p\in X $. This implies $\cV\cong \Lambda^2_+\cU$. We will define a product as in the complex Frenet framing \eqref{eq complex Frenet hol} of an alternating holomorphic curve.

Using $\cL_0=\cO$, if $s_0=1\in H^0(\cL_0)$ , then define
\[s_0\times (\cdot)=\diag(+\i,-\i,-\i,0,+\i,+\i,-\i): \cL_{-3}\oplus\cdots\oplus \cL_3\to \cL_{-3}\oplus \cdots\oplus \cL_3.\]
If we set $\cL_{-2}=\cL_{-3} \cL_1,$ then any nonzero complex number $\xi$ determines an isomorphism 
\[\xi:\cL_{-3}\otimes\cL_1\to\cL_{-2}= \cL_{-3} \cL_1.\]
 Using the isomorphism, we can define a product as in Lemma \ref{lem product in complex FF}. 
Namely, given sections $s_{-3},s_1$ of $\cL_{-3}$ and $\cL_1$: 
\[s_{-3}\times s_1=\xi \cdot s_{-3}\otimes s_1.\] 
The multiplication table is the same as that of $\{f_{-3},\dots,f_3\}$ of $\Im(\oct')_\C$ from \eqref{eq basis of complexification} where the $\sqrt 2$ is replaced by $\xi$.

 \begin{lemma}
With the above choice of product $\times:\Lambda^2(\cU\oplus \cV)\to \cU\oplus \cV$, the Higgs bundles \eqref{eq: Higgs field in splitting} is a $\sG_2'$-Higgs bundle if and only if we have the following equation in $H^0(\cK\otimes\Hom(\cL_{-3},\cL_{-2}))$
\[\alpha(\cdot)=-\frac{\i}{2}\gamma(s_0)\times (\cdot),\]
 where $s_0=1\in H^0(X,\cL_0)$ is a trivializing unit section. 
\end{lemma}
\begin{proof}
To be a $\sG_2'$-Higgs bundle, the Higgs field $\Phi$ must be a derivation. We compute $\Phi(s_0\times s_{-3})=\Phi(s_0)\times s_{-3}+s_0\times \Phi(s_{-3})$ implies $\alpha(s_{-3})=-\frac{\i}{2}\gamma(s_0)\times s_{-3}.$ Indeed,
\[\Phi(s_0)\times s_{-3}+s_0\times \Phi(s_{-3})=\gamma(s_0)\times s_{-3}-\i\alpha(s_{-3}),\]
while $\Phi(s_0\times s_{-3})=\i\Phi(s_{-3})=\i\alpha(s_{-3})$. To see that this is the only condition on the Higgs field, one computes $\Phi(s_a\times s_b)$ and $\Phi(s_a)\times s_b+s_a\times \Phi(b)$ using the multiplication table \eqref{eq basis of complexification}, we leave this to the reader. 
\end{proof}
\begin{remark}\label{remark }
	 Note that if $\cL_{-2}=\cL_{-3} \cL_1$ and $\alpha$ is a scalar multiple of $\gamma,$ then we can define a product $\times$ making a Higgs bundle of the form \eqref{eq:Higgsdiagram} a $\sG_2'$-Higgs bundle. 
\end{remark}

\subsection{Harmonic bundles} 
A $\sG$-harmonic bundle is a $\sG$-Higgs bundle equipped with a metric which solves a PDE related to the flatness of a connection. 
We will focus on the vector bundle description.

\begin{definition}
	An $\sSL(n,\C)$-harmonic bundle on $X$ is an $\sSL(n,\C)$-Higgs bundle $(\cE,\Phi)$ equipped with a hermitian metric $h$ which induces the trivial metric on $\Lambda^n\cE$ and satisfies the Hitchin equations
	\begin{equation}
		\label{eq Hitchin eq}F_{h}+[\Phi,\Phi^*]=0~,
	\end{equation}
	where $F_h$ is the curvature of the Chern connection $\nabla_h$ of $h$ and $\Phi^*$ is the hermitian adjoint of $\Phi.$ 
\end{definition}
Let $(E,h)$ be a smooth bundle unitary bundle. Two harmonic bundles $(\cE,\Phi,h)$ and $(\cE',\Phi',h)$ on $(E,h)$ are isomorphic if they differ by unitary bundle automorphism.

Given an $\sSL(n,\C)$-harmonic bundle $(\cE,\Phi,h)$, the Hitchin equations \eqref{eq Hitchin eq} are equivalent to the flatness of the connection
\[D=\nabla_h+\Phi+\Phi^*~.\] 
Hence an $\sSL(n,\C)$-harmonic bundle defines a (conjugacy class of) representation $\rho:\pi_1(X)\to\sSL(n,\C)$ such that $E\cong \widetilde X\times_\rho\C^n$.
Given a representation $\rho:\pi_1(X)\to\sSL(n,\C)$, a metric $h_\rho$ on the flat bundle $\widetilde X\times_\rho\C^n$ can be interpreted as a $\rho$-equivariant map to the Riemannian symmetric space: 
\[h_\rho:\widetilde X\to\sSL(n,\C)/\sSU(n).\]
A metric $h_\rho$ is called harmonic if, for every compact subset $K\subset \widetilde X,$ it is a critical point of the energy  
\[\cE_K(h_\rho)=\frac{1}{2}\int_K|dh_\rho|^2.\]
This makes sense since, for two dimensional domains, the energy only depends on the conformal structure of the domain.
It turns out that a metric $h$ solves the Hitchin equations \eqref{eq Hitchin eq} if and only if the $\rho$-equivariant map $h_\rho$ is harmonic. 


For $\sSL(n,\C)$-Higgs bundles which have extra structures related to being a $\sG$-Higgs bundle, the harmonic metric $h$ is assumed to be compatible with these structures, i.e., the Chern connection $\nabla_h$ has holonomy in the compact group $\sK<\sK_\C$, and hence the associated flat connection $D$ has holonomy in $\sG$. In terms of harmonic maps, the associated harmonic map factors through a copy of the symmetric space $\sG/\sK\subset \sSL(n,\C)/\sSU(n).$
For example, a hermitian metric $h$ on $\cE$ is compatible with a nondegenerate symmetric bilinear from $Q_\cE$ if there exists a conjugate linear involution $\lambda:\cE\to\overline\cE$ such that $h(e_1,e_2)=Q(e_1,\lambda(e_2))$. In particular, viewing $h$ as an isomorphism $H:\cE\to\overline \cE^*$ and $Q_\cE$ as an isomorphism $Q:\cE\to\cE^*$, we have $\lambda=Q^{-1}\circ H$; hence $\overline H\circ Q^{-1}\circ H=\overline Q.$

An $\sSO_0(4,3)$-harmonic metric on an $\sSO_0(4,3)$-Higgs bundle $(\cU,Q_\cU,\cV,Q_\cV,\eta)$ consists of two hermitian metrics $h_\cU,h_\cV$ on $\cU,\cV$ which are compatible with $Q_\cU,Q_\cV$, respectively. Note that the associated flat bundle $E$ decomposes as $E=U\oplus V$, where $U\subset \cU$ and $V\subset \cV$. Moreover, $Q_\cU|_{U}\oplus (-Q_{\cV})|_{V}$ defines a parallel metric on $E$ of signature $(4,3).$

For an $\sSO_0(4,3)$-Higgs bundle which is a $\sG_2'$-Higgs bundle, the quadratic form $Q_\cV$ is determined by $Q_\cU$ and the product $\times.$ We have the following definition.

\begin{definition}
	A $\sG_2'$-harmonic metric on a $\sG_2'$-Higgs bundle $(\cU,Q_\cU,\cV,Q_\cV,\times,\eta)$ is a pair of compatible metrics $h_\cU,h_\cV$ on $\cU,\cV$, respectively, such that the involutions $\lambda_\cV,\lambda_\cU$ satisfy 
	\[\lambda_{\cV}(u_1\times u_2)=\lambda_\cU(u_1)\times \lambda_\cU(u_2).\]
\end{definition}
\begin{remark}
	Note that the flat bundle $E$ associated to a $\sG_2'$-harmonic bundle decomposes as $E=U\oplus V$ and has a parallel product $\times:\Lambda^2E\to E$ which identifies the fibers with $\Im(\oct').$
\end{remark}


 \subsection{Cyclic Higgs and harmonic bundles}\label{sec cyclic higgs and harmonic}
The explicit Higgs bundles considered in Section \ref{sec Example Higgs} have additional symmetries which make them $6$-cyclic Higgs bundles. 
\begin{definition}
 	An $\sSL(n,\C)$-Higgs bundle $(\cE,\Phi)$ is called $k$-cyclic if there is a holomorphic splitting $\cE=\cE_1\oplus\cdots\oplus\cE_k$ such that $\Phi(\cE_i)\subset\cE_{i+1}\otimes\cK$, where $i+1$ is taken mod $k.$ The splitting $\cE=\cE_1\oplus\cdots\oplus\cE_k$ will be called the cyclic splitting. 
 \end{definition}

 \begin{definition}
 An $\sSL(n,\C)$-harmonic bundle $(\cE,\Phi,h)$ is called a $k$-cyclic harmonic bundle if $(\cE,\Phi)$ is $k$-cyclic and the cyclic splitting is orthogonal with respect to the metric $h.$
 \end{definition}   

Two $k$-cyclic Higgs bundles $(\cE_1\oplus\cdots\oplus\cE_k,\Phi)$ and  $(\cE_1'\oplus\cdots\oplus\cE_k',\Phi')$ are isomorphic if there is a holomorphic bundle automorphism which identifies $\cE_j\cong \cE_j'$ for all $j$ and $\Phi$ and $\Phi'$. Similarly, two $k$-cyclic harmonic bundles are isomorphic if such an isomorphism is unitary. We note that cyclic Higgs bundles and cyclic harmonic bundles have a cyclic group symmetry. Namely, if $(\cE,\Phi)$ is a $k$-cyclic Higgs bundle, consider the unitary holomorphic bundle automorphism
\[g=\diag(\zeta^{a}\Id_{\cE_1},\zeta^{a-1}\Id_{\cE_2},\dots,\zeta^{a-k+1}\Id_{\cE_k}):\cE_1\oplus\cdots\oplus \cE_k\to\cE_1\oplus\cdots\oplus \cE_k,\]
where $\zeta$ is a primitive $k^{th}$ root of unity and $a$ is chosen so that $\det(g)=1.$ Then $g$ acts on $\Phi$ by
\[g\cdot\Phi=g^{-1}\circ \Phi\circ g=\zeta\cdot \Phi.\]

 \begin{remark}
When $X$ is compact, polystability of the above Higgs bundles  automatically implies the existence of cyclic harmonic metrics, see Proposition \ref{prop cyclic stab and metric}.
When $X$ is noncompact, see \cite{MochizukiLi1,MochizukiLi2} for  results regarding the existence and uniqueness of cyclic harmonic metrics.
\end{remark}

\begin{lemma}\label{lem metric splits on line bundles}
	The Higgs bundles from \eqref{eq:Higgsdiagram} are $6$-cyclic with cyclic splitting
	\[\cE_1\oplus\cE_2\oplus\cE_3\oplus\cE_4\oplus\cE_5\oplus\cE_6=(\cL_{-3}\oplus\cL_3)\oplus\cL_{-2}\oplus\cL_{-1}\oplus\cL_0\oplus\cL_{1}\oplus\cL_{2}.\]
	Moreover, if $h$ is a cyclic $\sSO_0(4,3)$-harmonic metric, then $h=h_{-3}\oplus h_{-2}\oplus\cdots\oplus h_{2}\oplus h_3$ where $h_a$ is a hermitian metric on $\cL_a$ and $h_{-a}=h_a^{-1}.$
\end{lemma}

\begin{proof}
	The Higgs bundle is $6$-cyclic since rearranging the summands from \eqref{eq:Higgsdiagram} yields
	\[\xymatrix{\cL_{-3}\oplus\cL_3  \ar@/_/[r]_{(\alpha\ \ \ \delta)} & \cL_{-2}  \ar@/_/[r]_\beta & \cL_{-1} \ar@/_/[r]_{\gamma}  & \cL_0 \ar@/_/[r]_{\gamma} & \cL_1 \ar@/_/[r]_\beta  & \cL_2 \ar@/_2pc/[lllll]_{(\delta\ \ \ \alpha)^T}}.\] 

	Recall that this is $\sSO_0(4,3)$-Higgs bundle, with $\cU=\cL_{-3}\oplus\cL_{-1}\oplus\cL_1\oplus\cL_3$, $\cV=\cL_{-2}\oplus\cL_0\oplus\cL_2$ and quadratic forms
	\begin{equation}\label{eq C bilinear forms}
		Q_\cU=\begin{pmatrix}
		&&1\\&\iddots\\1
	\end{pmatrix}\ \ \ \ \  \ \ \ \text{and}\ \ \ \  \ \ \ Q_\cV=\begin{pmatrix}
		&&1\\&1\\1
	\end{pmatrix}~, 
	\end{equation}
Thus, an $\sSO_0(4,3)$-cyclic harmonic metric $h_\cU\oplus\ h_\cV$ is diagonal in the cyclic splitting. The condition that $\lambda_\cU=Q_\cU^{-1}\circ h_\cU$ and $\lambda_\cV=Q_\cV^{-1}\circ h_\cV$ are conjugate linear involutions of $\cU$ and $\cV$, respectively, implies that the metric on $\cL_{-3}\oplus\cL_{3}$ splits as $h_{-3}\oplus h_3$, where $h_3^{-1}=h_{-3}$, and that the metric $h_j$ on $\cL_j$ satisfies $h_j^{-1}=h_{-j}.$
\end{proof}
\begin{lemma}\label{lem G2' cyclic metric}
	When $\cL_{-2}=\cL_{-3} \cL_{1}$ and $\alpha$ is a scalar multiple of $\gamma$, a $\sG_2'$-cyclic harmonic metric is an $\sSO_0(4,3)$-cyclic harmonic metric where the metric $h_{-2}$ on $\cL_{-2}$ is $h_{-3} h_1$ and the metric $h_2$ on $\cL_2$ is $h_3h_{-1}$.
\end{lemma}
\begin{proof}
Recall from Remark \ref{remark } that when $\cL_{-2}=\cL_{-3} \cL_1$ and $\alpha$ and $\gamma$ are scalar multiples of each other, the Higgs bundle under consideration is a $\sG_2'$-Higgs bundle.  
Thus, the conjugate linear involutions $\lambda_\cU=Q_\cU^{-1}\circ h_\cU$ and $\lambda_\cV=Q_\cV^{-1}\circ h_\cV$ associated to a $\sG_2'$-harmonic metric  satisfies
\[\lambda_\cU(u_1)\times \lambda_\cU(u_2)=\lambda_\cV(u_1\times u_2),\]
for all $u_1,u_2\in\cU.$  
If $\alpha=\frac{\i\xi}{2}\gamma$, then $s_{-3}\times s_1=\xi \cdot s_{-3}\otimes s_1.$ Using the metric decomposition from Lemma \ref{lem metric splits on line bundles},  we have 
\[\lambda_\cU(s_{-3})\times\lambda_\cU(s_1)=h_{-3}h_1s_{3}\times s_{-1}=h_{-3}h_1\xi(s_{3}\otimes s_{-1}),\]
\[\lambda_\cV(s_{-3}\times s_1)=\lambda_\cV(\xi(s_{-3}\otimes s_1))=  h_{-2}\xi(s_3\otimes s_1).\]
Hence, $h_{-2}=h_{-3}h_1$ and $h_2=h_3h_{-1}.$
\end{proof}


We now deduce a decomposition of the flat bundle associated to these $\sG_2'$-cyclic harmonic bundles which is analogous to  \cite[Theorem 2.33]{CTT}. 
Let $D=\nabla_h+\Phi+\Phi^*$ be the associated flat connection. Since the metric is diagonal, in the splitting $\cE=\cL_0\oplus\cL_1\oplus\cL_{-1}\oplus\cL_{2}\oplus\cL_{-2}\oplus\cL_{-3}\oplus\cL_3,$ the connection $D$ decomposes as in \eqref{flat connection decomp} of Theorem \ref{thm hol descp of C FF} where $\cT$ is replaced with $\cL_1$, $\cB$ is replaced with $\cL_{-3}$, and $1$ and $1^*$ are replaced with $\alpha$ and $\alpha^*.$ 
The associated $\sG_2'$-flat bundle $E_D\subset\cE$ is the fixed point locus of the antilinear involution $\lambda:\cE\to\cE$ defined by $h(u,v)=Q(u,\lambda(v))$. 
The bundle $E_D$ decomposes as $E_D=U\oplus V$, where $U\subset \cL_{-3}\oplus\cL_{-1}\oplus\cL_1\oplus\cL_3$ and $V\subset\cL_{-2}\oplus\cL_0\oplus\cL_2.$  

\begin{proposition}\label{prop flat conn decomp}
	The flat bundle $E_D$ of a $\sG_2'$-cyclic harmonic bundle of the form \eqref{eq:Higgsdiagram} decomposes as 
	\[E_D=\ell\oplus U_1\oplus V_1\oplus  U_2,\]
	where $\ell\subset \cL_0$ and $V_1\subset\cL_{-2}\oplus\cL_2$ are negative definite subbundles of rank 1 and 2 respectively, and $\cU_1\subset \cL_{-1}\oplus\cL_1$ and $U_2\subset\cL_{-3}\oplus\cL_3$ are positive definite rank 2 subbundles. 
\end{proposition}
\begin{proof}
The complex bilinear forms $Q_\cU$ and $Q_\cV$ are given by \eqref{eq C bilinear forms}. Since the metric $h$ is diagonal, in the splitting $\cE=\cL_0\oplus(\cL_{-1}\oplus\cL_1)\oplus(\cL_{-2}\oplus\cL_2)\oplus(\cL_{-3}\oplus\cL_3)$, the  involution $\lambda$ is given by 
\[\lambda=-h_0\oplus\begin{pmatrix}&h_{1}\\h_1^{-1}&
	
\end{pmatrix}\oplus\begin{pmatrix}&-h_{2}\\-h_{2}^{-1}&
	
\end{pmatrix}\oplus \begin{pmatrix}&h_{3}\\h_{3}^{-1}&
	
\end{pmatrix},\]
where $h_j(w):\cL_j\to\cL_{-j}$ is the antilinear map defined by $h_j(w)(w')=h_j(w,w').$
\end{proof}

\subsection{Cyclic harmonic bundles and equivariant alternating holomorphic maps} We now establish a 1-1 correspondence between isomorphism classes of certain $\sG_2'$-cyclic harmonic bundles and equivariant alternating holomorphic curves. 
Fix a basepoint $x_0\in S$ and let $\pi_1(S)=\pi_1(S,x_0)$ denote the fundamental group of $S$. Fix also a universal covering $\widetilde S\to S$ and a lift $\tilde x_0$ of the basepoint. We will suppress the basepoints from the notation.
\begin{definition}
	An equivariant alternating holomorphic curve on $S$ is a pair $(\rho,f)$ consisting of a representation $\rho:\pi_1(S)\to\sG_2'$ and an alternating holomorphic curve $f:\widetilde S\to \H$ which is $\rho$-equivariant, that is, $f(\gamma\cdot x)=\rho(\gamma)\cdot f(x)$ for all $x\in\widetilde S$ and all $\gamma\in\pi_1(S).$
\end{definition}

When $S$ is simply connected an alternating holomorphic curve is of course the same as an equivariant alternating holomorphic curve. For example, $(S,\j)$ could be isomorphic to the complex plane as in \cite{EvansG2poly}. Note that the Riemann surface structure $\widetilde X=(\widetilde S,\j)$ induced by an equivariant alternating holomorphic curve $(\rho,f)$ descends to a Riemann surface structure $X=(S,\j)$. 

\begin{definition}\label{def eqiuiv alt hol}
	Two equivariant alternating holomorphic curves $(\rho_1,f_1)$ and $(\rho_2,f_2)$ are isomorphic if there exists $g\in\sG_2'$ and $\psi\in \Diff_0(S)$ such that 
	\[(\rho_1,f_1)=(\mathrm{Ad}_g\circ \rho_2, (g\cdot f)\circ \tilde\psi),\]
	where $\tilde \psi$ is the pullback of $\psi$ to $\widetilde S.$ 
\end{definition}
\begin{remark}
Since $\Diff_0(S)$ acts freely on the space of complex structures on $S$, two equivariant holomorphic curves $(\rho_1,f_1)$ and $(\rho_2,f_2)$ which induce the same Riemann surface $X$ are isomorphic if and only if there is $g\in\sG_2'$ such that 
	\[(\rho_1,f_1)=(\mathrm{Ad}_g\circ \rho_2,g\cdot f_2).\]
\end{remark}

For the discussion below, let $\cB$ be a holomorphic line bundle on $X$, and consider a $\sG_2'$-cyclic harmonic bundle $(\cE,\times,\Phi,h)$ of the form \eqref{eq:Higgsdiagram} with $\cL_{-3}=\cB$ and $\cL_{-1}=\cK,$ and assume $\beta$ is nonzero.  The Higgs bundle is written schematically as
\begin{equation}
		\label{eq cyclic G2 higgs B}\xymatrix{\cB  \ar@/_/[r]_{-\frac{\i}{ 2}} & \cB\cK^{-1}  \ar@/_/[r]_\beta & \cK \ar@/_/[r]_{1}  & \cO \ar@/_/[r]_{1} & \cK^{-1} \ar@/_/[r]_\beta  & \cB^{-1}\cK \ar@/_/[r]_{-\frac{\i}{ 2}} \ar@/_2pc/[lllll]_{\delta} & \cB^{-1}\ar@/_2pc/[lllll]_{\delta}}.
	\end{equation}
Here the map $1:\cO\to\cK^{-1}\otimes \cK$ is the identity and $-\frac{\i}{2}:\cB\to\cB\cK^{-1}\otimes \cK$ is determined by the product $\times.$ Namely, for a section $b$ of $\cB$ and $s_0=1\in H^0(\cO)$, we have
\[-\frac{\i}{2}(b)=-\frac{\i}{2}\cdot 1(s_0)\times b~.\]
\begin{proposition}\label{prop isos of B cyclic harmonic}
	Two $\sG_2$-cyclic harmonic bundles of the form \eqref{eq cyclic G2 higgs B} determined by $(\cB,\beta,\delta)$ and $(\cB',\beta',\delta')$ are isomorphic if and only if $(\cB,\beta,\delta)=(\cB',\lambda\beta,\lambda^{-2}\delta)$ for some $\lambda\in\sU(1).$ 
\end{proposition}

\begin{proof}
	Such an isomorphism is given by a diagonal unitary gauge transformation $g$ which has $\det(g)=1,$ preserves the orthogonal structures $Q_\cU$ and $Q_\cV$ and the product. As a result 
	\[g=\diag(\lambda,\lambda\mu^{-1},\mu^{-1},1,\mu,\lambda^{-1}\mu,\lambda^{-1}),\]
	where $\lambda,\mu\in\sU(1).$ Such a gauge transformation preserves the map $1:\cK\to\cO\otimes\cK$ if and only if $\mu=1.$ 
	Hence, such a gauge transformation acts on $(\beta,\delta)$ by $\beta\mapsto \lambda^{-1}\beta$ and $\delta\mapsto \lambda^2\delta.$
	\end{proof}

\begin{theorem}\label{thm 1-1 noncompact}
	There is a 1-1 correspondence between isomorphism classes of equivariant alternating holomorphic curves with induced Riemann surface $X$ and isomorphism classes of $\sG_2'$-cyclic harmonic bundles of the form \eqref{eq cyclic G2 higgs B} on $X$.
\end{theorem}

\begin{proof}
	One direction essentially follows from the description of the complex Frenet framing in Theorem \ref{thm hol descp of C FF}. 
As in \textsection \ref{ss cFrenet}, let $(E,\times,D)$ be the pullback of the trivial $\Im(\oct')$-bundle and connection on $\H$ by $f$ to $\widetilde X.$ Recall the Frenet framing of $f$ induces decomposition $E=\underline \R\oplus\T S\oplus \N S\oplus \B S$ which is orthogonal with respect to a positive definite metric $\g=1\oplus\g_\T\oplus(-\g_\N)\oplus\g_\B$.
By Theorem \ref{thm hol descp of C FF}, the complex Frenet framing of $f$ defines  a holomorphic structure $\cE$ on $E_\C$ which decomposes as
\[\cE=\cO\oplus\cK^{-1}\oplus\cK\oplus\cB^{-1}\cK\oplus\cB\cK^{-1}\oplus\cB\oplus\cB^{-1},\]
where we have written $\cK^{-1}$ instead of $\cT$ and $\cN=\cB^{-1}\cK$.
This splitting is orthogonal with respect to the induced hermitian metric $h$, and $\Phi=(D-\nabla_h)^{1,0}$ is holomorphic. In particular, $(\cE,\times,\Phi,h)$ defines a $\sG_2'$-cyclic harmonic bundle. Rearranging the summands of the splitting, the cyclic Higgs bundle on $\widetilde X$ has the desired form. Everything descends to $X$ equivariance. 
By \eqref{eq auto of C frentet} and Proposition \ref{prop isos of B cyclic harmonic}, two equivariant holomorphic curves which induce the Riemann surface $X$ are isomorphic if and only if the associated $\sG_2'$-cyclic harmonic bundles are isomorphic.

We now construct an isomorphism class of equivariant alternating holomorphic curves from a $\sG_2'$-cyclic harmonic bundle of the form \eqref{eq cyclic G2 higgs B}. By Proposition \ref{prop flat conn decomp}, the bundle $E_D\subset\cE$ decomposes as $\ell\oplus U_1\oplus V_1\oplus U_2.$ Where $\ell\subset \cO$ is a negative definite line subbundle. Moreover, $\ell$ is the span of the section $s_0=1:X\to\cO\subset \cE$, and so $s_0(x)$ is a vector in $E_D$ of norm $-1.$ 

The pullback $\widetilde E_D$ of $E_D$ to the universal covering is trivialized by parallel transport by $D.$ In this trivialization, the pullback of $s_0$ defines a $\rho$-equivariant map to the space of norm $-1$-vectors in the fiber over the base point $\widetilde E_{\tilde x_0}$. 
Choosing an identification of $(\widetilde E_{\tilde x_0},\times)\cong\Im(\oct')$ defines an equivariant curve $f:\widetilde X\to\H.$ By construction, the complex Frenet framing of $f$ the decomposition of the flat connection $D$ in the splitting $\cO\oplus \cK^{-1}\oplus\cK\oplus \cB^{-1}\cK\oplus \cB\cK\oplus\cB\oplus\cB^{-1}.$ In particular, $f$ is an alternating holomorphic curve. Changing the identification with $\Im(\oct')$ gives isomorphic equivariant alternating holomorphic curves.
\end{proof}

\section{Moduli spaces for compact surfaces }\label{sec higgs bundles and and hol curves}\label{s: moduli}
Fix a closed surface  $\Sigma$ with genus $g\geq 2$, and let $X$ denote a Riemann surface structure on $\Sigma.$ In this section we prove Theorem \ref{thm: moduli fixed Riemann} which gives a holomorphic description of the $\sG_2'$-Higgs bundles which arise from equivariant holomorphic curves on $\Sigma$ with induced Riemann surface $X$. In particular, the associated moduli space has many connected components.  
 Then, in Theorem \ref{thm complex anal moduli}, we describe the moduli space of equivariant alternating holomorphic curves on $\Sigma$, where the induced Riemann surface is allowed to vary, as a complex analytic space with a surjective holomorphic map to the Teichm\"uller space of $\Sigma$ and with a holomorphic action of the mapping class group; here the fibers are described by Theorem \ref{thm: moduli fixed Riemann}. 
 Finally, we describe how some of the connected components correspond to Hitchin representations for the groups $\sG_2'$ and $\sSL(3,\R)$, and discuss the totally geodesic case in \textsection\ref{ss:TotallyGeodesic}.
We start by recalling the moduli space of Higgs bundles on a compact Riemann surface and the nonabelian Hodge correspondence. 

\subsection{Moduli spaces for fixed Riemann surface}\label{ss:ModuliHB} 


To form the moduli space of Higgs bundles on $X$, we need the notion of stability.  An $\sSL(n,\C)$-Higgs bundle $(\cE,\Phi)$ on $X$ is called 
\begin{itemize}
	\item \emph{semistable} if, for every proper holomorphic subbundle $\cF\subset\cE$ such that $\Phi(\cF)\subset\cF\otimes \cK$, we have $\deg(\cF)\leq0,$
	\item \emph{stable} if it is the above inequality is always strict, and
	\item {\em polystable} if it is a direct sum of stable Higgs bundles of degree $0$, that is, $\cE=\cE_1\oplus\dots\oplus\cE_k$ and, for all $i$, $\deg(\cE_i)=0$, $\Phi(\cE_i)\subset\cE_i\otimes \cK$ and $(\cE_i,\Phi|_{\cE_i})$ is a stable Higgs bundle.  
\end{itemize}
The following theorem was proven by Nitsure \cite{NitsureHiggs} and Simpson \cite{SimpsonModuli1}. 
\begin{theorem}
	There is a quasi-projective variety $\cM(X,\sSL(n,\C))$, called the moduli space of $\sSL(n,\C)$-Higgs bundles, whose points parametrize isomorphism classes of polystable $\sSL(n,\C)$-Higgs bundles on a compact Riemann $X$.
\end{theorem}

There are appropriate notions of stability for $\sG$-Higgs bundles and associated moduli spaces $\cM(X,\sG)$, see \cite{SimpsonModuli1,HiggsPairsSTABILITY}. When $\sG$ is complex reductive, the moduli space $\cM(X,\sG)$ is quasi-projective.

Since (poly)stability is preserved by scaling the Higgs field by $\lambda\in\C^*$, there is a natural algebraic $\C^*$-action on the moduli space $\cM(X,\sG)$. The fixed point set of any subgroup of $\C^*$ thus defines an algebraic subvariety.


The following theorem relates the notion of stability to the existence of a harmonic metric. It was proven by Hitchin \cite{selfduality} for $n=2$ and Simpson \cite{SimpsonVHS} in general.
\begin{theorem}\label{thm NAH}
 	Let $X$ be a compact Riemann surface of genus $g\geq 2.$ An $\sSL(n,\C)$-Higgs bundle $(\cE,\Phi)$ on $X$ is polystable if and only if there is a hermitian metric $h$ on $\cE$ solving the Hitchin equations \eqref{eq Hitchin eq}. In particular,  $(\cE,\Phi,h)$ is an $\sSL(n,\C)$-harmonic bundle and $\sSL(n,\C)$-harmonic bundles define polystable Higgs bundles.
 \end{theorem}

 For $\sG$-Higgs bundles on compact Riemann surfaces, a similar theorem holds. Namely, if $(\cP,\Phi)$ is a $\sG$-Higgs bundle, then it is polystable if and only if there is a metric $h$ such that $(\cP,\Phi,h)$ is a $\sG$-harmonic bundle. This correspondence is usually called the {\em Kobayashi--Hitchin correspondence}. For $\sSL(n,\C)$-Higgs bundles which have extra structures related to being a $\sG$-Higgs bundle, the metric solving the Hitchin equations is compatible with these structures, see \cite{HiggsPairsSTABILITY} for details.

 For $k$-cyclic Higgs bundles on compact Riemann surfaces, we have the following result of Simpson which implies that polystable $k$-cyclic Higgs bundles on a compact Riemann surface are automatically $k$-cyclic harmonic bundles. 
\begin{proposition}\cite{KatzMiddleInvCyclicHiggs}\label{prop cyclic stab and metric}
	Let $(\cE,\Phi)$ be a $k$-cyclic Higgs bundle with $\cE=\cE_1\oplus\cdots\oplus\cE_k$. Then $(\cE,\Phi)$ is stable if and only if for all proper holomorphic subbundles $\cF=\cF_1\oplus\cdots\oplus\cF_k\subset\cE$ with $\cF_i\subset\cE_i$ and $\Phi(\cF)\subset\cF\otimes\cK$, we have $\deg(\cF)<0.$ Moreover, when $(\cE,\Phi)$ is stable, the splitting $\cE_1\oplus\cdots\oplus\cE_k$ is orthogonal with respect to the hermitian metric $h$ solving the Hitchin equations \eqref{eq Hitchin eq}.
\end{proposition}

A representation $\rho:\pi_1(\S)\to\sG$ is called {\em reductive} if post composing $\rho$ with the adjoint representation of $\sG$ decomposes as a direct sum of irreducible representations. The moduli space of $\sG$-conjugacy classes of representations $\pi_1(\S)\to\sG$ is called the {\em character variety} and denoted by 
\[\cX(\S,\sG)=\Hom^{\mathrm{red}}(\pi_1(\S),\sG)/\sG~.\]

Corlette's theorem \cite{canonicalmetrics} (proven by Donaldson \cite{harmoicmetric} for $\sSL(2,\C)$) asserts that given a representation $\rho:\pi_1(\S)\to\sG$, there exists a $\rho$-equivariant harmonic map $h_\rho:\widetilde X\to\sG/\sK$ if and only if $\rho$ is reductive. 
Combining Corlette's theorem with the Kobayashi--Hitchin correspondence defines a real analytic isomorphism between the moduli space  of $\sG$-Higgs bundles on $X$, the moduli space of $\sG$-harmonic bundles on $X$ and the $\sG$-character variety. This is called the {\em nonabelian Hodge correspondence}.

\subsection{Moduli spaces of holomorphic curves for compact Riemann surfaces}
For a compact Riemann surface $X,$ denote the set of isomorphism classes of  equivariant alternating holomorphic curves $(\rho,f)$ with induced Riemann surface $X$ by
\[\cH(X)=\{(\rho,f)~\text{equiv. alt. holomorphic curves with induced complex structure }X\}/\sG_2'.\]
By the nonabelian Hodge correspondence and Theorem \ref{thm 1-1 noncompact}, $\cH(X)$ is in 1-1 correspondence with the isomorphism classes of polystable $\sG_2'$-Higgs bundles on $X$ of the form  
\begin{equation}
		\label{eq cyclic G2 higgs B compact}\xymatrix{\cB  \ar@/_/[r]_{-\frac{\i}{ 2}} & \cB\cK^{-1}  \ar@/_/[r]_\beta & \cK \ar@/_/[r]_{1}  & \cO \ar@/_/[r]_{1} & \cK^{-1} \ar@/_/[r]_\beta  & \cB^{-1}\cK \ar@/_/[r]_{-\frac{\i}{ 2}} \ar@/_2pc/[lllll]_{\delta} & \cB^{-1}\ar@/_2pc/[lllll]_{\delta}},
	\end{equation}
where $\cB$ is a holomorphic line bundle, $\beta\in H^0(\cB^{-1}\cK^3)\setminus \{0\},$  $1:\cO\to\cK^{-1}\otimes \cK$ is the identity, and $-\frac{\i}{2}:\cB\to\cB\cK^{-1}\otimes \cK$ is defined by $-\frac{\i}{2}(b)=\frac{\i}{2}\cdot 1(s_0)\times b$, where $s_0=1\in H^0(\cO).$ 
This correspondence gives  $\cH(X)$ a complex analytic structure. 

The degree of the line bundle $\cB$ defines a continuous map $\deg:\cH(X)\to\Z$. In particular, 
\[\cH(X)=\coprod\limits_{d\in\Z}\cH_d(X),\]
where $\cH_d(X)=\deg^{-1}(d).$ The spaces $\cH_d(X)$ parameterized as follows.

\begin{theorem}
	\label{thm: moduli fixed Riemann} 
	Let $X$ be a compact Riemann surface of genus $g\geq 2$. 
	\begin{enumerate}
		\item If $g\leq d\leq 6g-6$, then $\cH_d(X)$ is biholomorphic to a rank $(2d-g+1)$ holomorphic vector bundle over the $(6g-6-d)$-symmetric product of $X$.
		\item If $0\leq d\leq g-1,$ then $\cH_d(X)$ is biholomorphic to bundle over a $H^1(X,\Z_2)$-cover of the $2d$-symmetric product of $X$ whose fiber is $(\C^{5g-5-d}\setminus\{0\})/\pm\Id$. 
		\item If $d\notin\{0,\dots,6g-6\}$, then $\cH_d(X)$ is empty.
	\end{enumerate}
\end{theorem}
\begin{remark}
	The upper bound on $d$ follows immediately from the assumption that $\beta\in H^0(\cB^{-1}\cK^3)\setminus\{0\}$ while the lower bound comes from polystability. To prove first two points, we parameterize the isomorphism classes of $\sG_2'$-Higgs bundles of the form \eqref{eq cyclic G2 higgs B compact} which are polystable. The proof gives explicit descriptions of the fibrations over the  appropriate symmetric product.
\end{remark}

The following proposition describes when the associated $\sSL(7,\C)$-Higgs bundle is polystable.
\begin{proposition}\label{prop:HiggsBundlesAreStable}
For a $\sG_2'$-Higgs bundle of the form \eqref{eq cyclic G2 higgs B compact}, let $d=\deg(\cB).$  Then $0\leq d\leq 6g-6$, and,
\begin{itemize} 
	\item for $g\leq d\leq 6g-6$, the associated $\sSL(7,\C)$-Higgs bundle is stable, 
	\item for $0<d\leq g-1$, the associated $\sSL(7,\C)$-Higgs bundle is polystable if and only if $\delta\neq0$, in which case it is stable,
	\item for $d=0$, the associated $\sSL(7,\C)$-Higgs bundle is polystable if and only if $\delta\neq0,$ in which case it is strictly polystable.
\end{itemize}
\end{proposition}

\begin{proof}
	 Recall that the $\sSL(7,\C)$-Higgs bundles are $6$-cyclic with cyclic splitting 
	\[(\cB\oplus \cB^{-1})\oplus \cB\cK^{-1}\oplus\cK\oplus\cO\oplus\cK^{-1}\oplus\cB^{-1}\cK.\]
	 By Proposition \ref{prop cyclic stab and metric}, to test stability of the $\sSL(7,\C)$-Higgs bundle it suffices to consider invariant subbbundles compatible with the cyclic splitting. 

	 When $\delta=0,$ $\cB^{-1}\oplus\cB^{-1}\cK$ is such an invariant bundle. If $d<g-1$, then the degree of $\cB^{-1}\oplus\cB^{-1}\cK$ is positive and invariant so the underlying Higgs bundle is unstable, while for $d=g-1$, $\cB^{-1}\oplus\cB^{-1}\cK$ is a degree zero invariant subbundle which does not have an invariant complement, so the $\sSL(7,\C)$-Higgs bundle is not polystable. Hence, the Higgs bundles with $\delta=0$ are polystable if and only if $d>g-1$, in which case they are stable. 

	 When $\delta\neq0$, the only invariant subbundle is the degree $-d$ kernel of 
	\begin{equation}
		\label{eq kernel d=0}
\begin{pmatrix}
	 	-\frac{\i}{2}\\\delta
	 \end{pmatrix}:\cB\oplus\cB^{-1}\to(\cB\cK^{-1})\otimes \cK~.\end{equation}
	 Hence, when $0<d\leq 6g-6$ and $\delta\neq 0$, the associated $\sSL(7,\C)$-Higgs bundle is stable. 

	 Finally, when $d=0$, $\delta\in H^0(\cB^2)\setminus\{0\}$ defines an isomorphism $\cB\cong\cB^{-1}.$ Denote the kernel of \eqref{eq kernel d=0} by $\cI$. Since $\delta\neq 0,$ $\cI\subset\cB\oplus\cB^{-1}$ is an orthogonal subbundle which is isomorphic to $\cB$. Hence, taking the orthogonal complement, we obtain a new splitting $\cB\oplus\cB^{-1}\cong \cI\oplus\cI^\perp$. 
	 In this splitting the $\sSL(7,\C)$-Higgs bundle is given by $(\cE,\Phi)=(\cI\oplus\cE',0\oplus \Phi')$, where $(\cE',\Phi')$ is the $6$-cyclic Higgs bundle 
	 \begin{equation}
	 	\label{eq polystable d0}(\cE',\Phi')=\xymatrix{ \cB\cK^{-1}  \ar@/_/[r]_\beta & \cK \ar@/_/[r]_{1}  & \cO \ar@/_/[r]_{1} & \cK^{-1} \ar@/_/[r]_\beta  & \cB^{-1}\cK \ar@/_/[r]_{\epsilon}  & \cI^\perp\ar@/_2pc/[lllll]_{\epsilon}}~,
	 \end{equation}
	 with $\epsilon$ is the restriction of \eqref{eq kernel d=0} to $\cI^\perp.$ 
	Since $\epsilon\neq0$, the cyclic Higgs bundle $(\cE',\Phi')$ is stable, and hence $(\cE,\Phi)$ is strictly polystable. 
\end{proof}

\begin{remark}\label{rem beta=0}
	Note that in the case $0\leq d\leq g-1,$ the $\sSL(7,\C)$-Higgs bundle associated to the above $\sG_2'$-Higgs bundle is polystable when $\beta=0$. Such Higgs bundles define equivariant holomorphic curve $f:\tilde X\to \H$ which are totally geodesic.
\end{remark}
\begin{proof}[Proof of Theorem \ref{thm: moduli fixed Riemann}]Denote the set of degree $n$ line bundles on $X$ by $\Pic^n(X)$ and the $n^{th}$ symmetric product by $\sSym^n(X).$ Taking the divisor $\mathrm{div}(\sigma)$ defines a biholomorphism between $\sSym^n(X)$ and the space of pairs $(\cN,[\sigma])$, where $\cN\in\Pic^n(X)$ and $[\sigma]\in\mathbf{P}(H^0(\cN))$.  

First consider the case when $d=\deg(\cB)$ satisfies $g\leq d\leq 6g-6$. The proof is similar to Hitchin's parameterization of the components of $\cM(X,\sPSL(2,\R))$ with nonzero Euler number in \cite{selfduality}. By Proposition \ref{prop isos of B cyclic harmonic} and Proposition \ref{prop:HiggsBundlesAreStable}, we have 
\[\cH_d(X)=\{(\cB,\beta,\delta)\}/\sim~,\]
where $\cB\in\Pic^d(X),$ $\beta\in H^0(\cB^{-1}\cK^3)\setminus\{0\}$ and $\delta\in H^0(\cB^2)$, and $(\cB,\beta,\delta)\sim(\cB,\lambda\beta,\lambda^{-2}\delta)$ for $\lambda\in\C^*.$ Define the bundle $\pi:\cY\to\sSym^{6g-6-d}(X)$ by $\pi^{-1}(D)=H^0(\cK^6(-2D))$. Since $d\geq g$, Riemann-Roch and Serre duality imply $\dim(\pi^{-1}(D))=2d-(g-1)$.
Hence, $\cY\to \sSym^{6g-6-d}(X)$ is a rank $2d-(g-1)$ vector bundle. 
There is a well defined injective map 
\[\Psi:\xymatrix@R=0em{\cH_d(X)\ar[r]&\cY\\[(\cB,\beta,\delta)]\ar@{|->}[r]&\left(\mathrm{div}(\beta),\beta^2\delta\right)}~.\] 
Moreover, the map $\Psi$ is surjective. Indeed, a point $(D,\theta) \in \cY$ determines $(\cL,[\beta])$, where $\beta\in \P(H^0(\cL))$ with $\mathrm{div}(\beta)=D$ and $\theta\in H^0(\cK^6(-2D)).$ Consider $\cB=\cL^{-1}\cK^3\in\Pic^d(X)$, since $\theta$ is degree six holomorphic differential which vanishes on $2D$, it can be written as $\theta=\beta^2\delta$ for $\delta\in H^0(\cB^2).$ 
The result follows since holomorphic bijections are biholomorphisms.

Now assume $0\leq d\leq g-1.$ By Proposition \ref{prop isos of B cyclic harmonic} and Proposition \ref{prop:HiggsBundlesAreStable}, we have 
\[\cH_d(X)=\{(\cB,\beta,\delta)\}/\sim~,\]
where $\cB\in\Pic^d(X),$ $\beta\in H^0(\cB^{-1}\cK^3)\setminus\{0\}$, $\delta\in H^0(\cB^2)\setminus\{0\}$, and $(\cB,\beta,\delta)\sim(\cB,\lambda\beta,\lambda^{-2}\delta)$ for $\lambda\in\C^*.$ Note that $\dim(H^0(\cB^{-1}\cK^3)=5g-5-d.$

Recall that the set of square roots of an even degree line bundle over $X$ is a $H^1(S,\Z_2)$-torsor. The $2d$-symmetric product $\sSym^{2d}(X)$ thus has a $H^1(X,\Z_2)$-cover $\widehat\sSym^{2d}(X)$ which parameterizes pairs $(\cB,[\delta])$, where $\cB\in\Pic^d(X)$  and $[\delta]\in \mathbf{P}(H^0(\cB^{2})\setminus\{0\}).$
There is a well defined injective map
\[\Theta: \xymatrix@R=0em{\cH_d(X)\ar[r]& \widehat\sSym^{2d}(X)\times H^0(\cK^6)\\[(\cB,\beta,\delta)]\ar@{|->}[r]&\left((\cB,[\delta]),\beta^2\delta\right)}\]
Let $\cW$ denote the image of $\Theta$. Projection on to the first factor defines a surjective map  $\pi:\cW\to\widehat\sSym^{2d}(X).$ The fiber $\pi^{-1}((\cB,[\delta]))$ consists of all nonzero holomorphic sections of $\cK^6$ which are squares of sections $\cB^{-1}\cK^3$ and vanish on $\mathrm{div}(\delta)$. In other words, the fiber is the image of the quadratic map $H^0(\cB^{-1}\cK^3)\setminus\{0\}\to H^0(\cB^{-2}\cK^6)\subset H^0(\cK^6)$ defined by $\beta\mapsto \beta^2.$ In particular, it is biholomorphic to $(\C^{5g-5-d}\setminus\{0\})/\pm\Id.$
\end{proof}
\begin{remark}
By Remark \ref{rem beta=0}, when $0\leq d\leq g-1$ and $\beta=0$, the resulting Higgs bundles are also polystable. By the above proof, the space of such Higgs bundles is biholomorphic the quotient of a rank $(5g-5-d)$ vector bundle over $\widehat\sSym^{2d}(X)$ by the fiberwise action of $\pm\Id.$
\end{remark}


\subsection{Varying the Riemann surface}
The moduli space $\cH(\Sigma)$ of equivariant alternating holomorphic curves on $\Sigma$ is defined by
\[\cH(\Sigma)=\{(\rho,f)~\text{equivariant alternating holomorphic curve }\}/\sim~,\] 
where  $(f_1,\rho_1)\sim (f_2,\rho_2)$ if there is $g\in\sG_2'$ and $f\in \Diff_0(\Sigma)$ such that $(\rho_1,f_1)=(\mathrm{Ad}_g\circ \rho_2,g\cdot f_2).$ 

Recall that the {\em Teichm\"uller space} $\cT(\Sigma)$ of $\Sigma$ is the quotient of the space of complex structures on $\Sigma$ which are compatible with the orientation by the group $\Diff_0(\Sigma).$ 
Taking the induced Riemann surface of an equivariant holomorphic curve defines a surjective map 
\[\pi:\xymatrix@R=0em{\cH(\Sigma)\ar[r]&\cT(\Sigma)~.\\[(\rho,f)]\ar@{|->}[r]&[(S,\j)]}\]

Recall also that the {\em mapping class group} $\Mod(\Sigma)$ of $\Sigma$ is the quotient $\Mod(\Sigma)=\Diff^+(\Sigma)/\Diff_0(\Sigma).$ The mapping class group acts naturally on $\cH(\Sigma)$ and $\cT(\Sigma)$, and the map $\pi$ is equivariant with respect to these actions.

Finally, recall that a \emph{complex analytic set} is a subset $V\subset\C^n$ which is locally the zero set of a finite number of holomorphic functions. A holomorphic function on an open set $U\subset V$ is the restriction of a holomorphic function defined on a neighborhood of $U$ in $\C^n$. This defines a structure sheaf $\mathcal O_V$ on $V$ and we call a \emph{complex analytic space}  a ringed space $(A,\mathcal O_A)$ which is locally isomorphic to an analytic set equipped with its structure sheaf. A morphism of ringed space between complex analytic space is called a \emph{holomorphic map}.

\begin{theorem}
	\label{thm complex anal moduli} The moduli space $\cH(\Sigma)$ has the structure of a complex analytic space. With respect to this structure, the mapping class group $\Mod(\Sigma)$ acts holomorphically and the projection to Teichm\"uller space $\pi:\cH(\Sigma)\to\cT(\Sigma)$ is a surjective holomorphic map. 
	Moreover, 
	\[\cH(\Sigma)=\coprod_{d\in\{0,\dots,6g-6\}}\cH_d(\Sigma)~,\]
	where, for each $X\in\cT(\Sigma)$,  $\cH_d(\Sigma)\cap\pi^{-1}(X)$ is biholomorphic to the space $\cH_d(X)$ from Theorem \ref{thm: moduli fixed Riemann}.
\end{theorem}

\begin{proof}
Recall now that from Section \ref{s:HBandHolCurves} provides a one-to-one correspondence between points in $\HH(\S)$ and pair of points $(X,(\cE,\Phi))$ where $X\in\TT(\Sigma)$ and $ (\cE,\Phi) \in \cM(\S,\GC)$ is a stable $\GC$-Higgs bundles of the form of equation (\ref{eq cyclic G2 higgs B}).

Consider the space
\[\cM(\Sigma,\GC)= \coprod_{X\in \TT(\Sigma)} \cM(X,\GC)~,\]
where $\cM(X,\GC)$ is the moduli space of poly-stable $\GC$-Higgs bundles over $X$. Since $\GC$ is a complex reductive algebraic group $\sG$, building on the construction of Simpson \cite{SimpsonModuli2}, it was proven in \cite[Theorem 7.5]{SO23LabourieConj} that $\cM(\S,\GC)$ has a $\Mod(\S)$-invariant complex analytic structure such that the natural projection map  $\pi: \cM(\Sigma,\GC) \to \cT(\Sigma)$ is holomorphic and $\pi^{-1}(X)$ is biholomorphic to $\cM(X,\GC)$.

The $\C^*$-action on each $\cM(X,\GC)$ yields a $\C^*$ action on $\cM(\S,\GC)$ which is analytic by \cite{SimpsonModuli2}. Given a primitive $6^{th}$ root of unity $\zeta$, the set of fixed points $\cM(\Sigma,\GC)^\zeta$ is thus analytic. 
To conclude the proof, observe that being of the form \eqref{eq cyclic G2 higgs B} is an open condition in the space of $6$-cyclic $\GC$-Higgs bundles. Since open subsets of analytic spaces are analytic, the result follows.
\end{proof}

\begin{remark}
Simpson explained to us a proof that the restriction of the projection $\pi: \cM(\Sigma,\GC) \to \TT(\S)$ to stable Higgs bundles is a flat family. This implies in particular that the total space $\cM(\Sigma,\GC)$, and hence $\cH(\Sigma)$, is smooth. Since the technical aspects of the proof are not inline with the main points of this paper, we decided not to include it here. Smoothness will be explicitly addressed in \cite{CTcyclic}.
\end{remark}

\subsection{The case $d=6g-6$}
Recall that a \emph{Fuchsian representation} is a morphism $\rho:\pi_1(\Sigma)\to\sPSL(2,\R)$ arising as the holonomy of a hyperbolic structure on $\S$ (compatible with the orientation). The space of Fuchsian representations defines a connected component of the character variety $\cX(\S,\sPSL(2,\R))$ which is naturally identified with the Teichm\"uller space $\cT(\Sigma).$ 

For $\sG_\C$ a complex semisimple Lie group of adjoint type, there is a special embedding $\sPSL(2,\C)\to\sG_\C$ called the principal embedding. The restriction of the principal embedding to $\sPSL(2,\R)$ defines an embedding into the split real form $\sG$ of $\sG_\C.$ 
Hence, postcomposing Fuchsian representations with the principal embedding define representations $\rho:\pi_1(\Sigma)\to\sG$. 
In \cite{liegroupsteichmuller}, Hitchin used Higgs bundles on a fixed Riemann surface to parameterize the connected component $\Hit(\S,\sG)\subset\cX(\S,\sG)$ containing such representations by a vector space of holomorphic differentials. As a result, $\Hit(\S,\sG)$ is called the \emph{Hitchin component} and representations in it are called $\sG$-\emph{Hitchin representations}. 

Denote the vector bundle of degree $6$ holomorphic differentials by $p:\mathcal{Q}_6(\Sigma)\to \cT(\Sigma)$, i.e, $p^{-1}(X)\cong H^0(\cK^6_X).$ For the case $d=6g-6,$ we have the following description of $\cH_{6g-6}(\Sigma).$
\begin{theorem}
The space $\pi:\cH_{6g-6}(\Sigma)\to \cT(\Sigma)$ is biholomorphic to the vector bundle $p: \mathcal{Q}_6\to \cT(\Sigma)$. Moreover, the points of $\cH_{6g-6}(\Sigma)$ corresponds to isomorphism classes of alternating holomorphic curves $(\rho,f)$ with $\rho$ a $\sG_2'$-Hitchin representation.
\end{theorem}

\begin{proof}
For a fixed Riemann surface $X$, the space $\cH_{6g-6}(X)$ is parameterized by isomorphism classes of $\sG_2'$-Higgs bundles of the form \eqref{eq cyclic G2 higgs B compact} with $\deg(\cB)=6g-6.$ 
Such Higgs bundles are determined by triple $(\cB,\beta,\delta)$, where $\cB\in\Pic^{6g-6}(X)$, $\beta\in H^0(\cB^{-1}\cK^3)\setminus\{0\}$ and $\delta\in H^0(\cB^2).$  Since $\deg(\cB)=6g-6$ and $\beta\neq 0$, we have $\cB=\cK^3$, and hence $\delta\in H^0(\cK^6).$ Furthermore, by Theorem \ref{thm: moduli fixed Riemann}, fixing $\beta=1\in H^0(\cB^{-1}\cK^3)$ determines the isomorphism class of such a Higgs bundles. Under the nonabelian Hodge correspondence, such Higgs bundles define $\sG_2'$-Hitchin representations, see \cite{liegroupsteichmuller}. We conclude the proof by noting that the fiber of $\pi:\cH_{6g-6}(\Sigma)\to \cT(\Sigma)$ is defined by the holomorphic section $\delta\in H^0(\cK^6)$ by Theorem \ref{thm: moduli fixed Riemann}.
\end{proof}

\begin{remark}
	In \cite{cyclicSurfacesRank2}, Labourie proved that the natural map $\cH_{6g-6}(\Sigma)\to \Hit(\S,\sG_2')$ is a diffeomorphism. A dimension count implies that $\cH_d(\Sigma)\to \cX(\S,\sG_2')$ maps onto a connected component of $\cX(\S,\sG_2')$ if and only if $d=6g-6$. In \textsection \ref{s cyclic surfaces} we will discuss the local structure of the map $\cH_d(\Sigma)\to \cX(\S,\sG_2')$ for general $d$.
\end{remark}

\subsection{The case $d=0$} \label{sec d=0}Recall that elements $\omega\in H^1(\Sigma,\Z_2)$ correspond to representations $\rho_\omega:\pi_1(\Sigma)\to\Z_2$, and also to holomorphic line bundles $\cI_\omega$ on a Riemann surface such that $\cI_\omega^2=\cO$. 
Given $\omega$ and a Riemann surface $X,$ an \emph{$\omega$-twisted holomorphic cubic differential} is a holomorphic section of $\cK_X^3\otimes \cI_\omega$. We denote the vector bundle of $\omega$-twisted holomorphic cubic differentials over Teichm\"uller space by $\mathcal{Q}_3^\omega$.
For the case $d=0,$ we have the following description of $\cH_{0}(\Sigma).$
\begin{lemma}
The space $\pi:\cH_{0}(\Sigma)\to \cT(\Sigma)$ decomposes as
\[\cH_{0}(\Sigma)=\coprod_{\omega\in H^1(\Sigma,\Z_2)}\cH_{0}^\omega(\Sigma),\]
where $\cH_{0}^\omega(\Sigma)$ is biholomorphic to the quotient of the complement of the zero section of $\mathcal{Q}_3^\omega$ by $\pm\Id$.  
\end{lemma}
\begin{proof}
For a fixed Riemann surface $X$, the space $\cH_{0}(X)$ is parameterized by isomorphism classes of $\sG_2'$-Higgs bundles of the form \eqref{eq cyclic G2 higgs B compact} with $\deg(\cB)=0.$ 
Such Higgs bundles are determined by triple $(\cB,\beta,\delta)$, where $\cB\in\Pic^{0}(X)$, $\beta\in H^0(\cB^{-1}\cK^3)\setminus\{0\}$ and $\delta\in H^0(\cB^2)\setminus\{0\}.$ Since $\delta\neq 0,$ $\cB^2\cong\cO_X$. Hence $H^1(\Sigma,\Z_2)$ parameterizes the choices for $\cB$, and this gives the decomposition into components $\cH_0^\omega(\Sigma)$. Note that $\beta$ is an $\omega$-twisted holomorphic cubic differential since $\beta\in H^0(\cB^{-1}\cK^3)$ and $\cB^{-1}\cong\cB.$ 
By the proof of Theorem \ref{thm: moduli fixed Riemann}, fixing $\delta\in H^0(\cB^2)$ determines the isomorphism class of the Higgs bundle up to $\beta\mapsto -\beta$. 
Hence, $\cH_0^\omega(\Sigma)$ is the biholomorphic to the quotient of the complement of the zero section of $\mathcal{Q}_3^\omega$ by $\pm\Id.$
\end{proof}
We now describe the equivariant holomorphic curves in $\cH_0^\omega(\Sigma)$ in more detail. 
Recall from \textsection\ref{ss:pseudosphere} that $\SS$ is the space of vectors in $\R^{4,3}$ of norm 1 and $\P_+(\R^{4,3})$ is the space of positive definite lines in $\R^{4,3}$. The group $\sG_2'$ acts transitively on both $\SS$ and $\P_+(\R^{4,3})$, and the stabilizers of a point are respectively isomorphic to $\sSL(3,\R)$ and  the nonsplit extension 
\[1\to\sSL(3,\R)\to\widehat\sSL(3,\R)\xrightarrow{\varphi} \Z_2\to1~.\]
Since taking orthogonal complements defines a bijection between points in $\P_+(\R^{4,3})$ and totally geodesic copies of $\mathbf{H}^{3,2}$ in $\H$, the $\sG_2'$-stabilizer of an $\mathbf{H}^{3,2}\subset\H$ is isomorphic to $\widehat\sSL(3,\R)$. 
Finally, note that given a representation $\hat\rho:\pi_1(\Sigma)\to \widehat\sSL(3,\R)$, we get 
\begin{itemize}
	\item an induced $\Z_2$-representation $\rho_\omega=\varphi\circ\hat\rho:\pi_1(\Sigma)\to \Z_2$, and
	\item an $\sSL(3,\R)$-representation $\hat\rho:\ker(\rho_\omega)\to\sSL(3,\R).$
\end{itemize}

\begin{theorem}\label{thm H0 hol curve discription}
For $\omega\in H^1(\Sigma,\Z_2),$ the isomorphism class of an equivariant alternating holomorphic curve $[\rho,f]$ lies in $\cH_{0}^\omega(\Sigma)$ if and only if $f(\widetilde \Sigma)$ is contained in a unique totally geodesic copy of  $\mathbf{H}^{3,2}$ inside $\H$  and 
\[\rho=\iota \circ\hat\rho:\pi_1(\Sigma)\to\widehat\sSL(3,\R)\to\sG_2'~,\] where
$\iota:\widehat\sSL(3,\R)\to\sG_2'$ is given by the stabilizer of the unique $\mathbf{H}^{3,2}$ containing $f(\widetilde\Sigma)$, and
 $\hat \rho: \pi_1(\Sigma)\to \widehat\sSL(3,\R)$ has the property that the induced $\Z_2$-representation is $\rho_\omega$ and the associated $\sSL(3,\R)$ representation $\hat\rho:\ker(\rho_\omega)\to\sSL(3,\R)$ is an $\sSL(3,\R)$-Hitchin representation which is not Fuchsian.
\end{theorem}
\begin{remark}
Note that an equivariant holomorphic curve $[\rho,f]$ lies in a totally geodesic copy of $\mathbf{H}^{3,2}\subset\H$ if and only if the representation fixes a positive definite line in $\R^{4,3}$, and that this is equivalent to $\rho$ factoring through the stabilizer of a positive definite line.
\end{remark}
\begin{proof}
	First suppose $[\rho,f]\in\cH^\omega_0$. Let $X$ be the induced Riemann surface and $\cI$ the holomorphic line bundle on $X$ associated to $\omega.$ 
	As in the proof of Proposition \ref{prop:HiggsBundlesAreStable}, the $\sSL(7,\C)$-Higgs bundle associated to $[\rho,f]$ is strictly polystable and isomorphic to $(\cE,\Phi)\cong(\cI\oplus \cE',0\oplus\Phi'),$ where $\cI\cong\cB$ is an orthogonal subbundle of $\cB\oplus\cB^{-1}$ defined by the kernel of \eqref{eq kernel d=0} and $(\cE',\Phi')$ is given by \eqref{eq polystable d0}.
	Note that $(\cE',\Phi')$ defines an $\sO(3,3)$-Higgs bundle, and that $(\cE,\Phi)$ is defines an $\sS(\sO(1)\times \sO(3,3))$-Higgs bundle. 
	This implies the $\sSO_0(4,3)$ representations $\rho$ associated via the nonabelian Hodge correspondence fixes a positive definite line in $\R^{4,3}.$ 
	Hence, the image of $f$ lies in the totally geodesic copy of $\mathbf{H}^{3,2}\subset\H$ associated to the positive definite line bundle $\cI$ and $\rho$ factors through the stabilizer of a positive definite line in $\R^{4,3}.$

Since $\cB\cong\cI,$ the $\Z_2$-factor of the representation $\rho$ is given by the $\omega\in H^1(\Sigma,\Z_2)$ determined by $\cB.$ 
As $\cI$ and $\cI^\perp$ are both isomorphic to $\cB$, the Higgs bundle $(\cE',\Phi')$ from \eqref{eq polystable d0} can be written as 
\[(\cE',\Phi')=\xymatrix{ \cI\cK^{-1}  \ar@/_/[r]_\beta & \cK \ar@/_/[r]_{1}  & \cO \ar@/_/[r]_{1} & \cK^{-1} \ar@/_/[r]_\beta  & \cI\cK \ar@/_/[r]_{1}  & \cI\ar@/_2pc/[lllll]_{1}}~.\]

First consider the case $\cI=\cO.$ In this case the $\rho_\omega$ is the trivial representation and $\hat\rho$ is an $\sSL(3,\R)$-representation. 
Consider the following injective holomorphic bundle maps 
\[\theta_\pm:\xymatrix@R=0em{\cK\oplus\cO\oplus\cK^{-1}\ar[r]&\cE'\\(x,y,z)\ar@{|->}[r]&(\pm z,x,y,z,\pm x,\pm y)}~.\]
The images of $\theta_\pm$ define holomorphic subbundles $\mathcal\cF_\pm\subset\cE'$ such that $\cE'= \cF_+\oplus \cF_-$. 
The bundles $\F_\pm$ are both $\Phi'$-invariant since
\[\Phi(\theta_\pm(x,y,z))=\Phi(\pm z,x,y,z,\pm x,\pm y)=(\pm(1),\pm\beta(z),1(x),1(y),\beta(z),\pm 1(x))\]\[=\theta_\pm(\phi_\pm(x,y,z))~,\]
where 
\[\phi_\pm=\begin{pmatrix}
	0&0&\pm\beta\\1&0&0\\0&1&0
\end{pmatrix}\in \End(\cK\oplus \cO\oplus\cK^{-1})\otimes \cK.\]
Moreover, $(\cF_\pm,\phi_\pm)$ is the $\sSL(3,\R)$-Higgs bundles in the Hitchin component described by the cubic differential $\pm\beta,$ see \cite{liegroupsteichmuller}. 
Hence, when $\omega=0,$ the representation $\hat\rho:\pi_1(\Sigma)\to\widehat\sSL(3,\R)$-factors through an $\sSL(3,\R)$-Hitchin representation. As $\beta$ is assumed to be nonzero, $\hat\rho$ is not Fuchsian.

 When $\omega\neq0,$ the $\Z_2$-representation $\rho_\omega$ is nontrivial. However, the  pullback of the Higgs bundle to the connected double cover $X_\omega\to X$ defines a point in $\cH_0^0(X_\omega)$.
 Hence the restriction of $\hat\rho$ to the kernel of $\rho_\omega$ is an $\sSL(3,\R)$-Hitchin representation which is not Fuchsian. 

Conversely, let $[\rho,f]$ be an equivariant holomorphic curve such that $f(\widetilde\Sigma)$ lies in a totally geodesic subspace $\mathbf{H}^{3,2}\subset\H$. Let $\widetilde W$ be the normal bundle of $\mathbf{H}^{3,2}$ in $\H$. If we denote by $\widetilde M$ the orthogonal of $\T \widetilde \S \oplus \N \widetilde \S$ inside $f^*\T\mathbf{H}^{3,2}$, then $f^*\widetilde W$ and $\widetilde M$ descends respectively to real line bundles $W$ and $M$ on $\S$. Thus we have that $\B \Sigma = W\oplus M$ is topologically trivial, and so has degree $0$. Thus, $[\rho,f]$ defines a point in $\cH_0$. 
\end{proof}

\subsection{Link with affine spheres}

The theory of Hitchin representations in $\SL(3,\R)$ is intimately linked with the theory of \emph{affine spheres} in $\R^3$. Namely, Loftin \cite{flatmetriccubicdiff} and Labourie \cite{CubicDifferentialsRP2} proved that given a Hitchin representation $\rho: \pi_1(\S) \to \SL(3,\R)$, there exists a unique $\rho$-equivariant \emph{affine sphere} $u_\rho: \widetilde \S \to \R^3$. Such an affine sphere defines a metric $g$ on $\S$ (called the \emph{Blaschke metric}) and a holomorphic cubic differential $q_3\in H^0(X,\cK^3)$, called the \emph{Pick form}, where $X$ is the Riemann surface associated to $g$. This defines a map from $\Hit(\S,\sSL(3,\R))$ to $\mathcal Q_3^0\to\cT(\Sigma)$ which they proved is a homeomorphism.

Theorem \ref{thm H0 hol curve discription} implies that the component $\cH_0^0(\S)$ consists of equivariant holomorphic curves $[\rho,f]$ lying in a totally geodesic copy of $\mathbf H^{3,2}$ in $\H$ and such that the underlying representation comes from a non-Fuchsian Hitchin representation $\hat \rho: \pi_1(\S)\to\SL(3,\R)$. The next theorem describes a link between the holomorphic curve $f$ and the affine sphere $u_{\hat\rho}$.

\begin{theorem}
Let $(\rho,f)$ be an alternating equivariant holomorphic curve whose equivalence class lies in $\cH_0^0$, $z$ be a point in $\mathbf S^{3,3}$ fixed by $\rho$, $\mathcal D_z^+$ be isotropic $3$-plane defined in Subsection \ref{ss:pseudosphere} and denote by $\hat \rho$ the corresponding Hitchin representation into $\text{Fix}_{\G}(x) \cong \SL(3,\R)$. Then the map
\[\begin{array}{llll}
u_z: & \widetilde \S & \longrightarrow & \mathcal D^+_z \\
& x & \longmapsto & f(x) + z\cdotp f(x) \end{array}\]
is a $\hat \rho$-equivariant affine sphere.
\end{theorem}

\begin{proof}
Since $f$ takes value in $z^\bot$, one easily checks that $u_z$ as well. Moreover, we have
\[\psi_z(u_z(x)) = z\cdotp u_z(x)= z \cdotp f(x) + z^2\cdotp f(x) = u_z(x)~,\]
so $u_z$ takes value in $\Ker(\Id-\psi_z) = \mathcal D^+_z$.

To see that $u_z$ is an affine sphere, observe that $u_z$ is obtained by considering the section $1$ of $\OO$ in $\cF_+$ which corresponds to an affine sphere by Baraglia \cite[Section 3.4.2]{g2geometry}
\end{proof}

\begin{remark}
Note that the representation $\rho$ fixes two points $z$ and $-z$ in $\mathbf S^{3,3}$. Taking $-z$ instead of $z$ is equivalent to considering the map $(f- z\cdotp f)$ which takes value in $\mathcal D_z^-$. The underlying representation is the contragradient of $\hat \rho$.
\end{remark}

\subsection{Totally geodesic surfaces}\label{ss:TotallyGeodesic}
In the correspondence between alternating holomorphic curves and $\G$-harmonic bundles of the form of equation \eqref{eq cyclic G2 higgs B compact}, the section $\beta$ is the $(1,0)$-part of the second fundamental form of the holomorphic curve. As alternating holomorphic curves are assumed to be not totally geodesic, $\beta$ is nonzero. 

In the equivariant setting, stability of the Higgs bundle  implies the section $\delta$ is not $0$ when $0\leq \deg(\cB)\leq g-1$, see  Proposition \ref{prop:HiggsBundlesAreStable}. On the other hand, setting $\beta=0$ gives rise to polystable Higgs bundles if and only if $0\leq\deg(\cB)\leq g-1$. Such Higgs bundles are given by
\[
\xymatrix{ \cK \ar@/_/[r]_{1}  & \cO \ar@/_/[r]_{1} & \cK^{-1}  }~\oplus~ \xymatrix{\cB  \ar@/_/[r]_{-\frac{\i}{ 2}} & \cK^{-1}\cB  & \cK\cB^{-1} \ar@/_/[r]_{-\frac{\i}{ 2}} \ar@/_1pc/[ll]_{\delta}  & \cB^{-1} \ar@/_1pc/[ll]_{\delta}}~,
	\]
They correspond to equivariant holomorphic curves $(\rho,f)$ with zero second fundamental form, and hence the image of $f$ is a totally geodesic holomorphic disk in $\H$. Such a holomorphic disk is the intersection of $\H$ with a copy of $\Im(\quat')$, where $\quat'$ is a split-quaternion subalgebra of $\oct'$. In particular, $\rho$ factors through the stabilizer of $\quat'$ which is isomorphic to $\SO(2,2)$ by Lemma \ref{lem:StabilizerSplitQuaternion}.

The $\SO(2,2)$-representation obtained can be described explicitly. 
Choose a square root $\cK^{\frac{1}{2}}$ of the canonical bundle and let $(\cF_1,\varphi_1)$ and $(\cF_2,\varphi_2)$ be the following $\SL(2,\R)$ Higgs bundles
\[(\cF_1,\varphi_1) := \xymatrix{\cK^{\frac{1}{2}}  \ar@/_/[r]_{-\frac{\i}{ 2}} & \cK^{-\frac{1}{2}}} ~\ ,\ ~ \ (\cE_2,\varphi_2) := \xymatrix{\cK^{\frac{1}{2}}\cB^{-1}  \ar@/_/[r]_{\delta} & \cK^{-\frac{1}{2}}\cB}~.\]
$(\cE_1,\Phi_1)$ is the Higgs bundle corresponding to the Fuchsian representation $\rho_1$ uniformizing $X$, while the second one corresponds to a representation $\rho_2:\pi_1\Sigma\to\sSL(2,\R)$ with Euler class $g-1-d$. Moreover, we have 
\[\xymatrix{\cB  \ar@/_/[r]_{-\frac{\i}{ 2}} & \cK^{-1}\cB  & \cK\cB^{-1} \ar@/_/[r]_{-\frac{\i}{ 2}} \ar@/_1pc/[ll]_{\delta}  & \cB^{-1} \ar@/_1pc/[ll]_{\delta}}= (\cF_1\otimes\cF_2, \varphi_1\otimes \Id + \Id\otimes \varphi_2)~.\]
meaning that, under the $2$-to-$1$ cover $\SL(2,\R)\times \SL(2,\R) \to \SO_0(2,2)$ coming from the tensor product of the standard representation, we have $\rho = \rho_1\otimes \rho_2$.

Finally, we describe how these totally geodesic holomorphic curves arise as limits of alternating ones. Given a $\G$-Higgs bundle $(\cE,\Phi)$ of the form \eqref{eq cyclic G2 higgs B compact} with $0\leq \deg(\cB)\leq g-1$, consider the $\C^*$-family given by $[\cE,t\Phi]$, for $t\in \C^*$.  
Acting by the gauge $g_t= \diag(1,t^{-1},t,1,t^{-1},t,1)$, we see that $(\cE,t\Phi)$ is equivalent to
\[\xymatrix{\cB  \ar@/_/[r]_{-\frac{\i}{ 2} \ \ \ } & \cB\cK^{-1}  \ar@/_/[r]_{t^3\beta} & \cK \ar@/_/[r]_{1}  & \cO\ar@/_/[r]_{1} & \cK^{-1} \ar@/_/[r]_{t^3\beta}  & \cB^{-1}\cK \ar@/_/[r]_{-\frac{\i}{ 2}} \ar@/_2pc/[lllll]_{\delta} & \cB^{-1} \ar@/_2pc/[lllll]_{\delta}}~.\]
In particular, $[\cE,t\Phi]$ corresponds to an equivariant alternating holomorphic curve $[\rho_t,f_t]\in\cH_d(\Sigma)$ for $0\leq d\leq g-1$. Taking the limit as $t$ goes to $0$, we get that $[\rho_t,f_t]$ converges to an equivariant totally geodesic holomorphic curve.

Recall from Theorem \ref{thm: moduli fixed Riemann} that for $0\leq d\leq g-1$ the spaces $\cH_d(X)$ are biholomorphic to the quotient of the complement of the zero section of a rank $5g-5-d$ vector bundle by the action of $\pm \Id.$ By the above discussion, we see that the totally geodesic holomorphic curves correspond to the zero section of this vector bundle. 
In particular, 
\begin{equation}
	\label{eq hat Hd}\cH_d(\Sigma)\subset\widehat\cH_d(\Sigma)\to\cT(\Sigma)~,
\end{equation}
where the fiber over $X\in\cT(\Sigma)$ of $\widehat\cH_d(\Sigma)$ is the quotient by $\pm\Id$ of the rank $5g-5-d$ vector bundle. By construction $\widehat\cH_d(\Sigma)\setminus\cH_d(\Sigma)$ is the zero section of this vector bundle and corresponds to totally geodesic equivariant holomorphic curves.


\section{Cyclic surfaces and infinitesimal rigidity}\label{s cyclic surfaces}

In this section, we introduce the notion of \emph{cyclic surfaces} (Definition \ref{def:CyclicSurface}), which is slightly more general than  the one introduced by Labourie in \cite{cyclicSurfacesRank2}. Cyclic surfaces are a special type of holomorphic curves in a homogeneous space $\Cy$ which we call the \emph{cyclic space} (Definition \ref{def:CyclicSpace}). The cyclic space in particular fibers over both over $\H$ and the Riemannian symmetric space $\Sym(\G)$. 
An important property of cyclic surfaces is that their projection to $\H$ are (possibly branched) alternating holomorphic curves and their projection to $\Sym(\G)$ are (possibly branched) minimal surfaces whose underlying harmonic bundle is cyclic. 

To ensure non-branched immersions in both case, we restrict to the class of $\alpha_1$-cyclic surfaces. We prove that $\alpha_1$-cyclic surfaces arise from the Frenet lift of a unique alternating holomorphic curve in Theorem \ref{thm:LinkCyclicHiggs}. We then use this interpretation to prove that equivariant  alternating holomorphic curves $[\rho,f]$ are infinitesimally rigid in Theorem \ref{thm:InfinitesimalRigidity}.





\subsection{Lie theory} Fix a maximal torus $\sT^\C$ in $\GC$ and denote by $\frak t^\C$ its Lie algebra. Let $\Delta \subset (\ft^\C)^*$ be the set of roots and, for any $\gamma\in \Delta$, denote by $\fg_\gamma$ the corresponding root space. The root space decomposition is then
\[\fg_2^\C = \ft^\C \oplus \bigoplus_{\gamma\in \Delta} \fg_\gamma~.\]
Writing $\fg_0=\ft^\C$, we have that $[\fg_\gamma,\fg_\delta] \subset \fg_{\gamma+\delta}$ for any $\gamma$ and $\delta$ in $\Delta\cup \{0\}$. In particular, if $\gamma+\delta$ is not in $\Delta\cup \{0\}$ we have $[\fg_\gamma,\fg_\delta]=0$.

Choosing a subset $\Delta^+\subset \Delta$ of positive roots defines two simple roots $\{\alpha_1,\alpha_2\}$ where $\alpha_1$ is the short root. Explicitly, we have
\[\Delta^+ = \big\{ \alpha_1~,~ \alpha_2 ~,~ \alpha_1+\alpha_2~,~ 2\alpha_1+\alpha_2~,~ 3\alpha_1+\alpha_2 ~,~3\alpha_1+2\alpha_2\big\}~.\]

\begin{center}
	\begin{tikzpicture}[scale=.8]
\draw[thick,<->] (-1,0) -- (1,0);
\draw[thick,<->] (-.5,-.866) -- (.5,.866);
\draw[thick,<->] (.5,-.866) -- (-.5,.866);
\draw[thick,<->] (0,-1.732) -- (0,1.732);
\draw[thick,<->] (-1.5,-.866) -- (1.5,.866);
\draw[thick,<->] (1.5,-.866) -- (-1.5,.866);
\node at (1.2,.15)  {$\alpha_1$};
\node at (-1.332,+1.1)  {$\alpha_2$};
\end{tikzpicture}
\end{center}


\noindent Any root $\gamma\in \Delta^+$ is uniquely written as $\gamma = n\alpha_1+ m \alpha_2$ with $n,m\in \mathbb N$; we define the \emph{length} of $\gamma$ as
\[\ell(\gamma) := n+m~.\] 
The longest root is $\eta:=3\alpha_1+2\alpha_2$ which has length $5$.
For any $k\in \Z_6$ set
\[\fg_k = \begin{cases}
	\bigoplus\limits_{-\ell(\gamma)=k} \fg_\gamma & \text{ if } k\neq 0 \\\ft^\C&\text{ if } k=0\end{cases}
	~.\]
The \emph{cyclic grading} is then defined to be the $\Z_6$-grading 
\begin{equation}\label{eq:Grading}
\fg_2^\C = \bigoplus_{k\in \Z_6}\fg_k~.
\end{equation}
In particular, we have $\fg_1 = \fg_{-\alpha_1}\oplus \fg_{-\alpha_2} \oplus \fg_\eta$.

\begin{lemma}
There exists a Cartan involution $\theta$ on $\fg_2^\C$ such that $\rho(\fg_k)=\fg_{-k}$ for any $k\in \Z_6$.
\end{lemma}

\begin{proof}
This can be proved by choosing a Chevalley basis (see \cite{cyclicSurfacesRank2} for more details).
\end{proof}

Fixing such a Cartan involution, we denote by $\sK$ the associated maximal compact subgroup of $\GC$ and by $\sT = \sT^\C\cap \K$ the maximal compact subgroup of $\sT^\C$. In particular, $\K$ is the compact real form of $\GC$ and $\Sym(\GC):=\GC/\K$ is the Riemannian symmetric space of $\GC$.

Denote the Lie algebra of $\sT$ by $\ft$. The cyclic grading \eqref{eq:Grading} gives an $\text{Ad}(\T)$-invariant splitting
\begin{equation}\label{eq:TinvariantSplitting}\fg_2^\C = \ft \oplus \i\ft \oplus \bigoplus_{k\in \Z_6\setminus\{0\}} \fg_k~.\end{equation}
We will denote by $\ft^\bot$ the direct sum $\i \ft \oplus \bigoplus_{k\in \Z_6\setminus \{0\}} \fg_k$.

Define the $\C$-linear involution $\sigma$ of $\fg_2^\C$ by
\[\sigma\vert_{ \fg_k} = (-1)^k \Id_{\fg_k}~,~k\in \Z_6~.\]
Since the grading is even, one easily checks that $\sigma$ is an automorphism of $\fg_2^\C$. Moreover, $\sigma$ commutes with $\theta$ and the corresponding real form $\lambda:= \sigma \circ \theta$ has fixed point the Lie algebra of $\G$.

One can be more explicit in the description of the objects above. Consider the basis $(f_{-3},f_{-2},...,f_3)$ of $\Im(\oct')\otimes\C$ as in equation \eqref{eq basis of complexification}, and let $F_k=\span_\C\{f_k\}$ for each $k$. The complex quadratic form $\q$ identifies $F^*_a$ with $F_{-a}$ for any $a$ while the map $(x,y) \mapsto \frac{-1}{\sqrt 2} x\times y$ identifies $F_{-2}$ with $F_{-3}\otimes F_1$.
The subgroup of $\GC$ preserving the splitting $F_{-3}\oplus F_{-2}\oplus\cdots\oplus F_3$  is a maximal torus $\sT^\C$. In particular, the vector space $\bigoplus_{k\in \Z_6\setminus \{0\}} \fg_k$ can be seen as the tangent space to $\GC/\sT^\C$ at the point $e\sT^\C$ and thus embeds in $\bigoplus_{k\neq l} \Hom(F_k,F_l)$. 
Furthermore, any root space $\fg_\gamma$ embeds in $\bigoplus_{k=-3}^{3-l} \Hom(L_k,L_{k+l})$ where $l=\ell(\gamma)$. We can thus represents any element $(a,b,d) \in \fg_{-\alpha_1}\oplus\fg_{-\alpha_2}\oplus \fg_\eta = \fg_1$ as
\begin{equation}\label{eq:spaceg1}
\xymatrix{F_{-3}  \ar@/_/[r]_{-\frac{\i}{2}a} & F_{-2}  \ar@/_/[r]_b & F_{-1} \ar@/_/[r]_a  & F_0 \ar@/_/[r]_a & F_1 \ar@/_/[r]_b  & F_2 \ar@/_/[r]_{-\frac{\i}{2}a} \ar@/_2pc/[lllll]_{c} & F_3 \ar@/_2pc/[lllll]_{c}}~,\end{equation}
where, as in \eqref{eq cyclic G2 higgs B compact}, the map $-\frac{\i}{2}a:F_{-3}\to F_{-2}$ is defined by $-\frac{\i}{2}a(f_{-3})=-\frac{\i}{2}a(f_0)\times f_{-3}$.
\subsection{Cyclic surfaces}\label{ss:CyclicSurfaces}

We continue to use the notation from the previous section.
\begin{definition}\label{def:CyclicSpace}
The \emph{cyclic space} $\Cy$ is the homogeneous space $\GC/\sT$.
\end{definition}

The natural projection $\GC \to \Cy$ is a principal $\sT$-bundle. For any linear action of $\sT$ on a vector space $V$, we denote by $[V]= \GC\underset{\sT}{\times} V$ the associated vector bundle over $\Cy$.

Let $D\in \Omega^1(\GC,\frak g_2^\C)$ be the Maurer-Cartan form on $\GC$. The Maurer-Cartan equation is then
\begin{equation}\label{eq:MaurerCartanEq}
d D + \frac{1}{2}[D\wedge D] =0 ~.
\end{equation}
Using the $\text{Ad}(\T)$-invariant splitting of equation \eqref{eq:TinvariantSplitting}, we can postcompose $D$ with the projection on each factor. The projection onto $\ft$ gives a connection $A$ on the principal $\sT$-bundle $\GC$. The 1-form $\omega:= D-A$ vanishes on the vertical tangent space of $\GC\to\Cy$ and thus descends to a nowhere vanishing $1$-form on $\Cy$ with values in the associated bundle $[\ft^\bot]$. In particular, the $\omega$ identifies $\T\Cy$ with $[\ft^\bot]$.
Since the splitting $\ft^\bot=\i \ft \oplus \bigoplus_{k\in \Z_6\setminus \{0\}} \fg_k$ is $\sT$-invariant, we obtain
\begin{equation}\label{eq:SplittingTX}\T\Cy = [\i\ft]\oplus \bigoplus_{k\in \Z_6\setminus\{0\}} [\fg_k]~.\end{equation}
The involutions $\theta,\sigma$ and $\lambda$ of $\fg_2^\C$ extends equivariantly to involutions on $\T \Cy$ which we denote with the same symbols. Since $\lambda$ is an automorphism of $\fg_2^\C$, the distribution in $\Cy$ given by the fixed points of $\lambda$ is integrable and its leaves correspond to orbits of subgroup of $\GC$ conjugated to $\G$.

For any $k\in \Z_6\setminus\{0\}$ we denote by $\omega_k\in \Omega^1(\Cy,[\fg_k])$ the projection of $\omega$ on $[\fg_k]$. Similarly, for any root $\gamma$, we denote by $\omega_{\i\ft}$ and $\omega_\gamma$ the projection of $\omega$ on $[\i\ft]$ and $[\fg_\gamma]$ respectively. Projecting Maurer-Cartan equation on $[\ft]$ and $[\fg_k]$ for $k\neq 0$, we get
\begin{equation}
	\label{eq:MCSplitted}
	\vcenter{\xymatrix@=0em{F_A + \frac{1}{2}\sum_{i+j=0} [\omega_i\wedge\omega_j]^{\ft} = 0 \\
	d^A\omega_k + \frac{1}{2}\sum_{i+j=k} [\omega_i\wedge\omega_j] = 0 }}
\end{equation}
where $[.,.]^{\ft}$ denotes the projection of the Lie bracket on $\ft$ parallel to $\ft^\bot$.

\begin{definition}\label{def:CyclicDistrib}
The \emph{cyclic distribution} $\DD\subset\T \Cy$ is the intersection of the fixed points of $\lambda$ with $[\fg_{-1}\oplus\fg_1]$.
\end{definition}
The $\sT$-equivariant endomorphism 
\[\JJ:\xymatrix@R=0em{\fg_{-1}\oplus \fg_1\ar[r]&\fg_{-1}\oplus \fg_1\\(x,y)\ar@{|->}[r]&(-\i x,\i y)}\]
commutes with $\lambda$ and square to $-\Id$. Thus, it induces an almost complex structure $\JJ$ on $\DD$. 

Given a smooth oriented surface $S$, an orientation preserving smooth map $f: S \to \Cy$ will called a \emph{holomorphic curve in} $\Cy$ if $f$ is tangent to the cyclic distribution and $df(\T S)$ is $\JJ$-invariant. 
Given $k\in \Z_6\setminus \{0\}$, we denote by $df_k$ the 1-form $\omega_k(df)$ with value in $f^*[\fg_k]$. Similarly, for any root $\gamma\in \Delta$, we denote by $df_\gamma$ the form $\omega_\gamma(df)$.

\begin{definition}\label{def:CyclicSurface}
Let $f: S \to \Cy$ be a holomorphic curve, 
\begin{itemize}
	\item $f$ is a cyclic surface if $df_{-\alpha_1}$ and $df_{-\alpha_2}$ are not identically zero,
	\item $f$ is an $\alpha_1$-cyclic surface if it is a cyclic surface and $df_{-\alpha_1}$ is nowhere vanishing.
\end{itemize} 
\end{definition}

\begin{remark}
In \cite{cyclicSurfacesRank2}, a cyclic surface is assumed to have $df_{-\alpha_i}$ nowhere vanishing for all simple roots $\alpha_i.$ When considering cyclic surfaces which are equivariant under surface group representations, this assumption forces the representation to be in the Hitchin component. Insisting $df_{\alpha_i}$ is nowhere vanishing for only one simple root allows more flexibility, and is still strong enough to ensure the infinitesimal rigidity of Theorem \ref{thm:InfinitesimalRigidity}. This loosening of the definition also appears in \cite{MySp4Gothen}. In \cite{CTcyclic}, we consider a more general setup which guarantees infinitesimal rigidity.
\end{remark}

\subsection{Cyclic surfaces and cyclic harmonic bundles} We now describe how cyclic surfaces are a differential geometric interpretation of $\G$-cyclic harmonic bundles. 
Given a cyclic surface $f: S \to \Cy$, denote by $\nabla=f^* A$ the pull-back connection on $f^*\T\Cy$.

\begin{lemma}\label{lemma:Cyclic} Let $f: S \to \Cy$ be a cyclic surface, then
\begin{enumerate}
	\item There is a unique complex structure $\j$ on $S$ such that, for any root space $\fg_\gamma$ in $\fg_1$, the form $df_{\gamma}$ has type $(1,0)$. In particular $\phi:= df_1$ is of type $(1,0)$.
	\item Up to postcomposing $f$ with an element in $\GC$, the map $f$ takes value in a $\G$-orbit.
	\item The Maurer-Cartan equations are equivalent to
	\[\left\{ \begin{array}{lll} F_\nabla - [\phi\wedge \rho(\phi)] & = & 0  \\
	 \overline\partial_\nabla \phi = 0 &  & \end{array}\right.,\]
	where $\overline\partial_\nabla$ is the $(0,1)$-part of $d^\nabla$.
\end{enumerate}
\end{lemma}

\begin{proof}
For the first item, $f$ is an immersion away from a discrete set $D$. Thus, there exists a unique complex structure $\j'$ on $S\setminus D$ such that $df\circ j= \JJ \circ df$. The complex structure $\j'$ extends uniquely on $S$. The result then follows from the fact that $\JJ$ is the multiplication by $\i$ on $[\fg_1] = [\fg_{-\alpha}\oplus \fg_{-\beta} \oplus \fg_\eta]$ so the projection on each $[\fg_\gamma]$ is holomorphic.

The second item follows from the fact that $\DD$ is tangent to the distribution fixed by $\lambda$ whose leaves are orbits of groups conjugated to $\G$.

For the third item, note that the only nonzero terms in the pull-back of equations \eqref{eq:MCSplitted} are in degree $0,\pm 1, \pm2$. Using $df_{-1}=-\rho(df_1)$, we see that the equations in degree $-1$ and $-2$ are equivalent to the ones in degree $1$ and $2$ respectively. The equation in degree $0,1,2$ are respectively
\[\xymatrix{F_\nabla - [\phi\wedge \rho(\phi)] = 0~,&d_\nabla \phi = 0&\text{and}&\phi\wedge \phi=0 }~.\]
Since $\phi$ is a $(1,0)$-from, $\phi\wedge \phi$ is automatically satisfied and $d_\nabla \phi = \bar\partial_\nabla \phi.$
\end{proof}

Denote by $\Sym(\GC)$ and $\Sym(\G)$ the (Riemannian) symmetric space of $\GC$ and $\G$, respectively. We see $\Sym(\G)$ as a totally geodesic subspace of $\Sym(\GC)$. Denote by $\pi: \Cy \to \Sym(\GC)$ the natural projection. Note that a cyclic surface induces a Riemann surface structure $X=(S,\j)$ on $S.$

\begin{theorem}\label{thm:LinkCyclicHiggs}
Given a cyclic surface $f: S \to \Cy$, the map $\pi\circ f:S\to \Sym(\GC)$ is a branched minimal immersion whose image is contained in $\Sym(\G)$. Moreover, the corresponding $\G$-harmonic bundle on the induced Riemann surface has the form \eqref{eq cyclic G2 higgs B} with $\alpha$ and $\beta$ non-zero.

Conversely, given a $\G$-cyclic harmonic bundle on a Riemann surface $X$ of the form \eqref{eq cyclic G2 higgs B} with $\alpha$ and $\beta$ non-zero, let $\rho$ be the holonomy representation of the flat connection. Then the cyclic harmonic metric defines a conformal $\rho$-equivariant harmonic map $h_\rho:\widetilde X\to \Sym(\G)$ which lifts to a $\rho$-equivariant cyclic surface in $\Cy$.
\end{theorem}

\begin{proof}
To prove that $\pi\circ f$ is a branched minimal immersion, observe that, when $\Cy$ is equipped with the pseudo-Riemannian metric induced by the Killing form, the map $\pi: \Cy \to \Sym(\G)$ is a pseudo-Riemannian submersion whose horizontal distribution is given by 
$$\text{H} \Cy = \{x\in \T\Cy~,~\theta(x)=-x\}~.$$ 
In particular, the horizontal distribution $\text{H} \Cy$ is naturally identified with $\pi^*\T \Sym(\GC) $. Moreover, this identification also identifies  the restriction of $A$ to $\text H \Cy$ with the pull-back by $\pi$ of the Levi-Civita connection on $\Sym(\G)$ (this can be easily seen using the Kozsul formula).

Since the cyclic distribution $\DD$ is a subdistribution of $\text H\Cy$, the $(1,0)$-part of the differential of $\pi\circ f$ is identified with $\phi=df_1$ which is holomorphic by item $(3)$ of Lemma \ref{lemma:Cyclic}. It follows that $\pi\circ f$ is harmonic. 

To prove that $\pi\circ f$ is a minimal immersion, we compute its Hopf differential. But since $f$ is tangent to the horizontal distribution, the Hopf differential of $\pi\circ f$ is equal to the one of $f$, which is $0$ since $df_1$ is tangent to the vector bundle $[\fg_1]$ which is isotropic with respect to the Killing form. So $\pi\circ f$ is weakly conformal and harmonic so is a branched minimal immersion.

The fact that $\pi\circ f$ lies in $\Sym(\G)$ is a direct consequence of item (2) of Lemma \ref{lemma:Cyclic}.

Consider now the associated $\G$-Higgs bundle $(\mathcal E,\Phi)$ on $(S,\j)$. In particular, $\mathcal E$ is associated to the principal $\sK^\C$ obtained by complexifying the principal $\sK$-bundle $(\pi\circ f)^*\G$ (where we see $\G$ as a principal $\sK$-bundle over $\Sym(\G)$). The fact that $\pi\circ f$ lifts to $\Cy$ means that $\mathcal E$ reduces to a holomorphic $\sT^\C$-bundle which is compatible with the metric. Under this reduction, the Higgs field $\Phi= d(\pi\circ f)^{1,0}$ lies in the bundle associated to the adjoint action of $\sT^\C$ on $\fg_1$. According to \eqref{eq:spaceg1}, this exactly means that $(\mathcal E,\Phi)$ has the form of \eqref{eq cyclic G2 higgs B} with $\alpha,\beta$ and $\delta$ corresponding respectively to $df_{-\alpha_1}, df_{-\alpha_2}$ and $df_\eta$.


Conversely, a diagonal harmonic metric on $(\mathcal E,\Phi)$ corresponds to a harmonic map to $\Sym(\G)$ which is compatible with the holomorphic $\sT^\C$-structure on $\mathcal E$. This means that the harmonic map lifts to a map $f: (S,\j) \to \Cy$. The Higgs field $\Phi$ is identified with $df^{1,0}$ which takes value in $[\fg_1]$.

By construction, the underlying $\sK$-bundle reduces to a $\sT$-bundle. This exactly means that $\mathcal E$ splits holomorphically as a direct sum of line bundles. Similarly, the fact that the lift is cyclic means that $\Phi$ as the form of \eqref{eq cyclic G2 higgs B}, where $\alpha=df_{-\alpha_1}~,~\beta=df_{-\alpha_2}$ and $\delta= df_\eta$.
\end{proof}

\subsection{Infinitesimal rigidity} Let $\Sigma$ be a closed oriented surface. An equivariant cyclic surface is a pair $(\rho,f)$, where $\rho:\pi_1(\Sigma)\to \sG_2'$ is a representation and $f:\widetilde\Sigma\to \Cy$ is a cyclic surface which is $\rho$-equivariant.

A \emph{smooth family of equivariant cyclic surfaces} is a smooth map $$F:\widetilde\S \times (-\epsilon,\epsilon) \to \Cy$$  such that the map $f_t:=F(\cdot,t)$ is a cyclic surface which is $\rho_t$-equivariant where the corresponding $[\rho]: (-\epsilon,\epsilon) \to \cX(\S,\G)$ is smooth. We denote by $(\rho_t,f_t)_{t\in (-\epsilon,\epsilon)}$ such a family, $\overset{\bbullet}{f_0}= \left.dF(\partial_t)\right\vert_{t=0}$ the vector field along $f_0$ and $[\overset{\bbullet}{\rho_0}] = d_0[\rho](\partial_t)$.

\begin{theorem}\label{thm:InfinitesimalRigidity}
Let $(\rho_t,f_t)_{t\in(-\epsilon,\epsilon)}$ be a smooth path of equivariant $\alpha_1$-cyclic surfaces such that $[\overset{\bbullet}{\rho_0}]=0$. Then there exists a smooth path $(g_t,\psi_t)\in \G\times \Diff_0(\S)$ such that $\overset{\bbullet}{f'_0}=0$ where $f'_t=g_t\circ f_t\circ\psi_t$.
\end{theorem}

The proof follows (and simplify) the ideas introduced by Labourie in \cite{cyclicSurfacesRank2}.
A tangent vector $\zeta$ to a smooth family $(f_t)_{t\in(-\epsilon,\epsilon)}$ of cyclic surfaces is called a \emph{Jacobi field}. We see $\zeta$ as a section of $f_0^*\T\Cy$. Denote the projection of $\zeta$ on $f^*[\fg_{-\alpha}]$ by $\zeta_{-\alpha}$. The following proposition is the key technical result used to prove Theorem \ref{thm:InfinitesimalRigidity}. We postpone the proof until the final section.

\begin{proposition}\label{prop:InfinitesimalRigidity}
Let $f: \widetilde \S \to \Cy$ be an $\alpha_1$-cyclic surface and let $\zeta$ be a Jacobi field along $f$. If $\zeta$ is $\pi_1(\S)$-invariant and $\zeta_{-\alpha_1}=0$, then $\zeta=0$.
\end{proposition}

\begin{proof}[Proof of Theorem \ref{thm:InfinitesimalRigidity} using Proposition \ref{prop:InfinitesimalRigidity}]
Consider a smooth family $(\rho_t,f_t)_{t\in(-\epsilon,\epsilon)}$  of equivariant $\alpha_1$-cyclic surfaces such that $[\overset{\bbullet}{\rho_0}]=0$. In particular, there exists a smooth path $(g_t)_{t\in(-\epsilon,\epsilon)}$ in $\G$ such that $\left.\frac{d}{dt}\right\vert_{t=0} g_t\rho_tg_t^{-1}=0$. Hence, the tangent vector to the family $(g_t\circ f_t)_{t\in (-\epsilon,\epsilon)}$ is a $\pi_1(\S)$-invariant Jacobi field.

By assumption, $df_{-\alpha_1}$ is never vanishing. So there exists a unique vector field $X$ on $\widetilde\S$ with $df_{-\alpha_1}(X)=-\zeta_{-\alpha_1}$. Let $\psi_t$ be the flow of $X$ at time $t$ and define $f'_t:=g_t\circ f_t\circ\psi_t$. In particular, each $f'_t$ is a $g_t\rho_t g_t^{-1}$-equivariant $\alpha_1$-cyclic surface, and the tangent vector to the family $(f'_t)$ is  a $\pi_1(\S)$-invariant Jacobi field $\zeta'$ which satisfies $\zeta'_{-\alpha}=0$. The result follows from Proposition \ref{prop:InfinitesimalRigidity}.
\end{proof}

\subsection{Infinitesimal deformations of cyclic surfaces}
The final section is devoted to the study of infinitesimal deformation of cyclic surfaces. In particular, we prove Proposition \ref{prop:InfinitesimalRigidity}.
We follow the idea of \cite{cyclicSurfacesRank2}, but our proof simplifies many relevant computations. We hope this sheds light on the important ingredients to ensure local rigidity.

Consider a Lie algebra bundle  $\pi: \hat E \to M$  on a manifold $M$, and let $\hat \nabla$ be a compatible connection (that is $\hat\nabla$ is a derivation for the Lie bracket). Consider $\Theta= \{\theta_1,\dots,\theta_k\}$ with $\theta_k \in \Omega^{\bullet}(M,\hat E)$. A \emph{solution to the Pfaffian system $\Theta$} is then a map $f: N \to M$ such that $f^*\theta_k=0$ for all $\theta_k\in \Theta$.
Let $\mathcal I_\Theta$ be the differential ideal generated by $\Theta$. 

\begin{lemma}\label{lem:Pfaffian}  A map $f: N \to M$ is a solution to the Pfaffian system $\Theta$ if and only if $f^*\alpha=0$ for any $\alpha\in \mathcal I_\Theta$.
\end{lemma}

\begin{proof}
The direct implication follows from the fact that $f^*[\alpha\wedge \beta]=[f^*\alpha\wedge f^*\beta]$ and $f^*d_{\hat\nabla}\alpha= d_\nabla f^*\alpha$ for $\nabla= f^*\hat\nabla$. The converse follows from the fact that $\Theta\subset \mathcal I_\Theta$.
\end{proof}

The following is proved in \cite[Proposition 7.14]{cyclicSurfacesRank2}:

\begin{proposition}\label{prop:InfinitesimalDeformationOfPfaffian}
Let $(f_t)_{t\in (-\epsilon,\epsilon)}$ be a smooth path of solution of the Pfaffian system $\Theta$ and let $\zeta = \frac{d}{dt}\vert_{t=0} f_t$. Then for any $\alpha= f^*\hat\alpha$ with $\hat\alpha\in \mathcal I_\Theta$ we have 
\[d^\nabla \circ \iota_\zeta \alpha= - \iota_\zeta \circ d^\nabla \alpha~.\]
\end{proposition}

Recall that the Maurer-Cartan equation implies that the connection $D=A+\omega$ on $[\ft]\oplus \T\Cy$ is flat, so $[\ft]\oplus \T\Cy$ is isomorphic to the trivial Lie algebra bundle $\Cy\times \fg_2^\C$. Consider the Pfaffian system
\[\Theta := \big\{\omega_1+\theta(\omega_{-1})~,~\omega_{\i\ft}~,~ \omega_j \text{ for }j\in \Z_6\setminus\{0,\pm 1\} \big\}~.\]

\begin{lemma}\label{lem:AppB}
The form $[\omega_{-1}\wedge\omega_{-1}]$ belongs to the ideal $\mathcal I_\Theta$.
\end{lemma}

\begin{proof}
The form $\omega_{-2}$ is in $\Theta$, so $d^A\omega_{-2}$ is in $\mathcal I_\Theta$. But equation \eqref{eq:MCSplitted} for $k=2$ gives
\[d^A \omega_{-2} + \frac{1}{2}\sum_{i+j=-2} [\omega_i\wedge \omega_j] = 0 ~.\]
In $\Z_6$ the equation $i+j=-2$ implies $i=j=-1$ or $i,j\notin \{\pm 1\}$. But for $i,j\notin \{\pm 1\}$ we have $\omega_i,\omega_j \in \Theta$ so $[\omega_i\wedge \omega_j]\in \mathcal I_\Theta$. In particular, the above equation implies that $[\omega_{-1}\wedge\omega_{-1}] \in \mathcal I_\Theta$.
\end{proof}

One easily checks that a map $f: S \to \Cy$ is tangent to the cyclic distribution $\DD$ if and only if it is solution to the Pfaffian system $\Theta$.

\begin{definition}\label{def:JacobiField}
Let $f: S \to \Cy$ be a cyclic surface and  $\nabla = f^* A$. A \emph{Jacobi field} of $f$ is a section $\zeta$ of $f^*\T\Cy$ such that, for any $\alpha= f^*\hat\alpha$ with $\hat\alpha \in \mathcal I_\Theta$, we have
\[d^\nabla \circ \iota_\zeta \alpha= - \iota_\zeta \circ d^\nabla \alpha~.\]
\end{definition}

Since cyclic surfaces are solution to the Pfaffian system $\Theta$, Proposition \ref{prop:InfinitesimalDeformationOfPfaffian} implies that a vector field tangent to a deformation of cyclic surfaces is in particular a Jacobi field.

Let $f: S \to \Cy$ be a cyclic surface with $\phi=df_1$ and $\phi^\dagger = df_{-1} = -\theta(df_1)$ and let $\j$ be the complex structure on $S$ associated to $f$. Define 
\[\xymatrix{[\fG]=\bigoplus_{j\neq 0,\pm 1} [\fg_j]&\text{ and}&\Omega= \bigoplus_{j\neq 0,\pm 1} \omega_j.}\]
A Jacobi field $\zeta$ which is tangent to a path of cyclic surfaces $f_t$ with $f_0=f$, decomposes as
\[\zeta = \zeta_0 + \zeta_1 + \zeta_{-1} + Z,\]
where $\zeta_j=\iota_{\zeta}(f^*\omega_j)$ is a section of $f^*[\fg_j]$ and $Z = \iota_\zeta (f^*\Omega)$ is a section of $f^*[\fG]$.

\begin{proposition}\label{prop:Labourie'sComputations}
A Jacobi field $\zeta$ of a cyclic surface satisfies  $\zeta_0=0$, $\lambda(Z)=Z$ and $\zeta_1=-\theta(\zeta_{-1})$.
\end{proposition}

\begin{proof}
The second item of Lemma \ref{lemma:Cyclic} implies that, up to postcomposing $f_t$ by an element in $\GC$, the cyclic surfaces $f_t$ are contained in a given $\G$-orbit. Since the tangent space of such an orbit is $\text{Fix}(\lambda)$, we have that $\lambda (\zeta) = \zeta$. 
But $\lambda(\zeta) = -\zeta_0 - \theta(\zeta_{-1}) - \theta(\zeta_1) + \lambda(Z)$. Projecting on $[\i\ft]$ gives $\zeta_0=0$, projecting on $\fG$ gives $\lambda(Z)=Z$, while projecting on $[\fg_1]$ gives $\zeta_1=-\theta(\zeta_{-1})$.
\end{proof}
The following appears in \cite[Proposition 7.6.1]{cyclicSurfacesRank2}. For completeness, we repeat them here. 
\begin{lemma}\label{lem Labcomp}
	\begin{enumerate}
		\item $d^\nabla Z = [(\zeta_1+Z)\wedge \phi]^{\fG} + [(\zeta_{-1}+Z)\wedge \phi^\dagger]^{\fG}$.
		\item $[\partial_\nabla \zeta_{-1}\wedge \phi^\dagger] = [[Z\wedge \phi]^{\fg_{-1}}\wedge \phi^\dagger]$
	\end{enumerate}
\end{lemma}

\begin{proof}
For the first item, we have by Definition \ref{def:JacobiField}
\[d^\nabla Z = d^\nabla(\iota_\zeta (f^*\Omega)) = - \iota_\zeta(d^\nabla (f^*\Omega))~.\]
However, projecting Maurer-Cartan equations on $[\fG]$ gives
\[-d^A \Omega = \frac{1}{2}[\Omega\wedge\Omega]^\fG + [\omega_0\wedge \Omega] +[\omega_0\wedge\Omega] +\frac{1}{2}[\omega_1\wedge\omega_1]+\frac{1}{2}[\omega_{-1}\wedge\omega_{-1}] + [\omega_1\wedge \Omega]^\fG  + [\omega_{-1}\wedge \Omega]^\fG~.\]
Pulling-back by $f$, using $f^*\Omega=f^*\omega_0=0$ and applying $\iota_\zeta$ yields
\[-\iota_\zeta (d^\nabla (f^*\Omega)) = [(\zeta_1+Z)\wedge \phi]^\fG+[(\zeta_{-1}+Z)\wedge \phi^\dagger]^\fG~.\]
For the second item, first compute
\begin{equation}\label{eq:LabourieComp1}d^\nabla (\iota_\zeta f^*[\omega_{-1}\wedge \omega_{-1}]) = 2 d^\nabla[\zeta_{-1}\wedge \phi^\dagger] = 2 [\partial_\nabla \zeta_{-1}\wedge \phi^\dagger]~,\end{equation}
where we used the fact that $\phi^\dagger$ has type $(0,1)$ and satisfies $d^\nabla\phi^\dagger=0$. On the other hand, using Lemma \ref{lem:AppB} we have
\begin{equation}\label{eq:LabourieComp2}d^\nabla(\iota_\zeta f^*[\omega_{-1}\wedge\omega_{-1}]) = -\iota_\zeta\big(d^\nabla (f^*[\omega_{-1}\wedge\omega_{-1}])\big) = -2\big[\iota_\zeta (d^\nabla (f^* \omega_{-1}))\wedge \phi^\dagger\big]~.\end{equation}
But the Maurer-Cartan equations on $[\fg_{-1}]$ gives
\[-d^A\omega_{-1} = [\omega_0\wedge \omega_{-1}] + [\Omega \wedge \omega_1]^{\fg_{-1}} + \frac{1}{2}[\Omega\wedge\Omega]^{\fg_{-1}}~.\]
This yields $-\iota_\zeta (d^\nabla (f^*\omega_{-1})) = [Z\wedge \phi]^{\fg_{-1}}~.$ Plugging back in equations \eqref{eq:LabourieComp2} and comparing with equation \eqref{eq:LabourieComp1} give the result.
\end{proof}

\begin{proposition}\label{prop:Labourie'sComputations2}
	A Jacobi field of a cyclic surface satisfies
	\[\partial_\nabla \overline \partial_\nabla Z =  \big[[Z\wedge \phi]\wedge \phi^\dagger \big]^\fG~,\]
	were the superscript $\fG$ denotes the projection on $f^*[\fG]$.
\end{proposition}
\begin{proof}

Taking the $(0,1)$-part in the first item of Lemma \ref{lem Labcomp} gives
$\overline\partial_\nabla Z = [(\zeta_{-1}+Z)\wedge \phi^\dagger]^\fG~.$ So,
\[\partial_\nabla\overline\partial_\nabla Z = [\partial_\nabla \zeta_{-1} \wedge \phi^\dagger]^\fG + [\partial_\nabla Z \wedge \phi^\dagger]^\fG~.\]
By  Lemma \ref{lem Labcomp} we have
\[\partial_\nabla\overline\partial_\nabla Z = [[Z\wedge \phi]^{\fg_{-1}}\wedge \phi^\dagger]^\fG + [[(\zeta_1+Z)\wedge \phi]^{\fG}\wedge \phi^\dagger]^\fG~.\]
Note  that $[[\zeta_1\wedge \phi]\wedge \phi^\dagger]$ is valued in $f^*[\fg_1]$. So its projection on $f^*[\fG=]$ vanishes. Moreover, $[Z\wedge\phi]$ is valued in $f^*[\fg_{-1}]\oplus f^*[\fG]$, so we obtain
$\partial_\nabla\overline\partial_\nabla Z = [[Z\wedge \phi]\wedge \phi^\dagger]^\fG~.$
\end{proof}

In the case $S=\widetilde \S$ and $\zeta$ is $\pi_1(\S)$-invariant, the Bochner technique gives the following.
\begin{proposition}
Suppose $\zeta$ is a $\pi_1(\Sigma)$-invariant Jacobi field, then $Z=0$.
\end{proposition}

\begin{proof}
Denote the Killing form on $\fg_2^\C$ by $B$ and $\langle \cdot,\cdot\rangle = B(\theta(\cdot,\cdot)$ the corresponding Hermitian metric. By a slight abuse of notation, we denote in the same way their extension to the bundle of forms with value in $f^*\T\Cy$.  On the one hand, we have
\begin{eqnarray*}
\int_\Sigma \langle Z , \partial_\nabla \overline\partial_\nabla Z\rangle & = & \int_\Sigma B(\theta(Z) , d_\nabla \overline\partial_\nabla Z ) \\
& = & - \int_\Sigma B(d_\nabla \theta(Z) , \overline \partial_\nabla Z ) \\
& = & - \int_\Sigma B(\theta(d_\nabla Z) , \overline \partial_\nabla Z) \\
& = & - \int \langle \overline \partial_\nabla Z , \overline\partial_\nabla Z\rangle \\ 
& \leq &  0 ~,
\end{eqnarray*}
where in the second line we used Stokes theorem, for the third we used the fact that $\nabla$ preserves $\theta$ and for the last one, we used the fact that in a local holomorphic coordinates $z$, the $(1,1)$-form $\langle \overline \partial_\nabla Z,\overline\partial_\nabla Z\rangle$ is a nonnegative function times $dz\wedge d\overline z$.

On the other hand, by Propositions \ref{prop:Labourie'sComputations} and \ref{prop:Labourie'sComputations2}, we have
\begin{eqnarray*}
\int_\Sigma \langle Z , \partial_\nabla \overline\partial_\nabla Z\rangle & = & \int_\Sigma B\left(\theta(Z) , \big[ [Z\wedge \phi] \wedge \phi^\dagger \big] \right)  \\
& = & \int_\S B \left( \theta(Z)\wedge \phi^\dagger , Z \wedge \phi \right) \\
& = & - \int_\S B\left(\theta(Z\wedge \phi) , Z\wedge \phi \right) \\
& = & - \int_\Sigma \langle Z\wedge \phi , Z\wedge \phi \rangle \\
& \geq & 0 ~,
\end{eqnarray*}
where we used that, locally, $\langle Z\wedge \phi , Z\wedge \phi \rangle$ is a nonnegative function time $d\overline z\wedge dz$. 

This thus gives $Z\wedge \phi=0$. Applying the involution $\lambda$ and using $\lambda(\phi)=-\theta(\phi)=\phi^\dagger$, we also obtain $Z\wedge \phi^\dagger =0$.
To show that $Z=0$, we compute the component $Z_\gamma$ along $f^*[\fg_\gamma]$ vanishes for each root $\gamma$. To do so, start by projecting the equation $[Z\wedge \phi]=0$ on $f^*[\fg_{\alpha_2}]$. Using the root system of $\fg_2^\C$, we obtain
\[[Z\wedge \phi]^{\fg_{\alpha_2}} = [Z_{\alpha_1+\alpha_2} \wedge df_{-\alpha_1}] = 0~,\]
but since $df_{-\alpha_1}$ is an isomorphism, this gives $Z_{\alpha_1+\alpha_2} =0$. Similarly, we have
\[[Z\wedge \phi]^{\fg_{\alpha_1+\alpha_2}} = [Z_{2\alpha_1+\alpha_2} \wedge df_{-\alpha_1}] = 0~,\]
so $Z_{2\alpha_1 +\alpha_2}$. The same argument gives $Z_{3\alpha_1+\alpha_2} = 0$. For the projection on negatives roots, we just use $\lambda(Z)=Z$ and the fact that that $\lambda$ maps $[\fg_\gamma]$ to $[\fg_{-\gamma}]$.
\end{proof}

We are now ready to prove Proposition \ref{prop:InfinitesimalRigidity}.

\begin{proof}[Proof of Proposition \ref{prop:InfinitesimalRigidity}]
Using the last two propositions, we get that a $\pi_1(\S)$-Jacobi field $\zeta$ on $\widetilde\S$ arising from a deformation of $\alpha_1$-cyclic surfaces has the form $\zeta= \zeta_1+ \zeta_{-1}$. In particular, Proposition \ref{prop:Labourie'sComputations2} gives
\[[\zeta_1\wedge \phi]=0~.\]
Assume $\zeta_{-\alpha_1}=0$, and project the above equation on $f^*[\fg_{-\alpha_1-\alpha_2}]$. This gives
\[[\zeta_{-\alpha_2} \wedge df_{-\alpha_1}] = 0~.\]
Again, the fact that $df_{-\alpha_1}$ is an isomorphism gives $\zeta_{-\alpha_2}=0$.

Finally, the last component of $\zeta_1$ is $\zeta_\eta= \zeta_{3\alpha_1+2\alpha_2}$. Projecting $[\zeta_1\wedge \phi]=0$ on $f^*[\fg_{3\alpha_1+\alpha_2}]$ gives 
\[[\zeta_\eta\wedge df_{-\alpha_2}] = 0~.\]
Since $df_{-\alpha_2}$ is holomorphic and not identically $0$, this implies that $\zeta_\eta=0$ in the complement of a discrete set. By continuity, $\zeta_1=0$.
Since $\zeta_{-1}=-\theta(\zeta_1)$, we conclude that $\zeta=0$. 
\end{proof}

\bibliographystyle{plain}
\bibliography{bib.bib}

\end{document}